\documentclass[twoside,11pt]{article}

\author{
Jaap Eldering\footnote{No academic affiliation (\texttt{jaap@jaapeldering.nl})}
\qquad
   Matthew Kvalheim\footnote{EECS Dept., University of Michigan, Ann Arbor, MI, USA (\texttt{kvalheim{@}umich.edu})}\qquad
  Shai Revzen\footnote{Depts.\ of EECS \& EEB, University of Michigan, Ann Arbor, MI, USA (\texttt{shrevzen{@}umich.edu})}}
\title{Global linearization and fiber bundle structure  of invariant manifolds}

\usepackage[utf8x]{inputenc}
\usepackage[T1]{fontenc}
\usepackage{lmodern}
\usepackage{amsmath}
\usepackage{amssymb}
\usepackage{amsthm}
\usepackage{stmaryrd} 
\usepackage{mathtools}
\usepackage{tikz-cd}
\usepackage{subfig}
\usepackage{lineno}
\renewcommand\linelabel[1]{}

\newcommand{\concept}[1]{\textit{#1}}

\newcommand{\N}{\mathbb{N}}

\newcommand{\R}{\mathbb{R}}

\newcommand{\Wsl}{W^s_{\text{loc}}}
\newcommand{\Ws}{W^s}

\newcommand{\Wu}{W^u}
\newcommand{\Ps}{P^s}
\newcommand{\Psl}{P^s_{\text{loc}}}

\newcommand{\Pul}{P^u_{\text{loc}}}

\newcommand{\slot}{\,\cdot\,} 

\newcommand{\T}{\mathsf{T}}
\newcommand{\D}{\mathsf{D}}

\newcommand{\Ver}{\mathsf{V}}
\newcommand{\id}{\textnormal{id}}
\newcommand{\rank}{\textnormal{rank\,}}
\newcommand{\interior}{\textnormal{int}\,}

\newcommand{\Lip}[1]{\textnormal{Lip}(#1)}

\DeclarePairedDelimiter\norm{\lVert}{\rVert}
\DeclarePairedDelimiter\minnorm{\llfloor}{\rrfloor}

\DeclareMathOperator\supp{supp}

\newtheorem{Lem}{Lemma}
\newtheorem{Th}{Theorem}
\newtheorem{Co}{Corollary}

\newtheorem{Prop}{Proposition}

\newcommand{\thistheoremname}{}
\newtheorem*{genericthm}{\thistheoremname}
\newenvironment{thmbis}[1]%
{\renewcommand{\thistheoremname}{Theorem~\ref{#1}$'$}%
  \begin{genericthm}}
{\end{genericthm}}

\theoremstyle{definition}
\newtheorem{Def}{Definition}
\newtheorem*{Def*}{Definition}

\newtheorem{Ex}{Example}[section]

\theoremstyle{remark}
\newtheorem{Rem}{Remark}

\usepackage{marvosym}
\usepackage[
    hypertexnames,%
    citecolor=blue,%
    colorlinks=true,%
    linkcolor=red%
]{hyperref}
\usepackage[all]{hypcap}

\begin{document}
\maketitle

\begin{abstract}
	We study global properties of the global (center-)stable manifold of a normally attracting invariant manifold (NAIM), the special case of a normally hyperbolic invariant manifold (NHIM) with empty unstable bundle.
	We restrict our attention to continuous-time dynamical systems, or flows.
	We show that the global stable foliation of a NAIM has the structure of a topological disk bundle, and that similar statements hold for inflowing NAIMs and for general compact NHIMs.
	Furthermore, the global stable foliation has a $C^k$ disk bundle structure if the local stable foliation is assumed $C^k$.
	We then show that the dynamics restricted to the stable manifold of a compact inflowing NAIM are globally topologically conjugate to the linearized transverse dynamics at the NAIM. 
	Moreover, we give conditions ensuring the existence of a global $C^k$ linearizing conjugacy.
	We also prove a $C^k$ global linearization result for inflowing NAIMs; we believe that even the local version of this result is new, and may be useful in applications to slow-fast systems.
	We illustrate the theory by giving applications to geometric singular perturbation theory in the case of an attracting critical manifold:
	we show that the domain of the Fenichel Normal Form can be extended to the entire global stable manifold, and under additional nonresonance assumptions we derive a smooth global linear normal form. 
\end{abstract}

\tableofcontents

\section{Introduction}\label{sec:intro}

Much of dynamical systems theory pertains to the behavior of points evolving under some smooth flow $\Phi:\R \times Q \to Q$ near an attracting invariant set.
One seeks techniques to better understand the behavior of these points.
Perhaps the most important method --- and the focus of this paper --- is the use of different coordinate systems near the attracting set, with respect to which the dynamics take a simpler form.
Particularly strong results in this direction hold in the case that the attracting invariant set is a normally attracting invariant manifold (NAIM).
This is a special case of a normally hyperbolic invariant manifold (NHIM), which is roughly defined as follows.
A manifold $M\subset Q$ is invariant if $\forall t\in \R:\, \Phi^t(M) = M$, and normal hyperbolicity means roughly that trajectories converge (or diverge) transversely to $M$ sufficiently faster than they converge (or diverge) within $M$ \cite{fenichel1971persistence,hirsch1977,normallyHypMan}.

Restricting, for now, to the case that $M$ has no boundary,
it is a well-known fundamental result \cite[ Thm~2, Thm~4.1]{fenichel1974asymptotic,hirsch1977} that a NAIM, as a special case of a NHIM, has an associated ``stable foliation'': this is a partition of the stability basin of $M$ into submanifolds $\Ws(m)$ for $m \in M$ (called ``leaves'') such that the flow $\Phi^t$ maps $\Ws(m)$ to $\Ws(\Phi^t(m))$, for any $t \geq 0$ and $m \in M$.
Furthermore, every $x$ in the stability basin of $M$ has a neighborhood $N_x$ such that $N_x$ is topologically a product of two Euclidean spaces; the first space indexes leaves and the second space locally parametrizes them (see Figure \ref{fig:foliation-vs-fiber-bundle}, left).
By using this foliation to define coordinates, one obtains a coordinate system in which the dynamics on $M$ are decoupled from the dynamics transverse to $M$.

It is also well known \cite{pugh1970linearization,hirsch1977,palis1977topological}--- and often used in the physical sciences, e.g., in the special case of the Hartman-Grobman theorem \cite{guckenheimer1983nonlinear} --- that for $M$ a NAIM (or NHIM), there exists an open neighborhood of $M$ in which the flow is topologically conjugate to its partial linearization.
For simplicity, we first describe this result in the special case that $M\subset Q$ has a neighborhood diffeomorphic to $M \times \R^{n}$ via a diffeomorphism which restricts to the identity on  $M\times \{0\}$.
In this case, we may write the flow $\Phi^t$ as $(\Phi_1^t, \Phi_2^t)$ on $M\times \R^{n}$.
Then this linearization result asserts the existence of a continuous change of coordinates $(p,v)\mapsto (q,w)$ on $M \times \R^n$ --- which restricts to the identity on $M\times \{0\}$ --- such that for any $(p,v) \in \R^n$, the trajectory $(\Phi_1^t(p,v), \Phi_2^t(p,v))$ is given by $(\Phi_1^t(q,0), \D\Phi_2^t(q,0)\cdot w)$ in the new coordinates.
In these new coordinates on $M\times \R^n$, not only are the dynamics on $M$ decoupled from the dynamics transverse to $M$, but the transverse component $w(t)\coloneqq \D\Phi_2^t(q,0)\cdot w$ of a trajectory is the solution of a nonautonomous linear differential equation.
Under additional spectral gap assumptions, this coordinate change can be taken to be continuously differentiable \cite{takens1971partiallyhyp,robinson1971differentiable,sell1983linearization,sell1983vector,sakamoto1990invariant}.
Needless to say, many key results in the sciences and engineering rely heavily on linear approximations of this form; this result shows that there exists a coordinate system in which such approximations become exact.

In this paper we prove several extensions of the familiar local results mentioned above, which we hope to be of both practical and theoretical interest.
Our results come in two flavors: (i) we show that the local topological and dynamical structure near the NAIM can be extended (often smoothly) to the entire stability basin, and (ii) we prove new local (and global) linearization results for NAIMs with nonempty boundary, subject to the requirement that the flow is ``inward'' at the boundary (inflowing NAIMs). 
The novelty of our results is that, to the best of our knowledge, all previously published work only established versions of our various results either (i) for hyperbolic attracting equilibria and periodic orbits rather than general NAIMs\footnote{It has recently come to our attention that in a soon-to-be published textbook \cite{mezic_book}, Igor Mezi\'{c} gives a very readable proof of a global linearization theorem for the case of arbitrary compact boundaryless NAIMs (see also Remark \ref{rem:mezic_remark}). In contrast, we also prove a more general result for arbitrary compact inflowing NAIMs, which may have nonempty boundary.}, (ii) for NAIMs without boundary, (iii) locally, or on proper subsets of the global stable manifold (in the case of a boundaryless NAIM, this is the stability basin), or some combination thereof.  
In contrast, our results apply to the entire global stable manifold, and they apply to the even broader class of systems consisting of inflowing NAIMs.
Thus our theorems can be used to prove results on compact domains of noncompact attracting manifolds, which can arise (for example) as intersections of a noncompact $M$ with a compact sub-level set of a function. 
Many noncompact hyperbolically attractive manifolds appear in the sciences and engineering, e.g., in the general context of slow-fast or multiple time scale systems \cite{kuehn2015multiple} studied using geometric singular perturbation theory (GSP) \cite{fenichel1979geometric,jones1995geometric,kaper1999systems}.
With the addition of a proper function having a strictly negative Lie derivative on one of its regular level sets, these give rise to compact inflowing NAIMs.
We remark that even if a compact domain of a noncompact attracting manifold is not inflowing, useful conclusions about the dynamics can sometimes still be obtained by making local modifications to the flow near the boundary of the domain in order to render it inflowing, and then applying theorems for inflowing NAIMs. 
We do precisely this in our applications to GSP in \S\ref{sec:applications}.

\subsection{Flavor of the key results}\label{sec:flavor_intro}
We begin by examining the (differential) topology of the global stable manifold, in a form depicted in Figure \ref{fig:foliation-vs-fiber-bundle} and formulated more precisely in Theorem \ref{th:fiber_bundle_theorem}.
We show that the entire global stable manifold of an inflowing NAIM has the structure of a ``disk bundle'': for $M$ of dimension $d$ in an $n$-dimensional ambient space, the global stable manifold admits a continuous ``projection'' onto $M$, and every point $m\in M$ has a neighborhood $U_m \subset M$ such that the preimage of $U_m$ through the projection is homeomorphic to the product of $U_m$ with $(n-d)$-dimensional Euclidean space $\R^{n-d}$ (``a disk'').
Furthermore, projection preimages (``fibers'') of points $m\in M$ are mapped via these homeomorphisms to sets of the form $\{m\} \times \R^{n-k}$.
We further extend this result by proving that, should the foliation near $M$ be $C^k$ smooth, then the entire global stable manifold has a structure of a $C^k$ disk bundle (for the definition, replace all homeomorphisms with $C^k$ diffeomorphisms above).
\linelabel{R1_1_a} Anticipating our global linearization results, one can think of this result as a ``weak'' or differential-topological version of global linearization of the global stable manifold: the global stable manifold always has the (differential) topological structure that one would naively expect from the (differential) topological structure of the local stable manifold.

This result has an application to geometric singular perturbation theory related to the so-called Fenichel Normal Form \cite{jones1994tracking, jones1995geometric,kaper1999systems,jones2009generalized}; for more details on the relevance of this normal form for slow-fast systems, see \S \ref{sec:applications}.
In the special case that the slow manifold is attracting, we show that our Theorem \ref{th:fiber_bundle_theorem} implies that the domain of the Fenichel normal form actually extends to the entire \emph{global} stable manifold of the slow manifold.

\linelabel{R1_1_b}We then proceed beyond ``weak'' linearization to the natural follow-up question, and show that in addition to the local topological structure, the local dynamical structure near an inflowing NAIM also extends to the entire global stable manifold: the flow on the global stable manifold is topologically conjugate to its linearization near $M$, and assuming some conditions on the relative rates of contraction of tangent vectors at $M$ evolving under the linearized flow, the global conjugacy of the flow to its linearization can be taken to be $C^k$.
In addition to this statement being a new global result, to the best of our knowledge, the local version of this linearization result is also new: linearization results previously appearing in the literature \cite{pugh1970linearization,robinson1971differentiable,palis1977topological,hirsch1977,sell1983linearization,sell1983vector,sakamoto1990invariant} have been stated for \emph{boundaryless} invariant manifolds.
This result provides a strong statement regarding how well dynamical systems can be modeled by their transverse linearizations and the dynamics on their attractor.
We give an application of this to singular perturbation theory, where the ``slow manifold'' attractors typically have boundary.
Thanks to our results for inflowing NAIMs we show that, under certain spectral conditions, singularly perturbed systems have a global normal form which is \emph{linear} in the fast variables.
This normal form is therefore stronger than the Fenichel Normal Form, which is generally (almost) fully nonlinear.

\subsection{Overview of main results}\label{sec:intro_overview}

Restated more technically, in this paper we prove some results for NHIMs which are of two types.

\begin{enumerate}
	\item Global versions of well-known local results for compact normally hyperbolic invariant manifolds (NHIMs), and compact, inflowing, normally attracting invariant manifolds (inflowing NAIMs).
	
	\item New (local and global) linearization results for inflowing NAIMs.
\end{enumerate}
\linelabel{R1_4_a}We restrict our attention to the case of flows on a finite-dimensional smooth  ambient manifold. 
We first investigate the structure of the global stable foliation of a compact normally hyperbolic invariant manifold $M\subset Q$ for a flow $\Phi^t$ on a smooth manifold $Q$.
We consider the following local-to-global result to be our first major contribution, depicted in Figure \ref{fig:foliation-vs-fiber-bundle}.
\begin{thmbis}{th:fiber_bundle_theorem}
	The global stable foliation of a NHIM is a topological disk bundle with fibers coinciding with the leaves of the foliation. 
    If additionally the local stable foliation and the NHIM are assumed $C^k$, then the global foliation is a $C^k$ disk bundle. 
    This bundle is isomorphic (as a disk bundle) to the stable vector bundle over the NHIM.
	A similar result holds for the global unstable foliation.
\end{thmbis}
In particular, if the $k$-center bunching condition (see Corollary \ref{co:fiber-bundle-NAIM-center-bunching} in \S \ref{sec:global_foliation_is_bundle}) is assumed, it follows that the global stable foliation is a $C^k$ disk bundle.
If both stable and unstable transverse directions are present at $M$, then $\Ws(M)\subset Q$ is generally only an immersed submanifold\footnote{Roughly speaking, this is because --- in the case that the unstable bundle is nonempty --- $\Ws(M)$ can accumulate on itself.
This is analogous to the ends of a curve approaching its midpoint to form a figure-eight.
The figure-eight is not an embedded submanifold, because the midpoint has no locally Euclidean neighborhood in the subspace topology, but the figure-eight is an immersed submanifold diffeomorphic to $\R$.}. Hence our result shows that the global stable manifold is a fiber bundle in its manifold topology, but \emph{not} in the subspace topology.
However if only stable transverse directions at $M$ are present, this technicality is avoided and $\Ws(M)\subset Q$ is a fiber bundle whose topology coincides with the subspace topology. 
(Embedded and immersed submanifolds are explained in more detail in \cite[Ch. 5]{lee2013smooth}.)

\begin{figure}
	\centering
	\def\svgwidth{.8\columnwidth}
	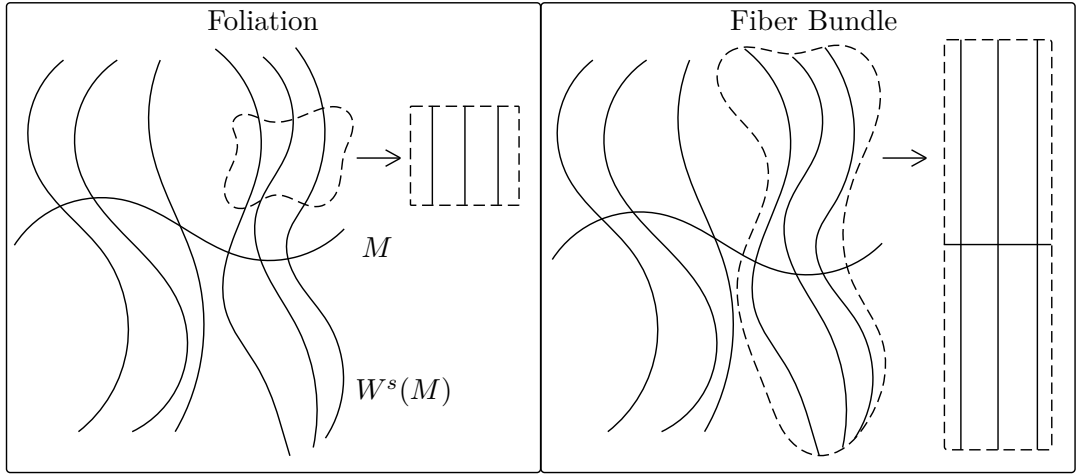
	\caption{The fact that the global stable foliation $\Ws(M)$ of a NHIM $M$ is a topological foliation implies that any point in $W^s(M)$ has a neighborhood on which the leaves of the foliation can be straightened via a homeomorphism, depicted on the left. Theorem \ref{th:fiber_bundle_theorem} shows that $\Ws(M)$ is actually a topological disk bundle, which means that $\Ws(M)$ admits \concept{local trivializations} whereby unions of entire fibers through a neighborhood of $M$ can be straightened via a homeomorphism, depicted on the right. Furthermore, if $\Wsl(M)$ is a $C^k$ foliation and $M\in C^k$, then $\Ws(M)$ is actually a $C^k$ fiber bundle, which means that these local trivializations can be chosen $C^k$. }\label{fig:foliation-vs-fiber-bundle}
\end{figure}

We also prove the corresponding fiber bundle result for the global stable foliation of a compact inflowing normally attracting invariant manifold (NAIM) $M$. 
I.e., $M$ is a NHIM with empty unstable bundle, but $M$ is allowed to have nonempty boundary, and inflowing means that $M$ is positively invariant and that the vector field points strictly inward at $\partial M$.
This is the result we actually prove, and indeed, the previously mentioned result follows from this one.

While our fiber bundle result might be expected by dynamicists, we could not find a direct proof in the literature. 
If the stable foliation happens to be smooth, then we will show that the map sending leaves to their basepoints on $M$ is a submersion with fibers diffeomorphic to disks, and it is shown in \cite[Cor.~31]{meigniez2002submersions} that this automatically implies that the stable manifold $\Ws(M)$ is a smooth disk bundle. 
On the other hand, our proof seems more elementary, directly shows that this bundle is isomorphic to $E^s$, and handles the general case in which the stable foliation is only continuous.

Next, we investigate global linearizations.
A classic result of NHIM theory is that the dynamics in a neighborhood of a NHIM are topologically conjugate to the dynamics linearized at the NHIM \cite{pugh1970linearization,hirsch1977,palis1977topological}, and there also exist conditions for $C^k$ linearization \cite{takens1971partiallyhyp,robinson1971differentiable,sakamoto1994smooth,sell1983linearization,sell1983vector,smoothInvariant}.
For the special case of a NAIM which is either an equilibrium point or a periodic orbit, \cite{lan2013linearization} showed that the linearizing conjugacy can be defined on the entire basin of attraction.
We generalize the results of \cite{lan2013linearization} in two ways: (i) we show that local linearizability implies global linearizability for arbitrary compact NAIMs\footnote{As mentioned in a previous footnote, Igor Mezi\'{c} gives a proof of this boundaryless result in his soon-to-be published textbook \cite{mezic_book} (see also Remark \ref{rem:mezic_remark}).}, and (ii) we prove a global linearizability result for inflowing NAIMs.
Since the slow manifolds for slow-fast systems typically have boundary, the latter result is necessary for our goal of deriving a linear normal form for such systems, and we consider this to be our second major theoretical contribution.
We state this result roughly below (for the precise statement, see Theorem \ref{th:smooth_inflowing_global_linearization} in \S \ref{sec:global_lin_inflowing}).
Recall that the global stable manifold is the basin of attraction in the case of a boundaryless NAIM.
For the precise definition of the global stable manifold of an inflowing NAIM, see Equation \eqref{eq:global_stable_manifold_def} in \S \ref{sec:construct_global_foliation}.

\begin{thmbis}{th:smooth_inflowing_global_linearization}
	The dynamics on the global stable manifold of an inflowing NAIM are globally topologically conjugate to the dynamics linearized at the NAIM.
	If certain additional spectral gap and regularity conditions are assumed, then additionally the dynamics are globally $C^k$ conjugate to the dynamics linearized at the NAIM.
\end{thmbis}
 
In order to prove this result, we use a geometric construction in Appendix \ref{app:wormhole} which may be of independent interest.
Generally speaking, in Appendix \ref{app:wormhole}, we show that any compact inflowing NAIM can be embedded into a compact boundaryless NAIM, in such a way that many properties of the original system are preserved, such as asymptotic rates.

After proving these results, we give two applications to geometric singular perturbation theory, under the assumption that the critical manifold is a NAIM (see the references in \S \ref{sec:applications} for examples, as well as \S \ref{sec:GSP_example}).
Our first application is to show that under this assumption, the Fenichel Normal Form appearing in the literature is valid on the entire union of \emph{global} stable manifolds $\cup_\epsilon\Ws(K_\epsilon)$ of the slow manifolds $K_\epsilon$, rather than just on the union of local stable manifolds $\cup_\epsilon \Wsl(K_\epsilon)$.
Our second application is to show that, assuming an additional ``nonresonance'' condition on the eigenvalues of points on the critical manifold and using our global linearization theorem, we derive a much stronger global normal form which is \emph{linear} in the fast variables.
We reiterate that a linearization result for inflowing NAIMs is essential here, since the slow manifolds for singular perturbation problems typically have boundary.

The remainder of the paper is organized as follows. 
In \S \ref{sec:preliminary_constructions} we give basic definitions, set notation, and give basic constructions to be used in the sequel. 
In particular, we construct the global stable foliation of a NHIM and show that the local stable foliation is a fiber bundle,
and remark that the same constructions work for inflowing NAIMs.
In \S \ref{sec:global_foliation_is_bundle} we give the proof that, if the local stable foliation and the NHIM are $C^k$, then the global stable foliation is a $C^k$ fiber bundle isomorphic (as a disk bundle) to $E^s$. 
In \S \ref{sec:global_linearization}, we show that the dynamics on the global stable manifold of an inflowing NAIM are globally conjugate to the linearized dynamics, and other related results.
In \S \ref{sec:applications}, we give applications to geometric singular perturbation theory.
In \S \ref{sec:conclusion} we conclude by summarizing what we have and have not done.
Finally, Appendix~\ref{app:linear-par-transp} contains a lemma on parallel transport, Appendix~\ref{app:wormhole} allows us to extend some results from boundaryless to inflowing NAIMs, and Appendix~\ref{app:fiber_bundles} reviews some terminology from the theory of fiber bundles, for those readers who are less familiar with it.

\subsection{Acknowledgements}

Jaap Eldering performed the major share of his contribution to this
work while holding a postdoctoral position at ICMC, University of São
Paulo, São Carlos, Brazil, supported by FAPESP grant 2015/25947-6.
Matthew Kvalheim and Shai Revzen were supported by ARO grants W911NF-14-1-0573 and W911NF-17-1-0306 to Revzen.
Kvalheim would like to thank Ralf Spatzier for many helpful conversations related to Theorem \ref{th:fiber_bundle_theorem}.
We also thank the two anonymous referees and the editors for finding a mistake in a lemma, for bringing the reference \cite{jones2009generalized} to our attention, and for their many useful suggestions.

\section{Preliminary constructions}
\label{sec:preliminary_constructions}
\subsection{Construction of the global (un)stable foliation of a NHIM}\label{sec:construct_global_foliation}

Let $Q$ be an $n$-dimensional $C^\infty$ Riemannian manifold, let $f\colon Q\to\T Q$ be a $C^{r\geq 1}$ vector field on $Q$ with $C^r$ flow $\Phi^t$ and let $M\subset Q$ be a compact $r$-normally hyperbolic invariant manifold ($r$-NHIM) for $\Phi^t$.
We recall from \cite{hirsch1977} the definition; specifically we use the most general definition of \concept{eventual relative normal hyperbolicity}.
This means that $M$ is a submanifold that is invariant under $\Phi^t$, and there exists a $\D\Phi^t$-invariant continuous splitting into a Whitney sum
\begin{align}\label{eq:NHIM_splitting}
\T Q|_M = \T M \oplus E^s \oplus E^u
\end{align}
such that $\D\Phi^t|_{E^s}$ and $\D\Phi^t|_{E^u}$ are exponentially contracting and expanding, respectively.
(See Appendix~\ref{app:fiber_bundles} for the definitions of vector bundles and Whitney sums.) 
Furthermore, any contraction or expansion of the tangential flow $\D\Phi^t|_{\T M}$ (up to power $r$) is dominated by the contraction of $\D\Phi^t|_{E^s}$, respectively the expansion of $\D\Phi^t|_{E^u}$.
More precisely, there exist $C > 0$ and $a < 0 < b$ such that for all $m \in M$, $t \ge 0$ and $0 \le i \le r$ we have
\begin{equation}\label{eq:NHIM-exp-rates}
  \minnorm{\D\Phi^t|_{E^u_m}} \ge \frac{e^{b t}}{C} \norm{\D\Phi^t|_{\T_m M}}^i
  \qquad\text{and}\qquad
  \norm{\D\Phi^t|_{E^s_m}} \le C e^{a t} \minnorm{\D\Phi^t|_{\T_m M}}^i.
\end{equation}
Here $\minnorm{A} \coloneqq    \inf \{ \norm{A v} : \norm{v}=1 \}$ denotes the minimum norm of a linear operator $A$.

Denote by $n = n_m + n_s + n_u$ the ranks of the various bundles and note that $n_m = \dim(M)$.
Since the stable and unstable cases are identical under time reversal, we restrict ourselves from now on to the stable case.
Tangent to the stable bundle $E^s$ there exists a \concept{local stable manifold} $\Wsl(M)$, a $C^r$ embedded submanifold\footnote{We will always assume without loss of generality that $\Wsl(M)$ has no boundary. Otherwise we can simply relabel its manifold interior as $\Wsl(M)$.}, with points in $\Wsl(M)$ asymptotically converging to $M$ in forward time.
$\Wsl(M)$ is invariantly fibered by embedded disks $\Wsl(m)$ comprising the leaves/fibers of the \concept{local stable foliation}: 
\begin{align}
\Wsl(M) = \coprod_{m\in M}\Wsl(m)
\end{align}
such that $\Wsl(m)$ intersects $M$ at the unique point $m$ and $\T_m \Wsl(m) = E^s_m$, see \cite[Thm~4.1]{hirsch1977}.
Each of the disks $\Wsl(m)$ is individually a $C^r$ embedded submanifold, but as a family there is generally only (H\"{o}lder) continuous dependence on the basepoint $m \in M$ \cite{hirsch1977,fenichel1974asymptotic}.
We denote by $\Psl:\Wsl(M)\to M$ the continuous projection map sending each fiber $\Wsl(m)$ to its corresponding basepoint $m\in M$.
Note that the $\Wsl(m), m \in M$ are only invariant as a foliation --- not each $\Wsl(m)$ individually --- since each $m \in M$ is generally not a fixed point of $\Phi^t$. 
This local invariance of the foliation $\Wsl(M)$ means that for all $t\geq 0$ and $m \in M$ we have\footnote{%
  When \emph{immediate} relative normal hyperbolicity is assumed (as in \cite[Thm~4.1]{hirsch1977}) then $\Wsl(M)$ is automatically forward invariant when it has constant diameter.
  In the case of \emph{eventual} relative normal hyperbolicity, standard proofs construct $\Wsl(M)$ as the local stable manifold of the map $\Phi^T$ for some fixed $T > 0$, so it might not be clear a priori that the inclusion holds for \emph{all} $t\geq 0$, though it is clear that it would hold for $t$ sufficiently large.
  However, we can always construct a \emph{new} $\Wsl(M)$ that is forward invariant for all $t \ge 0$, as a sublevel set of a Lyapunov function for $M$, see \cite{wilson1967structure,wilson1969smooth}.
}
\begin{align}\label{eq:local_inv}
\Phi^t(\Wsl(m))\subset \Wsl(\Phi^t(m)).
\end{align}
We also have a \concept{global stable manifold} $\Ws(M)\supset \Wsl(M)$  defined by\footnote{This definition works equally well for inflowing NAIMs (see \S \ref{sec:inflowing_NAIMs}), as opposed to the alternative definition $\Ws(M)\coloneqq \bigcup_{t \geq 0}\Phi^{-t}(\Wsl(M))$.}
\begin{align}\label{eq:global_stable_manifold_def}
\Ws(M)\coloneqq    \bigcup_{t \geq 0}\Phi^{-t}\left[(\Psl)^{-1}(\Phi^t(M))\right].
\end{align}

Each of the sets $\Phi^{-t}\left[(\Psl)^{-1}(\Phi^t(M))\right]$ is an embedded submanifold of $Q$ (diffeomorphic to $\Wsl(M)$), and thus $\Ws(M)$ is a $C^r$ immersed submanifold of $Q$ when given the final topology with respect to the family of inclusions $\Phi^{-t}\left[(\Psl)^{-1}(\Phi^t(M))\right]\hookrightarrow \Ws(M)$.
An atlas of charts for $\Ws(M)$ consists of the union of atlases for all of the manifolds $\Phi^{-t}\left[(\Psl)^{-1}(\Phi^t(M))\right]$ --- since the flow $\Phi^t$ is $C^r$, it can be checked that this is a $C^r$ atlas.

Let us now construct a global stable foliation as
\begin{align}
\Ws(M) = \coprod_{m\in M}\Ws(m), \qquad \Ws(m)\coloneqq   \bigcup_{t \geq 0}\Phi^{-t}(\Wsl(\Phi^t(m))).
\end{align}
Note that equation \eqref{eq:local_inv} implies that the union consists of strictly increasing sets, i.e., 
\begin{align}
\Phi^{-t}(\Wsl(\Phi^t(m))) \subset \Phi^{-t'}(\Wsl(\Phi^{t'}(m)))\quad \text{when } t \leq t'.
\end{align}
Let us prove that $\Ws(M)$ is invariant, that is, for all $t \in \R$ and $m \in M$ we have 
\begin{align}
\Phi^t(\Ws(m))=\Ws(\Phi^t(m)).
\end{align}
This follows from the following sequence of equivalent statements, with $t \in \R$ fixed:
\begin{alignat*}{2}
& &&x\in \Ws(\Phi^t(m))\\
&\exists \tau_0 \geq 0\colon  \forall \tau \geq \tau_0\colon  &&x \in \Phi^{-\tau}(\Wsl(\Phi^\tau\circ \Phi^t(m)))\\
&\exists \tau_0'\geq 0\colon  \forall \tau' \geq \tau_0'\colon\quad &&x \in \Phi^t\circ\Phi^{-\tau'}(\Wsl(\Phi^{\tau'}(m)))\\
& &&x \in \Phi^t(\Ws(m)).
\end{alignat*}
Note that each global leaf $\Ws(m)$ is a $C^r$ embedded submanifold of $\Ws(M)$.
To see this, note that given any $m \in M$ and $x \in \Ws(m)$, by definition of the global foliation there exists $t > 0$ such that $\Phi^t(x) \in \Wsl(\Phi^t(m))$.
Letting $U'$ be a neighborhood of $\Phi^t(x)$ in $\Wsl(M)$ and considering $U \coloneqq \Phi^{-t}(U') \ni x$, we see that any point $x \in \Ws(m)$ has a neighborhood $U \subset \Ws(M)$ with $U \cap \Ws(m)$ an embedded submanifold of $\Ws(M)$ (by invariance of the foliation), so it follows that $\Ws(m)$ is embedded in $\Ws(M)$.
(But since $\Ws(M)$ is generally only immersed in $Q$, any global leaf $\Ws(m)$ is generally only immersed in $Q$.)

We define the global projection $\Ps\colon \Ws(M) \to M$ to be the map that sends global fibers $\Ws(m)$ to their basepoints, just like the local projection $\Psl$.
Assume now that the local stable foliation is $C^{k\geq 0}$, by which we mean that $\Psl\in C^k$.
(Recall that $\Psl \in C^0$ automatically.)
We now show that this implies $\Ps \in C^k$.

Let $x \in \Ws(M)$ and $t \geq 0$ be such that $x \in \Phi^{-t}(\Wsl(M))$.
This implies that $x' = \Phi^t(x)\in\Wsl(M)$.
Choose a neighborhood $U_{x'}$ of $x'$ in $\Wsl(M)$.
Then $U_x\coloneqq    \Phi^{-t}(U_{x'})$ is a neighborhood of $x$ in $\Ws(M)$.
Now for any $y \in U_x$ we have that $\Phi^t(y)\in\Wsl(M)$ and by invariance of the local stable foliation it follows that
\begin{align}
\Ps(y) = (\Phi^{-t}\circ \Psl \circ \Phi^t)(y).
\end{align}
Thus it is clear that $\Ps\in C^k$ if $\Psl \in C^k$ and $k \leq r$ (i.e., $\Phi^t \in C^k$).

We conclude this section by showing that, not only is $\Ps\in C^k$ if $\Psl \in C^k$, but also that $\Ps$ is a submersion if $k \geq 1$.

\begin{Prop}\label{prop:P_submersion}
	If $\Psl$ is $C^k$ with $1 \le k \le r$, then $\Ps\colon \Ws(M) \to M$ is a $C^k$ submersion.
\end{Prop}

\begin{proof}
	We have already shown above that $\Ps \in C^k$, so it suffices to show that $\rank(\D\Psl|_{\T M}) = \dim(M)$ on all of $\Ws(M)$.
	Since $\Psl|_M = \id_M$, it follows that $\rank(\D\Psl|_{\T M}) = \dim(M)$.
	Since being full rank is an open condition, it follows that $\D\Psl$ is full rank on some relatively open neighborhood $U$ of $M$ in $\Wsl(M)$.
	
	Now let $x \in \Ws(M)$ be arbitrary.
	First, by construction of $\Ws(M)$ there exists a $T_1 > 0$ such that $\Phi^{T_1}(x) \in \Wsl(M)$.
	Next, since every point in $\Wsl(M)$ asymptotically converges to $M$, there exists $T_2 > 0$ such that $\Phi^{T_2}(\Phi^{T_1}(x))\in U$.
	Defining  $T\coloneqq T_1 + T_2 > 0$, we have $\Phi^T(x)\in U$.
	
	Since $\forall t \in \R\colon  \Ps\circ \Phi^t = \Phi^t \circ \Ps$, it follows that $\D \Phi^T_{\Ps(x)}\D \Ps_x = \D \Ps_{\Phi^T(x)}\D\Phi^T_x = \D (\Psl)_{\Phi^T(x)}\D\Phi^T_x$.
	The latter composition is formed of two surjective linear maps, and hence $\D \Phi^T_{\Ps(x)}\D \Ps_x$ is also a surjective linear map.
	The linear map $\D\Phi^T|_{\Ps(x)}$ is invertible since $\Phi^T$ is a diffeomorphism, so this implies that $\D \Ps_x\colon \T_x\Ws(M) \to \T_{\Ps(x)} M$ is surjective.
\end{proof}

\subsection{Fiber bundle structure of the local stable foliation}
\label{sec:local_stable_is_bundle}

Let $\pi\colon \T Q|_M \to M$ be the natural projection sending $v \in \T_m Q$ to $m$, and let $\widetilde{E}^s$ be any $C^r$ subbundle of $\T Q|_M$ which $C^0$ approximates $E^s$ \cite[p.72 Prop.~3.2.3]{normallyHypMan}. 
\linelabel{R2_1}(Recall that $E^s$ is generally only a continuous subbundle of $TQ|_M$.)
As shown in \cite{fenichel1974asymptotic,hirsch1977} there exists a fiber-preserving homeomorphism $\rho_0\colon U \subset \widetilde{E}^s\to \Wsl(M)$, where $U \subset \widetilde{E}^s$ is a connected neighborhood of the zero section.
Additionally, the restriction of $\rho_0$ to each fiber $\widetilde{E}^s_m$ is a  $C^r$ map.
Here we show that if additionally the local stable foliation of $\Wsl(M)$ is $C^{k\geq 1}$, then $\rho_0$ can be taken to be a $C^k$ fiber-preserving diffeomorphism.

Fiber bundle concepts from Appendix \ref{app:fiber_bundles} (in particular, Definition \ref{def:vector_bundles} and Example \ref{ex:how_to_show_its_a_bundle}) will be used in the proof of Lemma \ref{lem:rho_0_construction} below.
Here and in the rest of the paper, by a $C^k$ isomorphism of manifolds we mean a homeomorphism if $k = 0$ and a $C^k$ diffeomorphism if $k\geq 1$.
A $C^k$ fiber bundle isomorphism is a $C^k$ isomorphism of manifolds which is also fiber-preserving; see Appendix \ref{app:fiber_bundles}.

\begin{Lem}\label{lem:rho_0_construction}
	Let $M$ be a $1$-NAIM, and assume that $\Psl \in C^k$ (hence $\Wsl(M), M\subset Q$ are necessarily $C^k$ submanifolds).
	Then $\Psl\colon \Wsl(M)\to M$ is a disk bundle.
	More specifically, there exists a neighborhood $U$ of the zero section of $\widetilde{E}^s$ and a $C^{k}$ disk bundle isomorphism $\rho_0\colon U \to \Wsl(M)$ covering $\id_M$ (identifying $M$ with the zero section of $\widetilde{E}^s$).
\end{Lem}

\begin{Rem}\label{rem:forced_smoothness_lemma_remark}
	If $M$ is an $r$-NAIM for a $C^r$ vector field, then $M$ and $\Wsl(M)$ are automatically $C^r$ submanifolds of $Q$ (and hence $\Ws(M)$ is an immersed $C^r$ submanifold, as we have shown).
	See \cite[Ch. 1]{eldering2013normally} for a discussion of this.
	We will use Lemma \ref{lem:rho_0_construction} in proving Theorem \ref{th:fiber_bundle_theorem} for a $1$-NAIM which is a $C^r$ submanifold --- a slightly more general situation than an $r$-NAIM --- which explains the slightly weaker hypotheses here.
\end{Rem}

\begin{proof}
	As mentioned, for $k = 0$ the result is shown in \cite{fenichel1974asymptotic,hirsch1977} so we may assume $k \geq 1$. The latter case is implicit in the existing proofs of $C^k$ smoothness of local stable fibers, but we make it explicit here for later reference.
	\linelabel{R2_2}Consider the extended exponential map\footnote{Recall that we have endowed $Q$ with a Riemannian metric, used in the definition of spectral gap estimates \eqref{eq:NHIM-exp-rates}.
	 Here --- and throughout the rest of the paper --- we have in mind the exponential map associated to this Riemannian metric, although for the purposes of this Lemma, any metric will work equally well.
	}
	\begin{equation}\label{eq:exp-extented}
	\widehat{\exp} = (\pi,\exp)\colon \T Q|_M \to M \times Q
	\end{equation}
	that remembers the base point $m \in M$.
	This is a fiber bundle isomorphism between a neighborhood of the zero section of $\T Q|_M$ and a neighborhood of $\textup{diag}(M)$ in the trivial bundle $M \times Q$, covering the identity on $M$, where the zero section of $\T Q|_M$ and $\textup{diag}(M)$ are identified with $M$.
    Furthermore, we view $\Wsl(M) \subset Q$ as a $C^k$ submanifold of $M \times Q$ via the embedding $(\Ps,\id_Q)$, fibered by the images of the leaves $\Ws(m)$.
	It follows that $\widehat{\exp}^{-1}\bigl(\Wsl(M)\bigr)$ is a submanifold of $\T Q|_M$ fibered by the leaves $\widehat{\exp}^{-1}(\Wsl(m))$, and the leaf  $\widehat{\exp}^{-1}\bigl(\Wsl(m)\bigr)\subset \T_m Q$ is tangent to $E^s_m$ at the zero section since the derivative of $\widehat{\exp}|_{\T Q_m}$ at $0$ is the identity for any $m\in M$.
	Here we are making the usual linear identification $\T_0 \T_m Q \cong \T_m Q$.
	
	Let $\tilde \pi^s\colon \T Q|_M \to \widetilde{E}^s$ denote orthogonal projection onto $\widetilde{E}^s$.
	We have that $\tilde \pi^s \in C^{r}$, and when $\widetilde{E}^s$ is sufficiently $C^0$-close to $E^s$ then $\ker(\tilde \pi^s)$ and $E^s$ are transverse.
	Thus $\D\bigl(\tilde \pi^s \circ \widehat{\exp}^{-1}|_{\Wsl(m)}\bigr)$ is surjective for each $m \in M$, so by dimension counting this map is a linear bijection between $E^s_m\oplus \T_m M \cong \T_m \Wsl(M)$ and $\widetilde{E}^s_m\oplus \T_m M \cong \T_m \widetilde{E}^s$ for each $m \in M$.
	
	The global inverse function theorem \cite[\S~1.8 ex.~14]{guillemin1974differential} now implies that $\tilde \pi^s \circ \widehat{\exp}^{-1}|_{\Wsl(M)}$ is a $C^k$ diffeomorphism from some neighborhood of $\textup{diag}(M)$ onto a neighborhood $U$ of the zero section of $\widetilde{E}^s$.
	Thus the inverse
	\begin{equation*}
	\rho_0\colon U\to \Wsl(M)
	\end{equation*}
	is well-defined and is a fiber-preserving $C^k$ diffeomorphism onto its image. By construction it  maps the zero section to $M$ and covers the identity map.
	 \end{proof}

\subsection{Inflowing and Overflowing NAIMs}\label{sec:inflowing_NAIMs}
Suppose now that $M$ is a compact manifold but that $M$ has possibly nonempty boundary, $\partial M \neq \varnothing$.
If $\Phi^t(M) \subset M$ for all $ t \geq 0$ and the vector field $f$ points strictly inward at $\partial M$, we call $M$ an \concept{inflowing} invariant manifold.
Similarly, if $\Phi^t(M) \subset M$ for all $ t \leq 0$ and the vector field $f$ points strictly outward at $\partial M$, we call $M$ an \concept{overflowing} invariant manifold.
If $M$ is inflowing (respectively overflowing) invariant and has a splitting \eqref{eq:NHIM_splitting} satisfying exponential rates \eqref{eq:NHIM-exp-rates}, but with $E^u = \varnothing$, we call $M$ an \concept{inflowing (respectively overflowing) $r$-normally attracting invariant manifold ($r$-NAIM)}.
If $\partial M = \varnothing$ and $M$ is invariant, then $M$ is vacuously both inflowing and overflowing.
We refer to such an $M$ simply as an $r$-NAIM.
We sometimes use the term ``NAIM'' to refer to $1$-NAIMs or if we do not wish to emphasize the precise degree of hyperbolicity, and we similarly sometimes use ``NHIM''.

The main theorem about inflowing NAIMs is that, 
like boundaryless NHIMs, inflowing NAIMs also have a local stable manifold (with boundary) and a local stable foliation \cite{fenichel1974asymptotic,fenichel1971persistence}.
Note that in this case the local stable manifold has boundary, is codimension-0, and its manifold interior is an open neighborhood of the manifold interior of $M$.
Additionally, the interior of the global stable manifold is open in $Q$ and a neighborhood of the manifold interior of the NAIM.
Unlike boundaryless NHIMs, however, inflowing NAIMs do not generally persist under perturbations.

The main theorem about overflowing NAIMs is that, 
like boundaryless NHIMs, overflowing NAIMs persist under perturbations \cite{fenichel1971persistence}.
We will use this fact in \S \ref{sec:applications}.
Unlike boundaryless NHIMs, however, overflowing NAIMs do not generally possess a stable foliation.

\begin{Rem}\label{rem:inflowing_NAIM_extension}
	If $\Psl \in C^k$ for an inflowing NAIM, then the same proof as for boundaryless NHIMs shows that $\Ps\in C^k$ also.
	Furthermore, Proposition \ref{prop:P_submersion} and Lemma \ref{lem:rho_0_construction} also hold for inflowing NAIMs.
	The proof of Lemma \ref{lem:rho_0_construction} is identical.
	For the proof of Proposition \ref{prop:P_submersion}, one simply pays attention to the facts that (i) since $M$ is positively invariant, points never leave the stable foliation over $M$ when flowing forward in time, and (ii) if $\Phi^t(x) \in \Ws(m)$ for $t > 0$, $x \in \Ws(M)$, and $m \in M$, then $\Phi^{-t}(m)\in M$.
	Additionally, the same argument given in \S \ref{sec:construct_global_foliation} shows that each global fiber $\Ws(m)$ is now an embedded submanifold of $Q$, since the manifold interior of $\Ws(M)$ is open in $Q$ and thus trivially embedded.
	This argument works even for $\Ws(m)$ with $m \in \partial M$, since inflowing invariance implies that $\Phi^t(m) \in \interior M$ for $t > 0$, with $\interior M$ denoting the manifold interior of $M$. 
\end{Rem}

\section{The global stable foliation of a NHIM is a fiber bundle}
\label{sec:global_foliation_is_bundle}
As mentioned in \S \ref{sec:construct_global_foliation}, $\Psl\colon \Wsl(M)\to M$ is in general only (H\"{o}lder) continuous.
However, in many cases of interest $\Psl\colon \Wsl(M) \to M$ is $C^{k\geq 1}$ (and thus $\Ps\colon \Ws(M)\to M$ is also $C^k$ as shown in \S \ref{sec:preliminary_constructions}).
In this section we prove that if $\Psl \in C^{k\geq 0}$, then $\Ps\colon \Ws(M)\to M$ is a $C^k$ fiber bundle with fiber $\R^{n_s}$.
See Appendix \ref{app:fiber_bundles} for the relevant fiber bundle concepts. 
By reversing time the corresponding result that the global unstable manifold is a fiber bundle follows.

The topology on $\Ws(M)$ compatible with its fiber bundle structure is generally finer than the subspace topology induced from $Q$ since $\Ws(M)$ is generally only an immersed submanifold of $Q$, as discussed in \S \ref{sec:preliminary_constructions}.
Consequently, the individual fibers $\Ws(m)$ of $\Ws(M)$ are generally also only immersed submanifolds of $Q$, though they are embedded submanifolds of $\Ws(M)$ as we have seen in \S \ref{sec:preliminary_constructions}.

However if $M$ is a NAIM so that $E^u = \varnothing$, then $M$ is asymptotically stable and $\Ws(M)$ is an open neighborhood of $M$, hence trivially an embedded submanifold.
More generally, if $M$ has boundary and is an inflowing NAIM, then $\Ws(M)$ is an embedded codimension-0 submanifold with boundary\footnote{However, note that the boundary of $\Ws(M)$ is only $C^k$ if $\Psl \in C^k$, and hence generally not smooth if $\Psl \in C^0$ only.}.
Every boundaryless NHIM $M$ is a NAIM for the dynamics restricted to the invariant manifold $\Ws(M)$, and similarly $M$ is a NAIM for the time-reversed dynamics restricted to $\Wu(M)$.
Hence it suffices to prove that $\Ws(M)$ is a fiber bundle over $M$ for the case that $M$ is a NAIM.

To obtain the generality needed for our application in \S \ref{sec:applications}, we actually prove that $\Ws(M)$ is a fiber bundle for $M$ an inflowing NAIM --- since a boundaryless NAIM is vacuously inflowing, this implies the other results.

See \cite[Ch. 2]{hirsch1976differential} for the definition of the Whitney topologies, and also Remark \ref{rem:whitney} below.

\begin{Th}\label{th:fiber_bundle_theorem}
	Let $M\subset Q$ be a compact inflowing $1$-NAIM for the flow $\Phi^t$ generated by the $C^r$ vector field $f$
    on $Q$, and assume that $M\subset Q$ is a $C^r$ submanifold.
	Further assume that the local projection $\Psl\colon \Wsl(M)\to M$ is $C^k$, with $0 \leq k \leq r$.
	Then the global projection $\Ps\colon \Ws(M) \to M$ defines a $C^k$ fiber bundle structure on the global stable manifold $\Ws(M)$.
	Furthermore, $\Ws(M)$ is $C^k$ isomorphic (as a disk bundle) to any $C^k$ vector bundle over $M$ which approximates $E^s$ in the $C^0$ Whitney topology.
\end{Th}

In other words, under the hypotheses of Theorem \ref{th:fiber_bundle_theorem}, the global stable foliation $\Ws(M)$ is actually a $C^k$ disk bundle.

\begin{Rem}\label{rem:whitney}
	Since $M$ is compact, the weak and strong Whitney topologies coincide \cite[Ch. 2]{hirsch1976differential}.
	In simpler terms \cite[p.72]{normallyHypMan}, $\tilde{E}^s$ approximates $E^s$ in the $C^0$ Whitney topology if there exists a sufficiently small $\epsilon > 0$ such that for every $m\in M$, there exists a neighborhood $U_m \subset M$ of $m$ and local frames $(e_i)_{i=1}^{n_s}$, $(\tilde{e}_i)_{i=1}^{n_s}$ for $E^s, \tilde{E}^s$ such that for all $m' \in U_m$: $\|e_i(m')-e_i(m')\| < \epsilon$.  
\end{Rem}

\begin{Rem}\label{rem:forced_smoothness_theorem_remark}
	Let us reiterate Remark \ref{rem:forced_smoothness_lemma_remark}.
	If $M$ is an $r$-NHIM for a $C^r$ vector field, then $M$ and $\Wsl(M)$ are automatically $C^r$ submanifolds of $Q$.
	That is, an invariant manifold $M$ being $r$-normally hyperbolic causes ``forced $C^r$ smoothness'' of $M$ and of the local and global stable manifolds $\Wsl(M)$ and $\Ws(M)$ \cite[Ch. 1]{eldering2013normally}.
	(Of course this is $C^r$ smoothness of $\Wsl(M)$ and $\Ws(M)$ as submanifolds, not as foliations.)
	We state Theorem \ref{th:fiber_bundle_theorem} for a $1$-NAIM $M$ which is also assumed to be a $C^r$ submanifold, in order to obtain a slight amount of extra generality.
\end{Rem}

\begin{Rem}
	The hypotheses required to prove the global linearization Theorems \ref{th:global_linearization} and \ref{th:smooth_inflowing_global_linearization} are much stronger than the hypotheses required to prove the fiber bundle Theorem \ref{th:fiber_bundle_theorem}. 
	See Remark \ref{rem:compare_linearization_bundle_theorems} for more details.	
\end{Rem}

\begin{Co}\label{co:fiber-bundle-NAIM-C0-case}
	Let $M\subset Q$ be a compact inflowing $1$-NAIM for the flow $\Phi^t$ generated by the $C^1$ vector field $f$ on $Q$. 
	Then $\Ps\colon \Ws(M) \to M$ defines a $C^0$ fiber bundle structure on $\Ws(M)$, isomorphic (as a disk bundle) to $E^s$.
\end{Co}

\begin{proof}
	As mentioned earlier, $\Psl, E^s, \Wsl(M) \in C^0$ is automatically satisfied for a compact inflowing $1$-NAIM.
	Hence the result follows from Theorem \ref{th:fiber_bundle_theorem}.
\end{proof}

\begin{Co}\label{co:fiber-bundle-NAIM-center-bunching}
		Let $M\subset Q$ be a compact inflowing $1$-NAIM for the flow $\Phi^t$ generated by the $C^r$ vector field $f$ on $Q$, and assume that $M\subset Q$ is a $C^r$ submanifold.
	    Additionally, assume that there exist constants $K>0$ and $\alpha < 0$ such that for all $m \in M$, $t \ge 0$ and $0 \le i \le k < r$ the \concept{$k$-center bunching} condition holds: 
	    \begin{equation}\label{eq:fiber-bundle-cor-center-bunching}
	    \norm{\D\Phi^{t}|_{T_mM}}^i\norm{\D\Phi^t|_{E_m^s}}\leq Ke^{\alpha t}\minnorm{\D\Phi^t|_{T_mM}}.
	    \end{equation}
	    Then $\Ps\colon \Ws(M) \to M$ defines a $C^k$ fiber bundle structure on $\Ws(M)$, isomorphic (as a disk bundle) to $E^s$.
\end{Co}

\begin{proof}
	It is shown in \cite[Thm 5]{fenichel1977asymptotic} that the condition \eqref{eq:fiber-bundle-cor-center-bunching} implies that $\Psl, E^s \in C^k$.
	The result then follows from Theorem \ref{th:fiber_bundle_theorem}.
\end{proof}

\begin{Co}\label{co:fiber_bundle_with_unstable}
    Assume now that $M$ is a general compact $r$-NHIM, rather than a $C^r$ inflowing $1$-NAIM as in Theorem \ref{th:fiber_bundle_theorem}, and assume that $\Psl \in C^{k}$.
	Then $\Ps\colon \Ws(M) \to M$ defines a $C^k$ fiber bundle structure on $\Ws(M)$, when $\Ws(M)$ is endowed with the structure of an immersed submanifold as described in \S \ref{sec:construct_global_foliation}.
	This bundle is isomorphic (as a disk bundle) to any $C^k$ vector bundle over $M$ which approximates $E^s$.
	
	Similarly for the unstable manifold $\Wu(M)$, if $\Pul \in C^k$.
\end{Co}

\begin{proof}
	This follows immediately from Theorem \ref{th:fiber_bundle_theorem} and the remarks preceding it.
\end{proof}

\begin{Rem}\label{rem:too_many_corollaries}
	We leave it to the reader to formulate corollaries analogous to Corollaries \ref{co:fiber-bundle-NAIM-C0-case} and \ref{co:fiber-bundle-NAIM-center-bunching} for the case of general compact NHIMs.
\end{Rem}

We assume that $M\in C^r$ is an inflowing $1$-NAIM for the remainder of \S \ref{sec:global_foliation_is_bundle}, unless stated otherwise.

\subsection{Overview of the proof of Theorem \ref{th:fiber_bundle_theorem}}
\label{sec:overview_fiber_bundle_theorem}
(Recall that by a $C^k$ isomorphism of manifolds, we mean a homeomorphism if $k = 0$ and a $C^k$ diffeomorphism if $k\geq 1$.
A $C^k$ fiber bundle isomorphism is a $C^k$ isomorphism of manifolds which is also fiber-preserving, and a $C^k$ vector bundle isomorphism is a $C^k$ fiber bundle isomorphism which is linear on the fibers; see Appendix \ref{app:fiber_bundles}.)

By Lemma \ref{lem:rho_0_construction} and Remark \ref{rem:inflowing_NAIM_extension}, we have a $C^k$ isomorphism of fiber bundles $\rho_0\colon U\subset \widetilde{E}^s\to \Wsl(M)$, where $\widetilde{E}^s$ is a vector bundle approximating $E^s$ and $U \subset \widetilde{E}^s$ is open.
We will construct a global $C^k$ fiber-preserving isomorphism
$\rho\colon \widetilde{E}^s \to \Ws(M)$ using the local version
$\rho_0\colon U \to \Wsl(M)$, according to the following plan.
(It might be useful to first read Definition \ref{def:vector_bundles} and Example \ref{ex:how_to_show_its_a_bundle} from Appendix \ref{app:fiber_bundles}.)

First, we define a flow\footnote{ For simplicity of presentation, we henceforth ignore the fact that $\Psi^t$ and other ``flows'' that we subsequently define, such as $\Theta^t$, have possibly smaller domains of definition due to the fact that $M$ is only assumed inflowing invariant. 
The domains of these ``flows'' always contain an appropriate neighborhood of $M\times \R_{\geq 0}$, and we only flow backwards in time along one ``flow'' after flowing forward by an equal time along another appropriate ``flow''; as an example, consider equation \eqref{eq:rho_first_def}.
We will still call these objects ``flows'', and it should be clear what is meant when discussing such objects defined on a bundle over an inflowing invariant manifold.} $\Psi^t \coloneqq \rho_0^*(\Phi^t) = \rho_0^{-1}\circ \Phi^t\circ \rho_0$ on a neighborhood of $M$ contained in $U$.
Adapting a technique of \cite{pugh1970linearization}, we will find a $C^k$ Lyapunov function $V\colon\widetilde{E}^s\to [0,\infty)$ for $\Psi^t$, such that $V^{-1}(0)=M$, the sublevel set $U_c\coloneqq V^{-1}(-\infty, c)$ is positively invariant for all $0< c \leq 1$, and $V$ is strictly decreasing along trajectories starting in such level sets.
Furthermore, $V$ will be radially monotone (i.e., $V(\delta y) > V(y)$ if $\delta > 1$), and therefore it will have the nice property that any of its level sets intersect radial rays in each fiber $\widetilde{E}^s_m$ in precisely one point. 
This enables us to define a family of radial retractions $R_{c}\colon  \widetilde{E}^s\setminus M \to V^{-1}(c)$ onto level sets of $V$, and we will show that this family is $C^k$. 

We next construct a $C^k$ flow $\Theta^t$ on $\widetilde{E}^s \setminus M$ that preserves level sets of $V$ and covers $\Phi^t|_M$.
$\Theta^t(y)$ is defined to be $R_{V(y)}\circ \Pi^t(y)$, where the radial retraction family $R_c$ is as defined above, and $\Pi^t$ is the smooth linear parallel transport covering $\Phi^t|_M$, constructed in  Appendix \ref{app:linear-par-transp}.

The global $C^k$ isomorphism $\rho\colon  \widetilde{E}^s\to \Ws(M)$ is now constructed as follows.
First, to a point $x \in \widetilde{E}^s$ we assign a time $t(x)$ roughly proportional to the value $V(x) > 0$, but with $t \equiv 0$ on a neighborhood of $M$.
Next, we use the family of retractions $R_{c}$ to construct a (nonlinear) rescaling diffeomorphism that maps $\widetilde{E}^s$ diffeomorphically onto $U_1\coloneqq V^{-1}(-\infty, 1)$, with the image of $x$ denoted $\xi(x)\in U_1$.
Finally, we define $\rho$ by first flowing $\xi(x)$ forward by $\Theta_{t(x)}$, applying $\rho_0$, and then flowing backward in time by applying $\Phi^{-t(x)}$:
\begin{equation}\label{eq:rho_first_def}
\rho(x) = \Phi^{-t(x)} \circ \rho_0 \circ \Theta^{t(x)} \circ \xi(x).
\end{equation}
See Figure \ref{fig:fiber-bundle-proof}.
The map $\rho$ is $C^k$ and fiber-preserving by construction.
Properness of $\rho$ will follow from asymptotic stability of $M$, and this will in turn imply surjectivity of $\rho$.
The map $\rho$ will be injective on $V$ level sets since $x\mapsto t(x)$ will be constant on $V$ level sets.
Since $V$ is strictly decreasing along trajectories of $\Psi^t$ contained in $U_1$, it will follow that $\rho$ takes disjoint level sets of $V$ to disjoint subsets of $\Ws(M)$, so that $\rho$ will be injective.
Therefore $\rho$ is a homeomorphism since it is a continuous and closed bijection, so this will complete the proof if $k = 0$ --- if $k \geq 1$, a computation in the proof of Theorem \ref{th:fiber_bundle_theorem} in \S \ref{sec:proof_fiber_bundle_theorem} will show that $\D\rho$ is an isomorphism everywhere, completing the proof of Theorem \ref{th:fiber_bundle_theorem}.

The purpose of \S \ref{sec:preliminary_results} is to construct the technical devices $V$, $R_{c}$, and $\Theta^t$ that will be used in the proof of Theorem \ref{th:fiber_bundle_theorem}.
The idea behind the proof of Theorem \ref{th:fiber_bundle_theorem} is simple, but our constructions are careful in order to avoid the loss of degrees of differentiability of $\rho$.  

\begin{figure}
	\centering
	\def\svgwidth{1\columnwidth}
	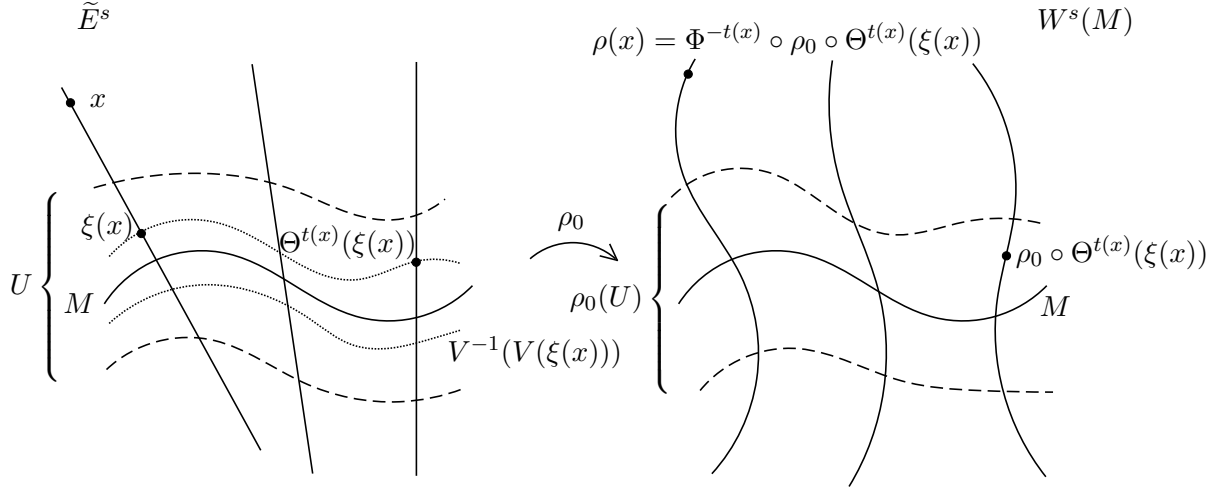
	\caption{An illustration of the proof of Theorem \ref{th:fiber_bundle_theorem}. The neighborhood $U\subset \widetilde{E}^s$ and its image $\rho_0(U)$ are bounded by dashed curves. The level set $V^{-1}(V(\xi(x))) \subset \widetilde{E}^s$ is depicted by the dotted curves.}\label{fig:fiber-bundle-proof}
\end{figure}

\subsection{Preliminary Results}
\label{sec:preliminary_results}

In order to carry out the proof of Theorem \ref{th:fiber_bundle_theorem}, we need some tools.
We will use the following result adapted from \cite{pugh1970linearization}; for the definition of a fiber metric, see Definition \ref{def:fiber_metric} in Appendix \ref{app:fiber_bundles}.

\begin{Prop}\label{prop:radially_monotone_lyap_funcs}
	Suppose that $M$ is an inflowing invariant manifold for the $C^r$ flow $\Phi^t$ on $Q$, and let $\pi\colon E\to M$ be a $C^r$ subbundle of $TQ|_M$ equipped with any fiber metric.
	Let $A^t$ be a continuous linear flow on $E$ such that the time one map is a uniform contraction on fibers:
	\begin{equation*}
	  \exists \alpha < 1\colon \forall m \in M\colon \norm{A^1|_{E_m}} \leq \alpha.
	\end{equation*}
	Let $\Psi^t$ be a $C^{k\geq 0}$ local flow with $\Psi$ defined at least on a set of the form $[0,1]\times E(\epsilon)$ for some $\epsilon > 0$, where $E(\epsilon)\coloneqq \{y \in E: \norm{y} \leq \epsilon\}$.
	Suppose that $\Psi^t$ also covers $\Phi^t$, leaves the zero section of $E$ invariant, and is Lipschitz close to $A^t$ for small $t$, by which we mean:
	\begin{equation}\label{eq:lipschitz_close}
	\forall 0 \leq t \leq 1, m \in M\colon   \Lip{(\Psi^t-A^t)|_{E_m}} \leq \mu <  \min{\left(\frac{1}{3}\kappa, 1-\alpha\right)},
	\end{equation} 
	for $\kappa\coloneqq    \inf\{\minnorm{A^t|_{E_m}}: m \in M, 0\leq t \leq 1\}$.
	Then there exists a continuous, nonnegative, and proper function $V\colon  E\to \R$ such that $V^{-1}(0)= M$ and\footnote{To limit excessive parentheses, here and henceforth we abuse notation by writing, e.g., $V^{-1}(a,b)$ instead of $V^{-1}((a,b))$, etc.}:
	\begin{enumerate}
		\item
		\label{item:lyap1} $V$ is radially monotone on $E$.
		For any $c > 0$, $V^{-1}(c)$ intersects each radial ray in exactly one point $y\in E$. 
		By a radial ray we mean any set of the form $\{\lambda x:\lambda>0\}$, where $x\in E$ is nonzero.
		
		\item
		\label{item:lyap2}
		For any $c$ with $0 < c \leq 1$, the sublevel set $V^{-1}(-\infty,c]$ is contained in $E(\epsilon)$ and is positively invariant under $\Psi^t$.

		\item 
		\label{item:lyap3}
		$V$ is radially bi-Lipschitz: there are constants $0 < b_1 < b_2 < 0$ such that for any $y \in E\setminus M$ and $\delta \neq 1$, we have the estimate
		\begin{equation}\label{eq:V_slope_bounds}
		 0< b_1\leq \frac{|V(\delta y)-V(y)|}{\norm{\delta y-y}} \leq b_2.
		\end{equation}
		
         \item
         \label{item:lyap4}
         If $\Psi \in C^{k\geq 1}$, then $V$ is $C^k$ on $E\setminus M$, and the derivative of $V$ along any trajectory of $\Psi$ starting in $V^{-1}(0,1)$ is strictly negative.
	\end{enumerate}
\end{Prop}

\begin{Rem}\label{rem:Rademacher_explanation}
  In the proof of Proposition \ref{prop:radially_monotone_lyap_funcs} below, we make use of Rademacher's theorem \cite[Thm 3.1.6]{federer1968geometric}.
  This is to provide a unified proof for both the $C^{k \ge 1}$ and $C^0$ cases.
  In the $C^{k \ge 1}$ case we differentiate $V$ in the radial direction in order to obtain the inequalities \eqref{eq:V_slope_bounds}.
  This is not possible in the $C^0$ case, but condition \eqref{eq:lipschitz_close} implies that the function $V$ is locally Lipschitz, hence Rademacher's theorem implies that $V$ is differentiable almost everywhere in the measure-theoretic sense. This is sufficient for our purposes.

  Note that $\Psi^t$ is actually radially differentiable in the context of Theorem \ref{th:fiber_bundle_theorem}, even when $k = 0$.
  However, by using Rademacher's Theorem we simplify the statement of Proposition \ref{prop:radially_monotone_lyap_funcs}, while weakening its hypotheses.
\end{Rem}

\begin{proof}
	Define the function $g\colon E(\epsilon)\to \R$ by
	\begin{equation*}
	g(y)\coloneqq    \int_{0}^{1} \norm{\Psi^t(y)}dt.
	\end{equation*}
	In the proof of \cite[Thm~4.1]{pugh1970linearization} it is shown that $g$ is continuous, radially monotone, and that for each $0 < \mu' \leq \mu$, $g^{-1}(\mu' \epsilon)$ intersects each radial ray in exactly one point $y \in E(\epsilon)$.
	It also follows from the proof that corresponding sublevel sets $g^{-1}(-\infty,\mu' \epsilon]$ are positively invariant.
	It follows from the last inequality in the proof of \cite[Lem.~4.2]{pugh1970linearization} that for any $y \in E(\epsilon)$ and $\delta > 0$:
	\begin{equation*}
	0< (\kappa - 3 \mu) \leq \frac{|g(\delta y)-g(y)|}{\norm{\delta y-y}} \leq (\alpha + \mu),
	\end{equation*}
	where $\delta > 0$ is small enough that this expression is defined.
	Now let us assume that $k \ge 1$ and $\Psi\in C^k$ --- it is clear that $g$ is $C^k$ on the complement of the zero section.
	We compute
	\begin{align*}
	\frac{\partial}{\partial t}g\circ \Psi^t(y)=\frac{\partial}{\partial t} \int_{0}^{1}\norm{\Psi^{t+s}(y)}ds = \frac{\partial}{\partial t}\int_{t}^{t+1}\norm{\Psi^s(y)}ds = \norm{\Psi^{1}\Psi^t(y)}-\norm{\Psi^t(y)} < 0,
	\end{align*}
	with the last term being negative since $A^1$ is an $\alpha$-contraction and $\Psi^1$ is $\mu$-Lipschitz close to $A^1$, with $\mu + \alpha < 1$, hence $\Psi^1$ decreases the norm of $\Psi^t(y)$.
	
	Replacing $g$ by $g/(\mu\epsilon)$, we may assume that $g$ satisfies \ref{item:lyap1}, \ref{item:lyap2}, and \ref{item:lyap4}, and also
	\begin{equation}\label{eq:g_bounds}
	0< b_1 \leq \frac{|g(\delta y)-g(y)|}{\norm{\delta y-y}} \leq \beta
	\end{equation}
	 if we define $b_1 \coloneqq (\kappa-3\mu)/(\mu \epsilon)$ and $\beta\coloneqq (\alpha + \mu)/(\mu \epsilon)$.
	
	Now let $0 < \epsilon' < \epsilon$ be such that $g^{-1}(-\infty,1] \subset  E(\epsilon')\subset E(\epsilon)$.
	We are going to extend $g$ to a $C^k$ function $V:E\to \R$ such that $V|_{E(\epsilon')}=g|_{E(\epsilon')}$, with $V$ satisfying \ref{item:lyap1}, \ref{item:lyap2}, \ref{item:lyap3}, and \ref{item:lyap4}.
	Let $\chi\colon [0,\infty) \to [0,\infty)$ be a $C^\infty$ nonnegative, increasing function satisfying $\chi \equiv 0$ on $[0,\epsilon']$ and $\chi \equiv 1$ on $[\epsilon,\infty)$, and define $\psi\colon E\to \R$ via $\psi(y)\coloneqq \chi(\norm{y})$.
	We now define $V\colon E\to \R$ via
	\begin{equation}\label{eq:V_def}
	V\coloneqq (1-\psi)g +  \psi \beta \norm{\slot} ,
	\end{equation}
	with the understanding that $V(y) = \beta \norm{y}$ for $\norm{y} >\epsilon$.
	Clearly $V$ is continuous.
	By the definition of $\psi$, we see that $V$ is $C^k$ on $E\setminus M$ and  $V|_{E(\epsilon')}=g|_{E(\epsilon')}$, so that when we replace $\epsilon$ by $\epsilon'$ then \ref{item:lyap2} and \ref{item:lyap4} are automatically satisfied.
	Clearly \ref{item:lyap3} implies \ref{item:lyap1}, so it suffices to show that $V$ satisfies \ref{item:lyap3}.
	To do this, fix any $y \in E\setminus M$.
	By \eqref{eq:g_bounds}, the function $\delta \mapsto g(\delta y)$ is locally Lipschitz, and the same is true of the other functions in Equation \eqref{eq:V_def} defining $V$.
	Since $V$ is a sum of products of such functions, $V$ is also locally Lipschitz.
	Hence even if $k=0$, by Rademacher's theorem $\delta \mapsto V(\delta y)$ and $\delta \mapsto g(\delta y)$ are differentiable except at a set of Lebesgue measure zero.
	The following statements must be interpreted to hold almost everywhere in the Lebesgue measure sense.
	We obtain
	\begin{equation*}
	 \frac{\partial}{\partial \delta}V(\delta y) = [(1-\psi)\frac{\partial}{\partial \delta}g(\delta y) + \psi \beta\norm{y}] + (\beta\norm{\delta y}-g)\frac{\partial}{\partial \delta}\psi(\delta y),
	\end{equation*}
	where here and henceforth $g$ and $\psi$ are implicitly evaluated at $\delta y$.
	From this, we obtain the inequalities
	\begin{align*}
	[(1-\psi)\frac{\partial}{\partial \delta}g(\delta y) + \psi \beta\norm{y}] \leq \frac{\partial}{\partial \delta}V(\delta y) \leq [(1-\psi)\frac{\partial}{\partial \delta}g(\delta y) + \psi \beta\norm{y}] + \beta \norm{\delta y}	\frac{\partial}{\partial \delta}\psi(\delta y).
	\end{align*}
	The leftmost inequality was obtained using $\frac{\partial}{\partial \delta}\psi(\delta y) \geq 0$ and the fact that  Equation~\eqref{eq:g_bounds} implies that $\beta \norm{\delta y}\geq g(\delta y)$, and the rightmost inequality was obtained since $g(\delta y), \frac{\partial}{\partial \delta}\psi(\delta y) \geq 0$.
	Now Equation~\eqref{eq:g_bounds} implies that $b_1\norm{y}\leq \frac{\partial}{\partial \delta}g(\delta y) \leq \beta \norm{y}$ for $\delta y \in \supp (1-\psi)$, and $\beta \geq b_1$, so it follows that $b_1\norm{y} \leq [(1-\psi)\frac{\partial}{\partial \delta}g(\delta y) + \psi \beta\norm{y}] \leq \beta\norm{y}$.
	Consequently, we have
	\begin{equation}\label{eq:V_deriv_bounds}
	b_1\norm{y} \leq \frac{\partial}{\partial \delta}V(\delta y) \leq \beta\norm{y} + \beta \norm{\delta y}\frac{\partial}{\partial \delta}\psi(\delta y)
	= \beta [ 1 + \delta \frac{\partial}{\partial \delta}\psi(\delta y) ] \norm{y}.
	\end{equation}
  The derivative term can be rewritten into a radial derivative
  \begin{equation*}
    \delta \left.\frac{\partial}{\partial \rho}\psi(\rho y)\right|_{\rho=\delta}
    = \left.\frac{\partial}{\partial r}\psi(r \delta y)\right|_{r=1}
    \eqqcolon \psi'(\delta y),
  \end{equation*}
  which is zero for $\delta y \not\in E(\epsilon)$, and bounded inside the precompact set $E(\epsilon)$.
	Defining $b_2\coloneqq \beta [ 1 + \sup_{x \in E(\epsilon)} \psi'(x) ] < \infty$, we see that the right hand side of \eqref{eq:V_deriv_bounds} is bounded by $b_2\norm{y}$.
	
	The function $\delta \mapsto V(\delta y)$ is absolutely continuous since it is locally Lipschitz, so the fundamental theorem of Lebesgue integral calculus\footnote{This technicality is needed only for the case that the differentiability degree $k = 0$. If $k \geq 1$, the mean value theorem or the elementary fundamental theorem of calculus will suffice.} implies that for any $\delta > 1$,
	\begin{equation*}
	V(\delta y) -V(y) = \int_{1}^\delta \frac{\partial}{\partial s}V(sy)\, ds \geq b_1\norm{y} (\delta - 1) = b_1\norm{\delta y - y},
	\end{equation*}
	and a similar argument shows that $V(\delta y) - V(y)\leq b_2\norm{\delta y  - y}$, with $b_2$ defined as before. 
	This completes the proof.
\end{proof}

By Proposition \ref{prop:radially_monotone_lyap_funcs}, for each $c > 0$ we may define a retraction $R_c\colon  E\setminus M \to V^{-1}(c)$ by sliding along radial rays.
We then define $R\colon  (E\setminus M) \times (0,\infty)\to E$ by $R(\slot,c)\coloneqq    R_c$.

\begin{Lem}\label{lem:radial_retraction}
Let all notation be as in Proposition \ref{prop:radially_monotone_lyap_funcs} and let $R\colon  (E\setminus M) \times (0,\infty)\to E$ be as defined above.
Then $R \in C^k$.
\end{Lem}

\begin{proof}
	If $k \geq 1$, then $V\in C^{k \geq 1}$ and Equation \eqref{eq:V_slope_bounds} together with the mean value theorem imply that the derivative of $V$ in the radial direction is nonzero.
	We may therefore apply the implicit function theorem to the function $F(\delta,x,c)\coloneqq   V(\delta x) - c$, defined on $(0,\infty)\times (E\setminus M)\times (0,\infty)$, to obtain a $C^{k}$ $\R$-valued function $\delta(x,c)$ such that $V(\delta(x,c)x) = c$.
	It follows that $R(x,c) = \delta(x,c)x$, and therefore $R \in C^{k}$.

	If $k = 0$, we will make use of a different argument which is effectively a ``Lipschitz implicit function theorem''.
	The argument is sketched as follows.
	We will define an auxiliary $C^0$ map $T\coloneqq (0,\infty)\times (E\setminus M)\times (0,\infty)\to \R$ such that $T_{x,c}\coloneqq T(\slot,x,c)$ has a unique fixed point given by $\delta(x,c)$, and additionally such that $T_{x,c}$ is a contraction mapping. 
	The domain of each $T_{x,c}$ is not a complete metric space, but the existence of the fixed point of each $T_{x,c}$ will follow from Proposition \ref{prop:radially_monotone_lyap_funcs} point \ref{item:lyap1}, and these fixed points $R_c(x)$ are unique since $T_{x,c}$ is a contraction.
	 The theorem then follows from the general fact that the fixed points $R_c(x)$ of a continuous family $T_{x,c}$ of contractions depends continuously on the parameters $(x,c)$.  
	
	We now proceed with the proof.
	Define a continuous function $T$ by
	\begin{equation*}
	T_{x,c}(\delta)\equiv T(\delta,x,c)\coloneqq   \delta - \frac{1}{b_2}\frac{V(\delta x)-c}{\norm{x}}
	\end{equation*}
	on $(0,\infty)\times (E\setminus M)\times (0,\infty)$, where $b_2$ is as in Proposition \ref{prop:radially_monotone_lyap_funcs}.
	We already know from Proposition \ref{prop:radially_monotone_lyap_funcs} that for each $x$ and $c$, $T_{x,c}$ has a unique fixed point $\delta(x,c)$.
	$T_{x,c}$ is a contraction uniformly in $x$ and $c$ since 
	\begin{align*}
	T_{x,c}(\delta_1)-T_{x,c}(\delta_2) &= \delta_1 - \delta_2 -\frac{1}{b_2 }\frac{V(\delta_1 x) - V(\delta_2 x)}{\norm{x}}
	\\
	&= \left(1-\frac{1}{b_2}\frac{V(\delta_1 x)-V(\delta_2 x)}{(\delta_1-\delta_2)\norm{x}}\right)(\delta_1-\delta_2),
	\end{align*} 
	so that by Equation \eqref{eq:V_slope_bounds} we have
	\begin{align*}
	|T_{x,c}(\delta_1)-T_{x,c}(\delta_2)|\leq k|\delta_1 - \delta_2|,
	\end{align*}
	where $k\coloneqq   \left(1-\frac{b_1}{b_2}\right)<1$.
	It follows that the fixed point $\delta(x,c)$ depends continuously on $(x,c)$, since it is a general fact that the fixed points of a (uniform) family of contractions $T_{x,c}$ depend continuously on the parameters $(x,c)$.
	Since $V(\delta(x,c)x) = c$, it follows that $R(x,c) = \delta(x,c)x$, and therefore $R \in C^0$.
	This completes the proof.
\end{proof}

\begin{Lem}[Nonlinear parallel transport]\label{lem:nonlin_par_transp}
    Let all notation be as in Proposition \ref{prop:radially_monotone_lyap_funcs}.
    Then there exists a $C^k$ flow $\Theta^t$ on $E$ such that $\Theta$ covers the base flow and preserves level sets of $V$:
	\begin{equation*}
	\forall t\colon V \circ \Theta^t = V.
	\end{equation*}
\end{Lem}

\begin{proof}
	Let $\Pi^t$ be any $C^r$ linear parallel transport covering $\Phi^t$ as in Lemma \ref{lem:smooth_parallel_transport} (see Appendix \ref{app:linear-par-transp}).
	We define $\Theta^t$ for $t > 0$ by flowing $x$ forward via the linear flow $\Pi^t$ and then projecting onto the $V(x)$ level set of $V$:
	\begin{equation*}
	\Theta^t(x)\coloneqq   R_{V(x)}\circ \Pi^t(x). 
	\end{equation*}
	It follows from Lemma \ref{lem:radial_retraction} that $\Theta \in C^k$.
	Since for each $t$ the linear flow $\Pi^t$ maps radial rays into radial rays,
	it follows that $\Theta^t$ is injective for fixed $t\geq 0$ and also that $\Theta$ indeed satisfies the group property. 
    
    \linelabel{R2_5}By Lemma \ref{lem:smooth_parallel_transport} in Appendix \ref{app:linear-par-transp}, $\Pi^t$ is a $C^r$ linear isomorphism for each $t > 0$, and $R$ is $C^k$ by Lemma \ref{lem:radial_retraction}.
    Using the fact that $\Pi^t$ preserves radial rays, it follows  that the map $x \mapsto R_{V(x)} \circ \Pi^{-t}(x)$ defined on $\Pi^t(E)$ is a $C^k$ inverse for $\Theta^t$, so $\Theta^t$ is a $C^k$ isomorphism onto its image.
    
\end{proof}

\subsection{The proof of Theorem \ref{th:fiber_bundle_theorem}}\label{sec:proof_fiber_bundle_theorem}

Now we start the proof that $\Ws(M)$ is a fiber bundle isomorphic to $\widetilde{E}^s$ over the inflowing NAIM $M$.
(For the reader new to fiber bundles, see Appendix \ref{app:fiber_bundles} and in particular Example \ref{ex:how_to_show_its_a_bundle}).

\begin{proof}[Proof of Theorem \ref{th:fiber_bundle_theorem}]
	
Let $\rho_0\colon  U\subset \widetilde{E}^s\to \Wsl(M)$ be the $C^k$ fiber-preserving isomorphism constructed using Lemma \ref{lem:rho_0_construction} and Remark \ref{rem:inflowing_NAIM_extension}.
With Proposition \ref{prop:radially_monotone_lyap_funcs} in mind, we define a $C^k$ local flow $\Psi^t$ on $U$ and a global $C^{\max\{k-1,0\}}$ linear flow $A^t$ on $\widetilde{E}^s$ as follows:

\begin{equation}
\Psi^t \coloneqq \rho_0^*(\Phi^t) = \rho_0^{-1}\circ \Phi^t\circ \rho_0, \qquad
\forall m \in M\colon   A^t|_{\widetilde{E}^s_m}\coloneqq    [\D (\rho_0|_{\widetilde{E}^s_{\Phi^t(m)}})]^{-1} \circ \D \Phi^t \circ \D(\rho_0|_{\widetilde{E}^s_m}),
\end{equation}
for all $t > 0$.
Here we are viewing $\D(\rho_0|_{\widetilde{E}^s_m})$ as a map $\widetilde{E}^s_m\to E^s_m$ via the canonical linear identification $\T_0\widetilde{E}^s_m \cong \widetilde{E}^s_m$.
Note that by compactness of $M$, the linear flow $A^t$ is eventually uniformly contracting relative to the fiber metric (see Def. \ref{def:fiber_metric} in Appendix \ref{app:fiber_bundles}) on $\widetilde{E}^s$ induced by the Riemannian metric on $TQ$: i.e., there exists $t_0 > 0$ and $0 \leq \alpha < 1$ such that
\begin{equation*}
	\forall m \in M\colon \norm{A^{t_0}|_{E_m}} \leq \alpha.
\end{equation*}

Furthermore, even if $k = 0$, the restrictions $\rho_0|_{\widetilde{E}^s_m}$ of $\rho_0$ to individual linear fibers of $\widetilde{E}^s$ are smooth (see, e.g., \cite[Thm~1, Thm~4.1]{fenichel1974asymptotic,hirsch1977}). 
It follows that $\Psi^t|_{\widetilde{E}^s_m}$ is smooth, and
\begin{equation}
\forall m \in M\colon   \D (\Psi^t|_{\widetilde{E}^s_m})_0 = A^t|_{\widetilde{E}^s_m}.
\end{equation}
This is because the restrictions $\rho_0|_{\widetilde{E}^s_m}$ of $\rho_0$ to individual linear fibers of $\widetilde{E}^s$ are smooth. 
By NHIM theory (see~\cite[p.~191, (2.4)]{pugh1970linearization}), we also have that the map $(m,y) \mapsto \D (\rho_0|_{\widetilde{E}^s_m})_y$ is uniformly continuous at the zero section in the sense that $\D (\rho_0|_{\widetilde{E}^s_{m'}})_y$ tends uniformly to $\D (\rho_0|_{\widetilde{E}^s_m})_0$ as $m' \to m$ and $\norm{y} \to 0$.

It therefore follows in either case ($k >0$ or $k = 0$), possibly after a rescaling of time, that $\Psi^t$ and $A^t$ satisfy the hypotheses of Proposition \ref{prop:radially_monotone_lyap_funcs} on some uniform neighborhood of the zero section --- the Lipschitz condition hypothesis in Proposition \ref{prop:radially_monotone_lyap_funcs} follows from the preceding sentence, see also~\cite[p.~191, (2.4)(b')]{pugh1970linearization}.
Hence we obtain a radially monotone function $V\colon\widetilde{E}^s \to \R$ as in Proposition \ref{prop:radially_monotone_lyap_funcs}, and the corresponding $C^k$ family of radial retractions $R_c\colon  \widetilde{E}^s\setminus M\to V^{-1}(c)$ with $0 < c < \infty$ as in Lemma \ref{lem:radial_retraction}.
As in Lemma \ref{lem:nonlin_par_transp}, we also obtain a $C^k$ flow $\Theta^t$ defined on $E\setminus M$, covering $\Phi^t$, and preserving level sets of $V$. 
For the sake of notation, for any $0< c <\infty$ we henceforth let $U_c$ denote the sublevel set $V^{-1}(-\infty,c)$.

We next define the following smooth functions. Let
$\chi \in C^\infty\bigl([0,\infty);[0,1)\bigr)$ be a global
diffeomorphism such that $\chi(\delta) = \delta$ for $\delta \le \frac{1}{2}$ and
$\chi'(\delta) \in (0,1)$ for $\delta > \frac{1}{2}$. Secondly, define
$\tau(\delta) \coloneqq    \delta - \chi(\delta)$. Hence we have
$\tau \in C^\infty\bigl([0,\infty);[0,\infty)\bigr)$ with $\tau(\delta) = 0$ for $\delta \le \frac{1}{2}$ and
$\tau'(\delta) > 0$ for $\delta > \frac{1}{2}$. Thus $\tau$ restricted to
$(\frac{1}{2},\infty)$ is a diffeomorphism onto $(0,\infty)$. 

Finally, we construct the global fiber bundle isomorphism
$\rho\colon \widetilde{E}^s \to \Ws(M)$ as follows. 
Let $x \in \widetilde{E}^s_m$ at the base point $m\in M$.
Define the
rescaled $\xi(x) \coloneqq    R_{\chi(V(x))}(x)$ if $\norm{x} \neq 0$ and $\xi(x)=x$ otherwise.
Note that
$\xi(x) \in U_1 \subset \widetilde{E}^s$ is $C^k$ dependent on $x$ by Lemma \ref{lem:radial_retraction}, and $\xi$ is a $C^k$ isomorphism since its $C^k$ inverse is given by $y\mapsto R_{\chi^{-1}(V(y))}(y)$.
Secondly, define $t(x) \coloneqq    \tau(V(x))$ if $\norm{x}\neq 0$ and $t(x) = 0$ otherwise, and note that $t$ is $C^k$ dependent on $x$ since $\tau(\delta) = 0$ for $\delta \le \frac{1}{2}$. 
Now define
\begin{equation}\label{eq:global-fiber-iso}
  \rho\colon \widetilde{E}^s \to \Ws(M), \qquad
  \rho(x) = \Phi^{-t(x)} \circ \rho_0 \circ \Theta^{t(x)} \circ \xi(x).
\end{equation}
By construction it is clear that $\rho$ is a $C^k$ fiber-preserving map
covering the identity on $M$.

We now show that $\rho$ is injective.
First, note that $\rho$ restricted to any level set of $V$ is injective since the function $t \mapsto t(x)$ is constant on such level sets by construction, and for any fixed $t_0\geq 0$ the map $x \mapsto \Phi^{-t_0} \circ \rho_0 \circ \Theta^{t_0} \circ \xi(x)$ is a $C^k$ isomorphism. 
Hence it suffices to show that $\rho(V^{-1}(a))\cap\rho(V^{-1}(b)) = \varnothing$ for any $a \neq b$, $a,b > 0$.
Let $t_1 = t(V^{-1}(a))$, $t_2 = t(V^{-1}(b))$, and assume without loss of generality that $b > a$ and hence $t_2 > t_1$.
The following are equivalent statements:
\begin{alignat*}{2}
&&
\rho(V^{-1}(a)) &\cap \rho(V^{-1}(b)) = \varnothing \\
&\Longleftrightarrow\qquad &
\Phi^{-t_1}    \circ \rho_0 \circ \Theta^{t_1}\circ \xi\bigl(V^{-1}(a)\bigr) &\cap
\Phi^{-t_2}    \circ \rho_0 \circ \Theta^{t_2}\circ \xi \bigl(V^{-1}(b) \bigr) = \varnothing \\
&\Longleftrightarrow\qquad &
\Phi^{t_2-t_1} \circ \rho_0 \circ \Theta^{t_1}\circ \xi\bigl(V^{-1}(a)\bigr) &\cap
\rho_0 \circ \Theta^{t_2}\circ \xi\bigl(V^{-1}(b)\bigr) = \varnothing \\
&\Longleftrightarrow\qquad &
\Psi^{t_2-t_1}      \circ \Theta^{t_1}\circ \xi\bigl(V^{-1}(a)\bigr) &\cap
\Theta^{t_2}\circ \xi\bigl(V^{-1}(b)\bigr) = \varnothing \\
&\Longleftrightarrow\qquad &
\Psi^{t_2-t_1}\circ \xi \bigl(V^{-1}(a)\bigr) &\cap \xi(V^{-1}(b)) = \varnothing\\
&\Longleftrightarrow\qquad &
\Psi^{t_2-t_1}\bigl(V^{-1}(\chi(a))\bigr) &\cap V^{-1}(\chi(b))= \varnothing,
\end{alignat*}
where we used in the last line that by construction of $\xi$, $\xi(V^{-1}(a))=V^{-1}(\chi(a))$ and $\xi(V^{-1}(b))=V^{-1}(\chi(b))$.
Since $a < b$ we have $0 < \chi(a) < \chi(b)< 1$, and since for any $t\geq 0$ we have that $V^{-1}(-\infty,a]$ is $\Psi^{t}$-invariant by Proposition \ref{prop:radially_monotone_lyap_funcs}, it follows that indeed $\Psi^{t_2-t_1}\bigl(V^{-1}(\chi(a))\bigr) \cap V^{-1}(\chi(b))= \varnothing$.
Hence $\rho$ is injective.

We continue with surjectivity of $\rho$.
Letting $(y_n)_{n\in \N}$ be any sequence in $\widetilde{E}^s$ with $\norm{y_n}\to \infty$, it follows that $t(y_n)\to \infty$ and $\chi(y_n)\to 1$.
For any $\delta > 0$, let $U_\delta$ denote the $\delta$-sublevel set of $V$, consistent with our notation $U_1$.
Let $K\subset \Ws(M)$ be any compact set.
By compactness and asymptotic stability of $M$, there exists $t_0 > 0$ such that $\forall t \geq t_0\colon \Phi^t(K)\subset U_{1/2}$.
It follows that for all sufficiently large $n\in \N$: 
\begin{align*}
\rho(y_n)\in \Ws(M) \cap \bigcup_{t\geq t_0} \Phi^{-t}(\rho_0(U_1\setminus U_{1/2})) \subset \Ws(M)\setminus K. 
\end{align*}
Hence $\rho$ takes diverging sequences to diverging sequences and is therefore a proper map, so $\rho$ is also a closed map.
We have already shown that the continuous map $\rho$ is injective.
Using these facts, we establish surjectivity of $\rho$ as follows.
Since $\rho$ maps the manifold interior $\interior \widetilde{E}^s$ of $\widetilde{E}^s$ into the manifold interior $\interior \Ws(M)$ of $\Ws(M)$, it follows by invariance of domain that $\rho|_{\interior \widetilde{E}^s}\colon \interior \widetilde{E}^s \to \interior \Ws(M)$ is an open map, and since we also know that $\rho$ is a closed map, it follows by connectivity that $\interior \Ws(M) = \rho(\interior \widetilde{E}^s)$.
Next, since $\rho(\partial \widetilde{E}^s) \subset \partial \Ws(M)$ and since $\partial \widetilde{E}^s$ and $\partial \Ws(M)$ are topological manifolds, we may invoke invariance of domain again and similarly conclude that $\rho(\partial \widetilde{E}^s) = \partial \Ws(M)$.
This completes the proof of surjectivity of $\rho$.

To summarize, we have shown that $\rho$ is a bijective, continuous, and closed map.
Therefore, $\rho$ is a homeomorphism.
This completes the proof if $k = 0$.

Assuming now that  $k \geq 1$, it suffices to show that $\rho$ is a local diffeomorphism.
Since $\rho$ agrees with the diffeomorphism $\rho_0$ on $U_{1/2}$, it suffices to consider $x\in \widetilde{E}^s \setminus M$.
Let $y \coloneqq \rho_0\circ \Theta^{t(x)}\circ \xi(x)$, $\xi'\coloneqq    \frac{\partial}{\partial \delta}\xi(\delta x)|_{\delta = 1}$, and $\kappa \coloneqq    \frac{\partial}{\partial\delta}t(\delta x)|_{\delta = 1} $.
A computation using $\frac{\partial}{\partial t}\Phi^t = \D\Phi^t\circ f$ shows that
\begin{equation}
\frac{\partial}{\partial \delta}\rho(\delta x)|_{\delta = 1} = \D \Phi^{-t(x)}\left[-\kappa f(y) + \D\rho_0\left(\kappa g(\xi(x))+\D\Theta^{t(x)}\xi'\right)\right],
\end{equation}
where $g\colon  \widetilde{E}^s\to \T \Tilde{E}^s$ is the vector field generating $\Theta$ --- note that $g$ is tangent to $V$ level sets.
The vector in brackets points outward to $\Phi^{t(x)}\circ\rho(V^{-1}(a))$, where $a = V(x)$.
To see this, first note that Proposition \ref{prop:radially_monotone_lyap_funcs} and the inequality $\tau' > 0$ imply that $\kappa > 0$, and therefore $-\kappa f(y)$ points outward to $\Phi^{t(x)}\circ\rho(V^{-1}(a))$.
Similar reasoning also shows that  $\xi'$  is outward pointing at $V$ level sets.
Since $\Theta^t$ is a flow, it follows that $\D\Theta^{t(x)}$ can be smoothly deformed to the identity through isomorphisms (in other words, an ``isotopy''), which implies that $\D\Theta^{t(x)} \xi'$ is also outward pointing at $V$ level sets.
Since $\rho_0$ is a diffeomorphism which maps the zero section of $\tilde{E}^s$ to $M$, $\D\rho_0$ maps outward pointing vectors at $V$ level sets to outward pointing vectors at $V\circ \rho_0^{-1}$ level sets. 
Taken together, these facts show that the quantity in brackets indeed points outward to $\Phi^{t(x)}\circ\rho(V^{-1}(a))$.
Now since $\Phi^t$ is a flow, $\D \Phi^{-t(x)}$ can also be smoothly deformed to the identity through isomorphisms, and therefore the same reasoning above in the case of $\D\Theta^{t(x)}$ establishes that $\frac{\partial}{\partial \delta}\rho(\delta x)|_{\delta = 1}$ is outward pointing to $\rho(V^{-1}(a))$ at $\rho(x)$.
On the other hand, $\D \rho$ takes a basis for $\T_x V^{-1}(a)$ to a basis for $\T_{\rho(x)}\rho(V^{-1}(a))$, so $\D \rho$ is an isomorphism. 
This completes the proof.
\end{proof}

\section{Global linearization} \label{sec:global_linearization}
A classic result in the theory of normally hyperbolic invariant manifolds is that the dynamics are ``linearizable'' on some (a priori small) neighborhood of the NHIM \cite{pugh1970linearization,hirsch1977,palis1977topological}, which is to say that there is some neighborhood $U$ of $M\subset Q$ and a fiber-preserving homeomorphism $\varphi\colon U \to \varphi(U)\subset E^s\oplus E^u$ onto a neighborhood of the zero section such that
\begin{align}\label{eq:linearizing_conjugacy}
\varphi \circ \Phi^t|_U  = \D \Phi^t|_{E^s\oplus E^u} \circ \varphi,
\end{align}
for all $t \in \R$ such that both sides of the expression are defined.
This is a vast generalization of the Hartman--Grobman Theorem.

In \cite{lan2013linearization}, this local result is extended to a global result for the special cases of exponentially stable equilibria and periodic orbits.
More precisely, it is shown that the domain of the linearization can actually be taken to be the entire basin of attraction for these attractors. 
As conjectured in the conclusion of \cite{lan2013linearization}, this globalization result should generalize to hold for arbitrary (boundaryless) NAIMs.
In this section, we establish this generalization; our methods are similar to theirs.

We would like to apply this linearization result in the context of slow-fast systems to derive linear normal forms on a neighborhood of a slow manifold, improving upon the Fenichel Normal Form to be discussed in \S \ref{sec:applications}.
However, the relevant slow manifolds are often compact manifolds with boundary.
To the best of our knowledge, neither the results mentioned above nor the existing (local) linearization results in the literature directly apply in this case \cite{pugh1970linearization,robinson1971differentiable,hirsch1977,palis1977topological,sell1983linearization,sell1983vector,sakamoto1994smooth,smoothInvariant}.
Thus, the content of this section can be divided as follows.

\begin{enumerate}
	\item In \S \ref{sec:global_lin_boundaryless}, we prove that the dynamics restricted to the basin of attraction of a compact \emph{boundaryless} NAIM are globally linearizable (and smoothly linearizable, assuming some additional hypotheses). 
	This is the content of Theorem \ref{th:global_linearization} and its corollaries. 
	\item In \S \ref{sec:global_lin_inflowing}, we turn to the main goal of \S \ref{sec:global_linearization}, which is to prove that the dynamics restricted to the global stable manifold of a compact \emph{inflowing} NAIM are globally linearizable (and smoothly linearizable, assuming some additional hypotheses).  
	This is the content of Theorem \ref{th:smooth_inflowing_global_linearization} and its corollaries.
	We prove this local result in the course of proving the stronger global result.
	To achieve this, we use a topological construction developed in Appendix \ref{app:wormhole}, which might be of independent interest.
\end{enumerate}

\begin{Rem}\label{rem:mezic_remark}
	After we had proved Theorem \ref{th:global_linearization} (global linearization for the boundaryless case), we learned that Igor Mezi\'{c} independently obtained this theorem before us.
	A very readable proof appears in his soon-to-be published textbook on Koopman operator theory \cite{mezic_book}.
	His proof technique is the same as ours.
	However, his result applies only to boundaryless NAIMs, and therefore we need our Theorem \ref{th:smooth_inflowing_global_linearization} (global linearization for inflowing NAIMs) for our goal of deriving a linear normal form for a class of slow-fast systems, which we do in \S \ref{sec:applications}.
\end{Rem}

\subsection{Global linearization for boundaryless NAIMs}\label{sec:global_lin_boundaryless}

In the following results, recall that by a $C^k$ isomorphism, we mean a homeomorphism if $k = 0$ and a $C^k$ diffeomorphism if $k \geq 1$.
Theorem \ref{th:global_linearization} will be used as a stepping stone to prove a global linearization result for inflowing NAIMs in \S \ref{sec:global_lin_inflowing} below, which we apply to slow-fast systems in \S \ref{sec:applications}.

\begin{Th}\label{th:global_linearization}	
	Let $M\subset Q$ be a compact (boundaryless) $1$-NAIM for the $C^r$ flow $\Phi^t$ on $Q$.
	Assume that $E^s \in C^k$, with $0 \leq k \leq r-1$, and that $\Phi^t$ is locally $C^k$ conjugate to the linear flow $\D\Phi^t|_{E^s}$ on some neighborhood of $M \subset \Ws(M)$.
	Then $\Phi^t$ is globally $C^k$ conjugate to $\D\Phi^t|_{E^s}$, which is to say that there exists a $C^k$ fiber-preserving isomorphism $\varphi\colon  \Ws(M)\to E^s$ such that
	\begin{equation}\label{eq:global_conjugacy}
	\forall t \in \R\colon    \varphi \circ \Phi^t = \D\Phi^t|_{E^s}\circ \varphi.
	\end{equation}
	Additionally, $\varphi$ agrees with the local conjugacy on its domain.
\end{Th}

\begin{Rem}
	The hypotheses of Theorem \ref{th:global_linearization} \emph{assume} the existence of a local linearizing $C^k$ conjugacy.
	Theorem \ref{th:global_linearization} shows that any local linearizing conjugacy may be extended to a global linearizing conjugacy having the same regularity. 
\end{Rem}

\begin{Rem}\label{rem:compare_linearization_bundle_theorems}
	The relationship between Theorems \ref{th:fiber_bundle_theorem},  \ref{th:global_linearization}, and \ref{th:smooth_inflowing_global_linearization} (see \S \ref{sec:global_lin_inflowing} below) are as follows. 
	The hypotheses of Theorem \ref{th:global_linearization} and \ref{th:smooth_inflowing_global_linearization} are much stronger than those required for Theorem \ref{th:fiber_bundle_theorem}, and in particular the hypotheses of Theorem \ref{th:fiber_bundle_theorem} are not sufficient to prove the conclusion of Theorems \ref{th:global_linearization} and \ref{th:smooth_inflowing_global_linearization}.
	However, the hypotheses of Theorem \ref{th:global_linearization} and \ref{th:smooth_inflowing_global_linearization} suffice to prove the conclusion of Theorem \ref{th:fiber_bundle_theorem} in the cases that $\partial M = \varnothing $ and $M$ is inflowing with $\partial M \neq \varnothing$, respectively, since the conjugacy $\varphi\colon W^s(M)\to E^s$ is in particular a $C^k$ fiber-preserving isomorphism.
\end{Rem}

\begin{Rem}
	As pointed out in \cite{lan2013linearization}, the flow is automatically locally $C^1$ linearizable near an exponentially stable equilibrium or periodic orbit.
	See the references therein.
	It is shown in \cite{pugh1970linearization,hirsch1977,palis1977topological} that the flow is always locally $C^0$ linearizable near a NHIM.
	There are also various results in the literature giving conditions ensuring that $\Phi^t$ is locally $C^k$ linearizable near a general invariant manifold, such as  \cite{sakamoto1994smooth,takens1971partiallyhyp,robinson1971differentiable,sell1983linearization,sell1983vector}.
	See also \cite[Chap.~VI]{smoothInvariant} for similar results, as well as historical remarks.
	In particular, we obtain the following easy corollary.
\end{Rem}

\begin{Co}\label{co:easy_boundaryless_lin}
Let $M$, $Q$, $\Ps$, and $\Phi^t$ be as in Theorem \ref{th:global_linearization}.
Assume that $\Phi^t$ is a $C^1$ flow and that $M$ is a $1$-NAIM.
Then $\Phi^t|_{\Ws(M)}$ is globally topologically conjugate to $\D\Phi^t|_{E^s}$.
\end{Co}
\begin{proof}
	As remarked already, it is shown in \cite{pugh1970linearization,hirsch1977,palis1977topological} that $\Phi^t$ is locally topologically conjugate to $\D\Phi^t|_{E^s}$ near $M$.
	Thus Theorem \ref{th:global_linearization} yields the existence of a global topological conjugacy between $\Phi^t$ and $\D\Phi^t|_{E^s}$.
\end{proof}

\begin{Rem}
	This furnishes a proof alternative to the one given in \S \ref{sec:proof_fiber_bundle_theorem} that for a $1$-NAIM $M$, $\Ps\colon   \Ws(M)\to M$ is always a topological fiber bundle isomorphic to $E^s$ over $M$.
\end{Rem}

\begin{proof}[Proof of Theorem \ref{th:global_linearization}]
	By assumption, there exists a neighborhood $U$ of $M$ and a $C^k$ fiber-preserving isomorphism $\varphi_\text{loc} \colon  U\to \varphi_\text{loc}(U)\subset E^s$  such that for all $t > 0$:
	\begin{equation}\label{eq:semiconjugacy}
	\varphi_\text{loc}\circ \Phi^t|_U = \D\Phi^t|_{E^s}\circ \varphi_\text{loc}.
	\end{equation}
	We now extend this local conjugacy to a global one.
	
	\linelabel{R2_6}Let $V$ be a strict $C^\infty$ Lyapunov function for the flow $\Phi^t$ \cite{wilson1967structure,wilson1969smooth}, and let $f$ be the vector field generating that flow.
    $V$ is nonnegative, $V^{-1}(0)=M$, $V$ is proper, and the Lie derivative $L_f V$ of $V$ along trajectories not contained in $M$ is strictly negative. 
	Since $V$ is proper, there is $c> 0$ such that $V^{-1}(c) \subset U$, for example, take $c < \inf_{x \in \Ws(M)\setminus U}V(x)$.

	Since $L_f V < 0$ on $V^{-1}(c)$, it follows that the vector field $f$ intersects $V^{-1}(c)$ transversally.
	The properties of $V$ imply that for all $x \in \Ws(M)\setminus M$ there exists a unique ``impact time'' $\tau(x)\in \R$ such that $\Phi^{\tau(x)}(x)\in V^{-1}(c)$.
	Using transversality of $f$ to $V^{-1}(c)$ and the implicit function theorem applied to $(x,\tau)\mapsto V(\Phi^\tau(x))$, we see that  $\tau\colon \Ws(M)\setminus M\to \R$ is $C^{r}$.
	
	Now define a map $\varphi\colon \Ws(M)\to E^s$ by 
	\begin{align}\label{eq:global_varphi_def}
	\varphi(x) \coloneqq    \left\{
	\begin{array}{lr}
	\D \Phi^{-\tau(x)}\circ \varphi_\text{loc} \circ \Phi^{\tau(x)}(x), &  x \in \Ws(M)\setminus M\\
	\varphi_\text{loc}, &  x \in U.
	\end{array}
	\right.
	\end{align}
	See Figure \ref{fig:global-linearization-proof}.
	Note that $\varphi$ is well-defined because Equation \eqref{eq:semiconjugacy} implies that the two functions in \eqref{eq:global_varphi_def} agree on $U \setminus M$, and hence $\varphi\in C^k$ since clearly both maps in \eqref{eq:global_varphi_def} are.
	Note also that $\varphi$ maps fibers $\Ws(m)$ into fibers $E^s_m$ by invariance of the stable foliation and stable vector bundle under the nonlinear and linear flows, respectively. 
	It is easy to check directly that $\varphi$ conjugates the flows as in Equation \eqref{eq:global_conjugacy} --- we now show that $\varphi$ is a $C^k$ isomorphism.
	
	We first show that $\varphi\colon \Ws(M) \to E^s$ is injective.
	Define the $C^k$ function $V'\colon \varphi_\text{loc}(U)\to \R$ by $V'\coloneqq    V \circ (\varphi_\text{loc})^{-1}$.	
	We have that $\forall v = \varphi_\text{loc}(x) \in \varphi_\text{loc}(U)$ and all $t>0$: 
	\begin{align}\label{eq:lyap_func_Es}
	V'\circ \D \Phi^t(v) = V' \circ \D \Phi^t\circ \varphi_\text{loc}(x) = V' \circ  \varphi_\text{loc} \circ \Phi^t(x) = V\circ \Phi^t(x).
	\end{align}
	It follows that $V'$ is strictly decreasing along trajectory segments of $\D\Phi^t|_{E^s}$ contained in $\varphi_\text{loc}(U)$, so that any trajectory of $\D\Phi^t|_{E^s}$ starting in $\varphi_\text{loc}(U\setminus M)$ intersects the $c$ level set $\Sigma \coloneqq (V')^{-1}(c)$ of $V'$ in precisely one point.
	Now suppose that $\varphi(x) = \varphi(y)$ with $x \neq y$.
	Then we have $\D\Phi^{-\tau(x)}(v) = \D \Phi^{-\tau(y)}(w)$, where $v \coloneqq \varphi_\text{loc}\circ \Phi^{\tau(x)}(x)\in \Sigma$ and $w \coloneqq \varphi_\text{loc}\circ \Phi^{\tau(y)}(y)\in \Sigma$. 
	It follows that $v = \D\Phi^{\tau(x)-\tau(y)}(w)$ which, by the previous comments, implies that $\tau(x)=\tau(y)$ and that $v = w$.
    By injectivity of $\varphi_\text{loc}$ we therefore have $\Phi^{\tau(x)}(x)=\Phi^{\tau(x)}(y)$.
    Since $\Phi^{\tau(x)}$ is injective, $x = y$.
    Hence $\varphi$ is injective.
		
	We next show that $\varphi\colon \Ws(M) \to E^s$ is surjective.
	Note that $\varphi_\text{loc}$ maps $M$ one-to-one and onto the zero section of $E^s$ (this must be the case since homeomorphisms preserve $\omega$-limit sets). 
	Now consider any $v \in E^s\setminus M$, identifying $M$ with the zero section as usual.
	Since $M$ is a NAIM, the zero section of $E^s$ is asymptotically stable for the linear flow $\D\Phi^t|_{E^s}$.
	This fact and continuity imply that there is $t_0 > 0$ such that $\D\Phi^{t_0}(v) \in \Sigma$.
	Let $x' \in U$ be the unique point with $\varphi_\text{loc}(x')= \D\Phi^{t_0}(v)$.
	Setting $x = \Phi^{-t_0}(x')$, we see that $\tau(x) = t_0$ and that \begin{equation*}
	\varphi(x) = \D\Phi^{-\tau(x)}\circ \varphi_\text{loc} \circ \Phi^{\tau(x)}(x) = \D\Phi^{-t_0}\circ \varphi_\text{loc} \circ \Phi^{t_0}(x) = \D\Phi^{-t_0}\circ \varphi_\text{loc} (x') = \D\Phi^{-t_0}\circ \D\Phi^{t_0}(v)= v.
	\end{equation*} 
	
	To complete the proof, it suffices to prove that $\varphi^{-1}\in C^k$.
	Since $\varphi^{-1}|_{\varphi(U)} = (\varphi_\text{loc})^{-1}$, $\varphi^{-1}$ is $C^k$ on $U$.
	Now let $v \in E^s\setminus \varphi(U)$.
	By asymptotic stability of the zero section for $\D\Phi^t|_{E^s}$, there exists $t_0 > 0$ such that $\D\Phi^{t_0}(v)\in U$.
	Since $\varphi^{-1} = \Phi^{-t_0}\circ \varphi^{-1}\circ \D \Phi^{t_0}|_{E^s}$ by \eqref{eq:global_conjugacy}, it follows that 
	\begin{equation*}
	\varphi^{-1}|_{\D\Phi^{-t_0}(\varphi(U))} = \Phi^{-t_0}\circ (\varphi_\text{loc})^{-1}\circ \D \Phi^{t_0}|_{\D\Phi^{-t_0}(\varphi(U))}
	\end{equation*} 
	is a composition of $C^k$ maps, so that $\varphi^{-1}$ is $C^k$ on a neighborhood of $v$. 
	This completes the proof.
\end{proof}

\begin{figure}
	\centering
	\def\svgwidth{1\columnwidth}
	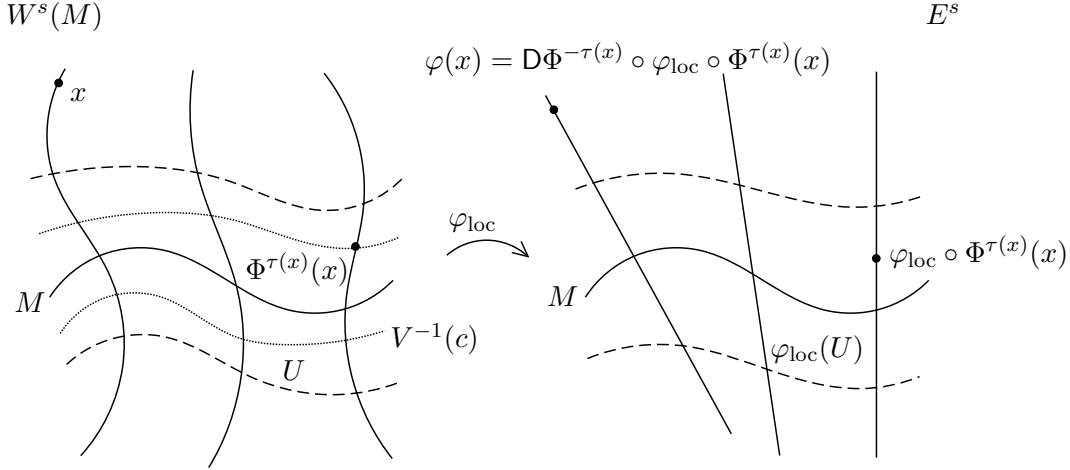
	\caption{An illustration of the proof of Theorem \ref{th:global_linearization}. The neighborhood $U\subset \Ws(M)$ and its image $\varphi_{\text{loc}}(U)$ are bounded by dashed curves. The level set $V^{-1}(c) \subset \Ws(M)$ is depicted by the dotted curves.}\label{fig:global-linearization-proof}
\end{figure}

\subsection{Global linearization for inflowing NAIMs}\label{sec:global_lin_inflowing}

We next proceed to our main goal for \S \ref{sec:global_linearization}, which is to prove a global linearization theorem for inflowing NAIMs.
The key tool we use is Proposition~\ref{prop:wormhole} in Appendix~\ref{app:wormhole}, which shows that many results about boundaryless NAIMs can be transferred to inflowing NAIMs. 
We reiterate that this result is necessary for our derivation of a linear normal form (see Theorem \ref{th:GSP_sakamoto_linearization}) for slow-fast systems, since the slow manifolds for these systems are typically manifolds with boundary.

\begin{Th}\label{th:smooth_inflowing_global_linearization}
	Let $M\subset Q$ be a compact inflowing $r$-NAIM for the flow $\Phi^t$ generated by the $C^r$ vector field $f$ on $Q$, where $r \geq 3$.
	Assume further that there exist constants $0 < \delta < -\alpha < -\beta$ and $K \geq 1$ such that $-\alpha > r \delta$, $-\beta < -2\alpha - (r-1)\delta$, and such that for all $t \geq 0$
	\begin{equation}
	\begin{aligned}\label{eq:sakamoto-rates-linearization-section}
	K^{-1}e^{-\delta t} \leq \minnorm{\D \Phi^t|_{TM}} &\leq \norm{\D \Phi^t|_{TM}} \leq K e^{\delta t}, \\
	K^{-1}e^{-\delta t} \leq \minnorm{(\D \Phi^t|_{TM})^{-1}} &\leq \norm{(\D \Phi^t|_{TM})^{-1}} \leq K e^{\delta t}, \\
	K^{-1}e^{\beta t} \leq \minnorm{\D \Phi^t|_{E^s}} &\leq \norm{\D \Phi^t|_{E^s}} \leq K e^{\alpha t}
	\end{aligned}
	\end{equation}
	hold uniformly on $TM$ and $E^s$.
	Then $E^s\in C^{r-1}$ and $\Phi^t|_{\Ws(M)}$ is globally $C^{r-1}$ conjugate to $\D\Phi^t|_{E^s}$.
\end{Th}

\begin{proof}
	By Proposition~\ref{prop:wormhole} in Appendix~\ref{app:wormhole}, there exists a $C^\infty$ manifold $\widehat{Q}$, an open neighborhood $U \supset \Ws(M)$, a $C^\infty$ embedding $\iota\colon U\to \widehat{Q}$, a $C^r$ vector field $\hat{f}$ on $\widehat{Q}$ generating a $C^r$ flow $\widehat{\Phi}^t$, and a $C^r$ compact and boundaryless $r$-NAIM $\widehat{N}\subset \widehat{Q}$ for $\widehat{\Phi}^t$ with the following properties.
	
	\begin{enumerate}
		\item \label{item:vector_field_restrict_smooth_inflowing_linearize} $\iota_*(f|_{\Ws(M)}) = \hat{f}|_{\iota(\Ws(M)}$.
		\item \label{item:stable_fibers_restrict_smooth_inflowing_linearize} $\forall m \in M: \iota(\Ws(m)) = \widehat{W}^s(\iota(m))$, where $W^s(M)$ and $\widehat{W}^s(\widehat{N})$ are the global stable foliations of $M$ for $f$ and $\widehat{N}$ for $\hat{f}$, respectively.
		\item \label{item:sakamoto_rates_persist_smooth_inflowing_linearize}There exist constants $\delta',\alpha', \beta'$ arbitrarily close to $\delta,\alpha,\beta$ such that \eqref{eq:sakamoto-rates-linearization-section}  holds uniformly on $T\widehat{N}$ and $\widehat{E}^s$, after replacing $\delta$, $\alpha$, $\beta$, $M$, $E^s$, and $\Phi^t$ by $\delta'$, $\alpha'$, $\beta'$, $\widehat{N}$, $\widehat{E}^s$, and $\widehat{\Phi}^t$, respectively. 
		Here, $\widehat{E}^s$ is the stable vector bundle of $\widehat{N}$ for $\widehat{\Phi}^t$.	
		\end{enumerate}
		In \cite[p. 335, Thm B]{sakamoto1994smooth} it is shown\footnote{%
		Comparing~\eqref{eq:sakamoto-rates-linearization-section} to \cite[Eq.~(2.7)]{sakamoto1994smooth}, note that we express the conditions on $\D\Phi^t|_{TM}$ only for $t \geq 0$, but also on the inverse to prevent issues with $\Phi^t(m)$ leaving $M$. We do not have an equivalent for Sakamoto's estimate for $Z$ since we have no unstable bundle. Furthermore, the lack of minimum norms for the lower bounds seems to be a minor oversight in Sakamoto's conditions.
		Next, Sakamoto's result is actually only stated for a NAIM with trivial normal bundle, but this is easily extended to the general case by locally writing the dynamics on the normal bundle and then extending them to the total space of the direct sum with an \concept{inverse bundle}; this useful trick is briefly mentioned in a different context on \cite[pp.~333-334]{sakamoto1994smooth}, but see also \cite[Sec.~3]{pugh1997holder} for more details. 
		Finally, it is actually claimed in \cite[p. 333]{sakamoto1994smooth} that $\widehat{E}^s \in C^{r}$, but to the best of our knowledge this seems to be a minor oversight --- the references in \cite[p. 333]{sakamoto1994smooth} provided to support this statement either claim $C^{r-1}$ smoothness only \cite[Lem. 3.2 (ii)]{sakamoto1990invariant}, or omit the details of higher degrees of smoothness in their proof \cite[Thm 3.1]{yi1993generalized}, \cite[Thm 3.1]{yi1993stability}. 
		For a proof that $\widehat{E}^s \in C^{r-1}$, see \cite[Thm 7]{fenichel1971persistence} or \cite[App. B]{sakamoto1990invariant}. 
		} that item \ref{item:sakamoto_rates_persist_smooth_inflowing_linearize} implies that $\widehat{W}^s(\widehat{N}),\widehat{E}^s\in C^{r-1}$ and that $\widehat{\Phi}^t$ is locally $C^{r-1}$ conjugate to $\D \widehat{\Phi}^t|_{\widehat{E}^s}$ near $\widehat{N}$.
		By Theorem \ref{th:global_linearization}, it follows that $\widehat{\Phi}^t|_{\widehat{W}^s(\widehat{N})}$ is globally $C^{r-1}$ conjugate to $\D\widehat{\Phi}^t|_{\widehat{E}^s}$.
		Let $\widehat{\varphi}\colon \widehat{W}^s(\widehat{N})\to \widehat{E}^s$ be such a conjugacy.
		Items \ref{item:vector_field_restrict_smooth_inflowing_linearize} and \ref{item:stable_fibers_restrict_smooth_inflowing_linearize} imply that restriction of $\widehat{\varphi}$ yields a well-defined global $C^{r-1}$ conjugacy $\widehat{\varphi}|_{\widehat{W}^s(\iota(M))}\colon \widehat{W}^s(\iota(M)) \to  \widehat{E}^s|_{\iota(M)}$ between $\widehat{\Phi}^t|_{\widehat{W}^s(\iota(M))}$ and $\D\widehat{\Phi}^t|_{\widehat{E}^s|_{\iota(M)}}$.
		Hence $\varphi\coloneqq \iota|_{\widehat{W}^s(\iota(M))}^{-1}\circ \widehat{\varphi} \circ \iota$ is a global $C^{r-1}$ conjugacy from $\Phi^t|_{\Ws(M)}$ to $\D\Phi^t|_{E^s}$.   
\end{proof}

\begin{Co}\label{co:top_inflowing_global_linearization}	
	Let $M\subset Q$ be a compact inflowing $1$-NAIM for the $C^r$ flow $\Phi^t$ generated by the $C^r$ vector field $f$ on $Q$.
	Then $\Phi^t|_{\Ws(M)}$ is globally topologically conjugate to $\D\Phi^t|_{E^s}$.
\end{Co}

\begin{proof}
The proof is identical to that of Theorem \ref{th:smooth_inflowing_global_linearization}, but with \cite[Thm 2]{pugh1970linearization} used instead of \cite[p. 335, Thm B] {sakamoto1994smooth} to provide a local linearizing $C^0$ conjugacy.	
\end{proof}

\section{Applications to Geometric Singular Perturbation Theory}\label{sec:applications}

We give two applications of Theorem \ref{th:fiber_bundle_theorem} and Theorem \ref{th:smooth_inflowing_global_linearization} to slow-fast systems in the context of geometric singular perturbation theory (GSP).
Our applications assume the special case in which the slow manifold is attracting.
Both applications are improvements of the so-called Fenichel Normal Form, discussed below, and are contained in Theorems \ref{th:GSP_global_fen_norm_form} and \ref{th:GSP_sakamoto_linearization} below.

The Fenichel Normal Form \cite{jones1994tracking, jones1995geometric,kaper1999systems,jones2009generalized} is the form that the equations of motion take near the slow manifold of a slow-fast system, when written in local coordinates which are adapted to the slow manifold and its stable and unstable foliations.
One application of this normal form was to derive the estimates used to prove the so-called Exchange Lemma and its extensions, which are useful tools for establishing the existence of heteroclinic and homoclinic orbits in slow-fast systems; see, e.g., \cite{jones1994tracking,jones1995geometric,jones1996tracking_exp,brunovsky1996tracking,kaper2001primer, liu2006geometric,schecter2008exchange,jones2009generalized} and the references therein.
In the special case that the slow manifold is attracting, another application of the Fenichel Normal Form is to dimensionality reduction: in this normal form, the dynamics of the transformed slow variable are decoupled from the transient dynamics of the transformed fast variable, and therefore the transformed slow dynamics serves as a reduction of the full dynamics in a clear way.
Stated differently, in the coordinates placing the system in Fenichel Normal Form, the map $\Ps$ sending stable fibers to their basepoints is simply an orthogonal projection; the coordinate change ``straightens out'' the stable fibers $\Ws(m)$, for $m$ in the slow manifold. 
\linelabel{E_4}We remark that the recent paper \cite{jones2009generalized} is a useful source of historical information on the Fenichel Normal Form.

As mentioned above, for our applications to GSP we will assume the special case in which the critical manifold is a NAIM. 
This special case arises naturally in many concrete applications, such as in understanding nonholonomic dynamics as a limit of friction forces \cite{eldering2016realizing}, in biolocomotion \cite{eldering2016role}, in the context of chemical reactions and combustion \cite{lam1994csp}, in various problems in control theory \cite{kokotovic1976singular}, and many more \cite[Ch. 20]{kuehn2015multiple}.
For general background on GSP, one may consult, e.g., the seminal paper \cite{fenichel1979geometric}, the expository articles \cite{kaper1999systems,jones1995geometric}, or the recent book \cite{kuehn2015multiple}.

Our two applications are as follows.

\begin{enumerate}

	\item Using Theorem 
	\ref{th:fiber_bundle_theorem} and assuming that the slow manifold is a NAIM, we show that the Fenichel Normal Form is valid on the union of \emph{global} stable manifolds $\cup_\epsilon\Ws(K_\epsilon)$ of slow manifolds $K_\epsilon$, rather than just on the union of local stable manifolds $\cup_\epsilon \Wsl(K_\epsilon)$.
	This is the content of Theorem \ref{th:GSP_global_fen_norm_form}.
	
	\item Using Theorem \ref{th:smooth_inflowing_global_linearization} and assuming that the slow manifold is a NAIM, we show that under additional spectral assumptions on the critical manifold, there exists a stronger normal form which is \emph{linear} in the fast variables.
	This normal form can be viewed as a stronger version of the Fenichel Normal Form.
	Additionally, this linear normal form is also valid on the union of \emph{global} stable manifolds $\cup_\epsilon\Ws(K_\epsilon)$ of slow manifolds $K_\epsilon$.
	This is the content of Theorem \ref{th:GSP_sakamoto_linearization}.
\end{enumerate}

The remainder of this section is as follows.
We first introduce the context for Theorems \ref{th:GSP_global_fen_norm_form} and \ref{th:GSP_sakamoto_linearization} by describing the GSP setup in \S \ref{sec:GSP_setup}.
Next, \S \ref{sec:GSP_global_fen_norm_form} contains the global extension of the Fenichel Normal Form as an application of Theorem \ref{th:fiber_bundle_theorem}.
Following this, \S \ref{sec:GSP_linearized} contains the derivation of the linear normal form, as well as its global extension, as an application of Theorem \ref{th:smooth_inflowing_global_linearization}.
Next, in in \S \ref{sec:GSP_discuss} we discuss our results and relate them to the so-called method of straightening out fibers (SOF method) recently appearing in the literature \cite{kristiansen2014_SOF_method}.
Finally, in \S \ref{sec:GSP_example} we illustrate our results in an example involving a classical mechanical system. 

\subsection{Setup and classic results}\label{sec:GSP_setup}

Consider a singularly perturbed system of the form
\begin{equation}\label{eq:GSP-fast}
\begin{split}
x' &= f(x,y,\epsilon)\\
\epsilon y' &= g(x,y,\epsilon),
\end{split}
\end{equation}
where $x \in \R^{n_x}$ and $y \in \R^{n_y}$ are functions of ``slow time'' $\tau$, $\epsilon$ is a small parameter, and\footnote{Note that we have adopted Fenichel's convention of letting $x$ denote the ``slow'' variable here, as a matter of personal style.} $f,g \in C^{r \geq 2}$.
For all $\epsilon \neq 0$, this system is equivalent via a time-rescaling $t = \tau/\epsilon$ to the regularized system
\begin{equation}\label{eq:GSP-slow}
\begin{split}
\dot{x} &= \epsilon f(x,y,\epsilon)\\
\dot{y} &= g(x,y,\epsilon).
\end{split}
\end{equation}
We let a ``prime'' denote a derivative with respect to $\tau$, and a ``dot'' denote a derivative with respect to the ``fast time'' $t$.
 
Now suppose that $K_0 \subset \interior \widehat{K}_0 \subset \widehat{K}_0$ are compact manifolds with boundary contained in $S\coloneqq \{(x,y):g(x,y,0)=0\}$, with $\interior \widehat{K}_0$ denoting the manifold interior of $\widehat{K}_0$.
Noting that $S$ consists of critical points of the $(\epsilon = 0)$ system, let us assume that the eigenvalues of $\D_2 g(x,y,0)$ have strictly negative real part on $\widehat{K}_0$. In particular, this implies that $\widehat{K}_0$ can be locally written as a graph $\widehat{K}_0 \coloneqq \{ (x,F(x)) \}$ over some domain $B \subset \R^{n_x}$.

By making local modifications to the vector field defined by \eqref{eq:GSP-slow} in arbitrarily small neighborhoods of $\partial K_0$ and $\partial \widehat{K}_0$, we may henceforth assume without loss of generality that the vector field is inward pointing at $\partial K_0$ and outward pointing at\footnote{Similar constructions are carried out in greater detail in \cite[\S~2]{josic2000synchronization}.} $\partial \widehat{K}_0$.

By our assumption on the eigenvalues of $\D_2 g$, we have that $K_0\times \R$ and $\widehat{K}_0\times \R$ are noncompact NAIMs for the dynamics
\begin{equation}\label{eq:GSP_aug_eps_zero}
\begin{split}
\dot{x} &= 0\\
\dot{y} &= g(x,y,0)\\
\dot{\tilde{\epsilon}} &= 0,
\end{split}
\end{equation}
since the equations for $\dot{x}$ and $\dot{y}$ in \eqref{eq:GSP_aug_eps_zero} are independent of $\tilde{\epsilon}$.
Here, $\tilde{\epsilon}\in \R$ is a new parameter, and its relation to $\epsilon$ will be determined subsequently.
We compactify these NAIMs by replacing $\R$ with its one-point compactification $S^1$, and we thereby henceforth consider \eqref{eq:GSP_aug_eps_zero} to be defined on $\R^{n_x+n_y} \times S^1$.
For this new domain of definition, $K_0\times S^1$ and $\widehat{K}_0\times S^1$ are compact inflowing and overflowing NAIMs, respectively.

Next, following \cite[p. 142]{eldering2013normally}, we use a scaling parameter $\kappa > 0$ to slowly ``turn on'' the $\tilde{\epsilon}$ dependence. 
Let $\chi\colon \R\to [0,1]$ be a $C^\infty$ nonnegative bump function such that $\chi \equiv 1$ on $[-1,1]$ and $\chi \equiv 0$ on $\R\setminus (-2,2)$, and --- anticipating a parameter substitution $\epsilon = \kappa \tilde{\epsilon}$ --- consider the vector field defined by
\begin{equation}\label{eq:GSP_aug_saturate}
\begin{split} 
\dot{x} &=  \chi(\tilde{\epsilon})\kappa\tilde{\epsilon} f(x,y,\chi(\tilde{\epsilon})\kappa\tilde{\epsilon})\\
\dot{y} &=  g(x,y,\chi(\tilde{\epsilon})\kappa\tilde{\epsilon}) \\
\dot{\tilde{\epsilon}} &= 0.
\end{split}
\end{equation}
One can verify that this vector field can be made arbitrarily $C^r$-close to \eqref{eq:GSP_aug_eps_zero} by taking $0 < \kappa \ll 1$ sufficiently small.
It follows from Fenichel's theorem on persistence of overflowing NAIMs \cite[Thm~1]{fenichel1971persistence} that there exists a $\kappa > 0$ such that $\widehat{K}_0\times S^1$ persists to a $C^r$-nearby overflowing $r$-NAIM for \eqref{eq:GSP_aug_saturate}, and that $K_0\times S^1$ persists to an inflowing NAIM inside it, since $K_0\times S^1\subset \widehat{K}_0\times S^1$ and inflowing invariance is an open condition.
Because $K_0$ consisted entirely of critical points for \eqref{eq:GSP_aug_eps_zero}, by a theorem of Fenichel the local stable foliation of the inflowing NAIM is $C^{r-1}$ \cite[Thm~5]{fenichel1977asymptotic}.

We now make the change of variables $\epsilon = \kappa \tilde{\epsilon}$ and see that \eqref{eq:GSP_aug_saturate} is equivalent to
\begin{equation}\label{eq:GSP_aug_sat_notilde}
\begin{split}
\dot{x} &=  \chi(\epsilon/\kappa)\epsilon f(x,y,\chi(\epsilon/\kappa)\epsilon)\\
\dot{y} &=  g(x,y,\chi(\epsilon/\kappa)\epsilon) \\
\dot{\epsilon} &= 0,
\end{split}
\end{equation}
so it follows that \eqref{eq:GSP_aug_sat_notilde} has compact $r$-NAIMs $M$ and $\widehat{M}$ which are respectively inflowing and overflowing, and with $M$ contained in the manifold interior of $\widehat{M}$.
Since $M$ and $\widehat{M}$ are the images of the NAIMs for \eqref{eq:GSP_aug_saturate} through a diffeomorphism, the local stable foliation $\Wsl(M)$ of $M$ for \eqref{eq:GSP_aug_sat_notilde} is also $C^{r-1}$ in all variables $x,y,\epsilon$.

\subsection{Globalizing the Fenichel Normal Form}\label{sec:GSP_global_fen_norm_form} 

Continuing the analysis of \S \ref{sec:GSP_setup}, we may apply Theorem \ref{th:fiber_bundle_theorem} to deduce that the leaves of the global stable foliation of $M$ for \eqref{eq:GSP_aug_sat_notilde} fit together to form a $C^{r-1}$ disk bundle $\Ps\colon\Ws(M)\to M$ isomorphic (as a disk bundle) to $E^s$.
By the definition of $\chi$ we see that for $\epsilon \in [-\kappa,\kappa]$, \eqref{eq:GSP_aug_sat_notilde} reduces to the system
\begin{equation}\label{eq:GSP_aug}
\begin{split}
\dot{x} &= \epsilon f(x,y,\epsilon), \\
\dot{y} &=  g(x,y,\epsilon), \\
\dot{\epsilon} &= 0.
\end{split}
\end{equation}
As in \cite{jones1995geometric}, let us make the commonly made assumption\footnote{%
	A more general situation where $M$ cannot be written as a graph can be handled using a tubular neighborhood modeled on the normal bundle of $M$.
} that $\widehat{K}_0$ is the graph of a map $B\subset \R^{n_x} \to \R^{n_y}$, where $B$ is a closed ball in $\R^{n_x}$ --- as we have already noted, by the implicit function theorem this can always be achieved by shrinking $\widehat{K}_0$ if necessary.
Thus, if $\kappa$ is sufficiently small, we can write $M$ as the graph of a $C^r$ map $y = F(x,\epsilon)$ defined on a suitable open subset of $\R^{n_x}\times S^1$.
Making the coordinate change $(x,y,\epsilon)\mapsto (x, y-F(x,\epsilon),\epsilon)$, we may assume that $M$ is contained in $\R^{n_x}\times \{0\} \times S^1$.
Since we assumed that $\widehat{K}_0$ is contractible it follows that $M$ deformation retracts onto $S^1$, and hence the bundle $\Ps\colon \Ws(M) \to M$ must be trivializable over any subset of the form $M_0\coloneqq M\cap(\R^{n_x+n_y}\times (-\epsilon_0,\epsilon_0))$, for any sufficiently small $\epsilon_0>0$.
It follows that there exists a $C^{r-1}$ fiber-preserving diffeomorphism $\Ws(M_0)\cong M_0 \times \R^{n_y}$ of the form $(x,y,\epsilon) \mapsto (\tilde{x},\tilde{y},\epsilon)\coloneqq (\Ps(x,y,\epsilon), \phi(x,y,\epsilon),\epsilon)$, with $\phi(x,0,\epsilon)\equiv 0$.
Making this final coordinate change, it follows that when restricted to $\Ws(M_0)$, the system \eqref{eq:GSP_aug} takes the form:
\begin{equation}\label{eq:fen_norm_form_aug}
\begin{split}
\dot{\tilde x} &= \epsilon h(\tilde x,\epsilon), \\
\dot{\tilde y} &= \Lambda (\tilde x,\tilde y,\epsilon)\tilde y, \\
\dot{\epsilon} &= 0,
\end{split}
\end{equation}
where $(\tilde{x},\tilde{y},\epsilon) \mapsto \Lambda(\tilde x ,\tilde y,\epsilon)$ is a $C^{r-3}$ family of $n_x \times n_x$ matrices and $(\tilde{x},\epsilon)\mapsto h(\tilde{x},\epsilon)$ is\footnote{However, the maps $(\tilde{x},\epsilon)\mapsto \epsilon h(\tilde{x},\epsilon)$ and $(\tilde{x},\tilde{y},\epsilon)\mapsto \Lambda (\tilde x,\tilde y,\epsilon)\tilde {y}$ are $C^{r}$ and $C^{r-2}$, respectively. The first map is $C^r$ because $\epsilon h(x,\epsilon) \equiv f(x,F(x,\epsilon),\epsilon)$, and the right hand side is $C^r$ in $x$ and $\epsilon$.} $C^{r-1}$.
The $\dot{\tilde x}$ equation depends only on $\tilde x$ and $\epsilon$ because we are using an invariant fiber bundle trivialization for coordinates on $\Ws(M_0)$, and $\tilde{x}$ and $\tilde{\epsilon}$ are coordinates for $M_0$.
By our choice of coordinates, $\dot{\tilde x}$ is zero when $\epsilon = 0$ because $\dot{x} = 0$ when $\epsilon = 0$ --- this fact and Hadamard's Lemma implies that $\dot{\tilde x}$ is of the form $\epsilon h(\tilde x,\epsilon)$.
Hadamard's Lemma similarly implies that $\dot{\tilde{y}}$ is of the form $\Lambda(\tilde{x},\tilde{y},\epsilon)\tilde y$, because after our coordinate changes $M_0$ corresponds to the set of points in $\Ws(M_0)$ with $\tilde{y} = 0$, and also $M_0$ is positively invariant, so it must be the case that $\dot{\tilde y} = 0$ when $\tilde{y} = 0$.
Suppressing the $\dot{\epsilon} = 0$ equation, we have proven the following result, which we record here as a theorem.

\begin{Th}\label{th:GSP_global_fen_norm_form}
	Assume that $\widehat{K}_0$ can be written as the graph of a $C^r$ map $B\subset \R^{n_x}\to \R^{n_y}$, with $B$ a closed ball in $\R^{n_x}$.
	Then there exists  $\kappa > 0$ such that for any $\epsilon \in [-\kappa,\kappa]$, there is a $C^{r-1}$ fiber-preserving diffeomorphism $\varphi_\epsilon\colon \Ws(K_\epsilon)\to K_\epsilon \times \R^{n_y}$ such that in the coordinates $\tilde{x},\tilde{y}=\varphi_\epsilon(x,y)$, the system \eqref{eq:GSP-fast} takes the form
	\begin{equation}\label{eq:fen_norm_form}
	\begin{split}
	\dot{\tilde x} &= \epsilon h(\tilde x,\epsilon), \\
	\dot{\tilde y} &= \Lambda (\tilde x,\tilde y,\epsilon)\tilde y.
	\end{split}
	\end{equation}
	In the new coordinates, $K_{\epsilon}$ corresponds to $\{(\tilde{x},\tilde{y})|\tilde{y}=0\}$.
	The diffeomorphism $\rho_{\epsilon}$ is $C^{r-1}$ in $\epsilon$.
	Also, $h\in C^{r-1}$, $\Lambda\in C^{r-3}$, and the function $(\tilde{x},\epsilon)\mapsto \epsilon h(\tilde{x},\epsilon)$ is $C^{r}$.
	
	Restricting attention now to only positive values of $\epsilon> 0$, in the original slow time-scale \eqref{eq:fen_norm_form} is equivalent to 
	\begin{equation}\label{eq:fen_norm_form_slow_time}
	\begin{split}
	\tilde{x}' &= h(\tilde x,\epsilon), \\
	\epsilon \tilde{y}' &= \Lambda (\tilde x,\tilde y,\epsilon)\tilde y.
	\end{split}
	\end{equation}
\end{Th}

\begin{Rem}
Because of our assumption that the critical manifold of \eqref{eq:GSP-slow} was a NAIM, the normal form which we were able to derive and state in Theorem \ref{th:GSP_global_fen_norm_form} appears considerably simpler than the \concept{Fenichel Normal Form} --- c.f. \cite[p.~82]{jones1995geometric}, \cite[pp.~109-111]{kaper1999systems}, \cite[p. 973]{jones2009generalized} or \cite[pp.~72--73]{kuehn2015multiple}, although our normal form actually directly follows from the general Fenichel Normal Form.	
Our contribution is that, using Theorem \ref{th:fiber_bundle_theorem}, we have shown that this normal form is valid on a neighborhood which consists of the entire union of \emph{global} stable manifolds $\cup_{\epsilon}\Ws(K_\epsilon)$, as opposed to being valid merely on the union of local stable manifolds $\cup_\epsilon \Wsl(K_\epsilon)$.
\end{Rem}

\subsection{Smooth global linearization: a stronger GSP normal form}\label{sec:GSP_linearized}
In this section we continue to assume that the critical manifold is a NAIM for \eqref{eq:GSP-slow}, but we make the following additional ``nonresonance'' assumption on the eigenvalues of the critical points. 
Let $r_{\min}(x,y)\leq r_{\max}(x,y)<0$ denote the minimum and maximum real parts of eigenvalues of $D_2g(x,y)$, where $(x,y) \in \widehat{K}_0$, and $\widehat{K}_0$ is defined following \eqref{eq:GSP-slow}.
We assume that there exist negative real constants $\alpha,\beta$ such that $2\alpha < \beta < \alpha < 0$ and
\begin{equation}\label{eq:GSP_sakamoto_eval_assumption}
\forall (x,y) \in \widehat{K}_0: \beta < r_{\min}(x,y)\leq r_{\max}(x,y) < \alpha.
\end{equation}

The payoff for this assumption is that we can obtain a $C^{r-1}$ normal form which is \emph{linear} in $\tilde{y}$, improving upon the Fenichel Normal Form \eqref{eq:fen_norm_form} significantly.
This normal form is also global in the sense that it holds on the entire union of global stable manifolds $\cup_\epsilon \Ws(K_\epsilon)$.
This is the content of the following result.

\begin{Th}\label{th:GSP_sakamoto_linearization}
	Assume that the vector field defined by \eqref{eq:GSP-fast} is $C^{r\geq 3}$, and assume that the condition \eqref{eq:GSP_sakamoto_eval_assumption}  holds for the regularized system \eqref{eq:GSP-slow}, and that $\widehat{K}_0$ can be written as the graph of a $C^r$ map $B\subset \R^{n_x}\to \R^{n_y}$, with $B$ a closed ball in $\R^{n_x}$.
	Then there exists  $\kappa > 0$ such that for any $\epsilon \in [-\kappa,\kappa]$, there is a $C^{r-1}$ fiber-preserving diffeomorphism $\varphi_\epsilon\colon \Ws(K_\epsilon)\to K_\epsilon \times \R^{n_y}$ such that in the coordinates $\tilde{x},\tilde{y}=\varphi_\epsilon(x,y)$, the system \eqref{eq:GSP-fast} takes the form
	\begin{equation}\label{eq:GSP_lin_norm_form}
	\begin{split}
	\dot{\tilde x} &= \epsilon h(\tilde x,\epsilon), \\
	\dot{\tilde y} &= A (\tilde x,\epsilon)\tilde y.
	\end{split}
	\end{equation}
	In the new coordinates, $K_{\epsilon}$ corresponds to $\{(\tilde{x},\tilde{y})|\tilde{y}=0\}$.
	The diffeomorphism $\varphi_{\epsilon}$ is $C^{r-1}$ in $\epsilon$.
	Also, $h\in C^{r-1}$, $A\in C^{r-1}$, and the function $(\tilde{x},\epsilon)\mapsto \epsilon h(\tilde{x},\epsilon)$ is $C^{r}$.
	
	Restricting attention now to only positive values of $\epsilon> 0$, in the original slow time-scale \eqref{eq:GSP_lin_norm_form} is equivalent to
	\begin{equation}\label{eq:GSP_lin_norm_form_slow_time}
	\begin{split}
	\tilde{x}' &= h(\tilde x,\epsilon), \\
	\epsilon \tilde{y}' &= A(\tilde x,\epsilon)\tilde y.
	\end{split}
	\end{equation}	
	If the condition \eqref{eq:GSP_sakamoto_eval_assumption} does not hold, then there exists a homeomorphism $\varphi_\epsilon$ such that the same result holds, but $\varphi_\epsilon$ is generally not differentiable in that case. 
\end{Th}

\begin{proof}
	Consider the compact inflowing NAIM $M$ for the system \eqref{eq:GSP_aug_sat_notilde} defined on $\R^{n_x + n_y}\times S^1$. 
	As described in \S \ref{sec:GSP_global_fen_norm_form}, our assumption that $\widehat{K}_0$ is a graph implies that if $\kappa > 0$ is sufficiently small, then $M$ is the graph of a $C^r$ map $(x,\epsilon)\mapsto y$.
	Hence we may assume without loss of generality that $M\subset \R^{n_x}\times \{0\} \times S^1$.
	  
	Let $E^s$ be the stable vector bundle of $M$ and let $\Phi^t$ be the flow of the system \eqref{eq:GSP_aug_sat_notilde} on $\R^{n_x + n_y}\times S^1$.
    By continuity and compactness, it can be shown that assumption~\eqref{eq:GSP_sakamoto_eval_assumption} implies that if $\kappa > 0$ is sufficiently small, then
    there exist constants $\delta > 0$ and $K \geq 1$ such that $-\alpha > r \delta$, and such that for all $t \geq 0$
    \begin{equation}\label{eq:GSP_sakamoto_rates}
	\begin{aligned}
	K^{-1}e^{-\delta t} \leq \minnorm{\D \Phi^t|_{TM}} &\leq \norm{\D \Phi^t|_{TM}} \leq K e^{\delta t}, \\
	K^{-1}e^{-\delta t} \leq \minnorm{(\D \Phi^t|_{TM})^{-1}} &\leq \norm{(\D \Phi^t|_{TM})^{-1}} \leq K e^{\delta t}, \\
	K^{-1}e^{\beta t} \leq \minnorm{\D \Phi^t|_{E^s}} &\leq \norm{\D \Phi^t|_{E^s}} \leq K e^{\alpha t}
	\end{aligned}
    \end{equation}
    uniformly on $TM$ and $E^s$.
    By Theorem \ref{th:smooth_inflowing_global_linearization}, there exists a global $C^{r-1}$ fiber-preserving diffeomorphism $\varphi\colon \Ws(M)\to E^s$  which conjugates $\Phi^t|_{\Ws(M)}$ to $\D \Phi^t|_{E^s}$ and maps $M$ diffeomorphically onto the zero section of $E^s$.
    
    Now in any local trivialization of $E^s$, the vector field generating the flow $\D\Phi^t|_{E^s}$ is of the form \eqref{eq:GSP_lin_norm_form} augmented with $\dot{\epsilon}=0$ (where  coordinates for the zero section are given by $\tilde{x}$ and coordinates for the fibers given by $\tilde{y}$).
    It follows that if we define $\varphi_\epsilon(\slot,\slot) \coloneqq \varphi(\slot,\slot,\epsilon)$, then it suffices to show that $\Ws(M)$ is trivializable over the subset $M\cap (\R^{n_x+n_y}\times [-\kappa,\kappa])$. 
    But $M\cap (\R^{n_x+n_y}\times [-\kappa,\kappa])$ is contractible since it is diffeomorphic to $K_0\times [-\kappa,\kappa]$, so  $\Ws(M)$ is indeed trivializable over $M_{\kappa}$.

    The statement about \eqref{eq:GSP_lin_norm_form_slow_time} follows easily by replacing $t$ with the rescaled slow time $\tau = \epsilon t$.
    
    Finally, to justify the last statement for the case that \eqref{eq:GSP_sakamoto_eval_assumption} does not hold, we simply apply Corollary \ref{co:top_inflowing_global_linearization} instead of Theorem \ref{th:smooth_inflowing_global_linearization}.
    This completes the proof.
\end{proof}	

\begin{Rem}\label{rem:GSP_sakamoto_codim_1}
	Assume that $n_y = \dim(y) = 1$, so that the fast variable is one-dimensional and the slow manifold is codimension-1.
	Then the eigenvalue condition \eqref{eq:GSP_sakamoto_eval_assumption} can always be made to hold by taking $\widehat{K}_0$ sufficiently small.
\end{Rem}

\begin{Rem}
	We see from the proof that, since $K_\epsilon$ is a manifold with boundary, our linearization result Theorem \ref{th:smooth_inflowing_global_linearization} for inflowing invariant manifolds is crucial.
	This is because, to the best of our knowledge, all of the linearization results in the literature assume a  boundaryless invariant manifold \cite{pugh1970linearization,takens1971partiallyhyp,robinson1971differentiable,hirsch1977,palis1977topological,sell1983linearization,sell1983vector,sakamoto1994smooth,smoothInvariant}.
\end{Rem}

\subsection{Discussion}\label{sec:GSP_discuss}
   We have proven Theorems \ref{th:GSP_global_fen_norm_form} and \ref{th:GSP_sakamoto_linearization}, both of which are statements about normal forms for slow-fast systems in the framework of geometric singular perturbation theory (GSP).
   These results assume that the slow manifold is attracting.
   
   Let us first discuss some literature regarding the Fenichel Normal Form for attracting slow manifolds, which is the subject of Theorem \ref{th:GSP_global_fen_norm_form}.
   Because of the practical benefits afforded by dimensionality reduction, there has been interest in actually \emph{computing} the coordinate change placing the system in Fenichel Normal Form for the attracting slow manifold case.
   Recently, the so-called method of straightening out fibers (SOF method) has been developed to iteratively approximate the Taylor polynomials of this coordinate change\footnote{The results of \cite{kristiansen2014_SOF_method} actually apply in more general situations, such as the case of a normally elliptic slow manifold. 
   	See \cite{kristiansen2014_SOF_method} for more details.}
   \cite{kristiansen2014_SOF_method}; similar techniques for systems near equilibria were previously developed in \cite{roberts1989appropriate,roberts2000computer}, and we also mention that it was shown in \cite{zagaris2004fast} that the Computational Singular Perturbation (CSP) method initially developed in \cite{lam1989understanding,lam1993using} iteratively approximates the first-order Taylor polynomial of this coordinate change. 
	
   Theorem \ref{th:GSP_global_fen_norm_form} does not yield a new normal form; it shows that, in the attracting slow manifold case, the domain of the coordinate change placing the system in Fenichel Normal Form actually extends to the entire global stable manifold of the slow manifold.
   This result seems to be of primarily theoretical interest.
   For example, the state-of-the-art SOF method only provides a means for computing Taylor polynomials centered at the slow manifold.
   Since these Taylor polynomials are only guaranteed to accurately approximate the coordinate near the slow manifold, they are unlikely to approximate the global coordinate change.
   Hence the global coordinate change, guaranteed to exist by Theorem \ref{th:GSP_global_fen_norm_form}, might not be explicitly computable except in special cases.
   
   On the other hand, Theorem \ref{th:GSP_sakamoto_linearization} does yield a new normal form, and also shows that the domain of the associated coordinate change extends to the entire global stable manifold.
   In order for the coordinate change to be differentiable, some additional spectral conditions \eqref{eq:GSP_sakamoto_eval_assumption} need to be satisfied, although these are automatically satisfied on a small enough domain of the slow manifold in the codimension-1 case (see Remark \ref{rem:GSP_sakamoto_codim_1}). 
   The payoff is that this normal form is \emph{linear} in the fast variables.
   Furthermore, by combining the SOF method of \cite{kristiansen2014_SOF_method} with additional normal form computations \cite{guckenheimer1983nonlinear,roberts1989appropriate,roberts2000computer}  for the fast variable, it seems to us that it should be possible in principle to compute the Taylor polynomials of this coordinate change in a systematic way.
   We hope to explore this in future work.
   Of course, computing this coordinate system \emph{globally} suffers the same difficulties mentioned in the previous paragraph.	
   Finally, we observe that our normal form is quite similar in form to the dynamics produced by ``high-gain'' nonlinear control schemes --- suggesting that linearly controlled fast variables are an inherent feature of a broad class of systems, rather than a convenient requirement imposed by control theorists.
	
\subsection{Example}\label{sec:GSP_example}
In this section, we consider an example of a forced pendulum with damping.
This example was chosen so that the natural state space is not Euclidean.
This will allow us to illustrate Theorems \ref{th:fiber_bundle_theorem}, \ref{th:global_linearization}, and \ref{th:smooth_inflowing_global_linearization} by directly applying these theorems to obtain stronger results than those obtainable via Theorems \ref{th:GSP_global_fen_norm_form} and \ref{th:GSP_sakamoto_linearization}, which we formulated for dynamics on a Euclidean space. 
 
We allow the damping coefficient of the pendulum to be a function of the pendulum angle, and consider an applied torque which depends on the pendulum angle and time.
We assume that the applied torque is periodic in time, and for simplicity we assume that the period is $2\pi$. 
Specifically, we consider the equations of motion
\begin{align}\label{eq:GSP_pendulum_equation}
 \epsilon\theta'' + \frac{\epsilon g}{l}\sin\theta + c(\theta)\theta' = \tau(\theta,t),
\end{align}
where $\epsilon$ is the pendulum mass which we assume to be small\footnote{Strictly speaking, in a physical context we should define $\epsilon$ to be a dimensionless quantity in order to refer to it as ``small'' in an absolute sense. However, this will cause no problem whatsoever for applying and illustrating our results, and we therefore do not bother with this.}, $l$ is the pendulum length, $g$ is the acceleration due to gravity, $c$ is the angle-dependent damping coefficient, and $\tau(\theta,t)$ is the applied torque --- not to be confused with the slow time variable that is also denoted by $\tau$ with some abuse of notation.
We are assuming that $\forall \theta,t: \tau(\theta, t + 2\pi) = \tau(\theta,t)$.
We define the angular velocity $\omega\coloneqq \theta'$.
The periodicity of $\tau$ allows us to introduce a circular coordinate $\alpha$ and write \eqref{eq:GSP_pendulum_equation} in the following extended state space form:
\begin{equation}\label{eq:GSP_pendulum_eq_ss_slow_time}
\begin{split}
\theta' &= \omega\\
\alpha' &= 1\\
\epsilon\omega' &= -\frac{\epsilon g}{l}\sin\theta - c(\theta)\omega + \tau(\theta,\alpha) .
\end{split}
\end{equation}
We consider $(\theta,\alpha)$ to be angle coordinates on the two-torus $T^2\coloneqq S^1\times S^1$, so that the state space is $T^2 \times \R$.
As in \S \ref{sec:GSP_setup}, for $\epsilon \neq 0$ this ``slow time'' system is equivalent via a time-rescaling $t = \tau/\epsilon$ to the ``fast time'' system
\begin{equation}\label{eq:GSP_pendulum_eq_ss_fast_time}
\begin{split}
\dot{\theta} &= \epsilon\omega\\
\dot{\alpha} &= \epsilon\\
\dot{\omega} &= -\frac{\epsilon g}{l}\sin\theta - c(\theta)\omega + \tau(\theta,\alpha) .
\end{split}
\end{equation}
To relate this to our earlier notation from \S \ref{sec:GSP_setup}, here $(\theta,\alpha)$ is playing the role of $x$ and $\omega$ is playing the role of $y$.
For $\epsilon = 0$, the set of critical points of \eqref{eq:GSP_pendulum_eq_ss_fast_time}
are given by $S\coloneqq \{(\theta,\alpha,\omega)\colon c(\theta)\omega = \tau(\theta,\alpha)\}$.

Let us first consider the special case of a constant positive damping coefficient $c(\theta) \equiv c_0 > 0$.
Then $S$ is the graph of the map $F_0(\theta,\alpha)\coloneqq \frac{1}{c_0}\tau(\theta,\alpha)$.
We henceforth assume that $\tau \in C^r$, with\footnote{Even if $r = \infty$, we can only derive results for a finite smoothness degree. This is because persistent NHIMs generally have only a finite degree of smoothness, even if the dynamics are $C^\infty$ and the spectral gap is infinite \cite[Remark 1.12]{eldering2013normally}.} $3 \leq r < \infty$.
It follows that $S$ is a $C^r$ manifold diffeomorphic to the torus $T^2$.
Furthermore, the eigenvalues of all critical points in the critical manifold $S$ are readily checked to be $(0,0,-c_0)$, with the zero eigenvalues corresponding to the tangent spaces of $S$ and $-c_0$ corresponding to $\text{span}\{(0,0,1)\}$.
Therefore, $S$ is an $r$-NAIM for \eqref{eq:GSP_pendulum_eq_ss_fast_time} when $\epsilon = 0$.
Since $\partial S = \varnothing$, there exists $\epsilon_0 > 0$ such that for all $0 \leq \epsilon \leq \epsilon_0$, there is a unique persistent NAIM $S_\epsilon$ close to $S$, with $S_0 = S$.
As in \S \ref{sec:GSP_global_fen_norm_form}, $S_\epsilon$ is the graph of a $C^r$ map $\omega = F(\theta,\alpha,\epsilon)$ with $F_0 = F(\slot,\slot,0)$.

Using a technique from \cite{smith1999perturbation}, we next prove the following proposition.

\begin{Prop}\label{prop:GSP_pend_GAS_open}
	For all sufficiently small $\epsilon > 0$, $S_\epsilon$ is globally asymptotically stable.
	In other words, for all sufficiently small $\epsilon>0$, we have $\Ws(S_\epsilon) = T^2 \times \R$.
\end{Prop}

\begin{proof}
	We already know that $S_\epsilon$ is locally asymptotically stable for $\epsilon>0$ sufficiently small, so it suffices to show that $S_\epsilon$ is globally attracting for $\epsilon>0$ sufficiently small.
	We fix any $\epsilon_0 >0$ and define
	\begin{equation*}
	\eta\coloneqq \frac{1}{c_0}\left(\frac{\epsilon_0 g}{l} + \max_{(\theta,\alpha)\in T^2}|\tau(\theta,\alpha)| + 1\right).
	\end{equation*}
	Note that for all $0 \leq \epsilon \leq \epsilon_0$, the compact subset 
	\begin{equation*}D_\eta\coloneqq \left\{(\theta,\alpha,\omega)\colon|\omega| < \eta \right\}
	\end{equation*}
	of $T^2 \times \R$  is positively invariant, and every point in $(T^2\times \R)\setminus D_\eta$ will flow into $D_\eta$ in some finite time; indeed, $\dot{\omega}<-1$ on $(T^2 \times \R_{\geq 0})\setminus D_\eta$, $\dot{\omega}>1$ on $(T^2 \times \R_{\leq 0})\setminus D_\eta$, and the vector field points inward at $\partial D_\eta$.
	Therefore it suffices to show such that $S_\epsilon$ attracts all states in $D_\eta$ for sufficiently small $\epsilon > 0$.
	Next, by the same reasoning as in \S \ref{sec:GSP_setup}, we know that the compact set $\bigcup_{0 \leq \epsilon \leq \epsilon_0} S_\epsilon$ is locally asymptotically stable for the augmented dynamics (adding $\dot{\epsilon} = 0$) on $T^2 \times \R \times \R$.
	Hence there exists $\delta > 0$ such that for all $\epsilon>0$ sufficiently small, the basin of attraction of $S_\epsilon$ contains the set $N_\delta$ of points $(\theta,\alpha,\omega)\in T^2 \times \R$ satisfying $|\omega-F(\theta,\alpha,\epsilon)| < \delta$.
	In order to obtain a contradiction, suppose that there exist arbitrarily small values of $\epsilon > 0$ such that $S_\epsilon$ does not attract all states in $D_\eta$, and let $\Phi^t_{\epsilon}$ denote the flow of \eqref{eq:GSP_pendulum_eq_ss_fast_time}.
	Then there exist sequences $(\epsilon_n)_{n\in \N}$ and $(\theta_n,\alpha_n,\omega_n)_{n\in \N}\subset D_\eta$ such that $\epsilon_n \to 0$ and $\forall t > 0, n > 0: \Phi^{t}(\theta_n,\alpha_n,\omega_n) \not \in N_\delta$.
	Since $D_\eta$ is compact, by passing to a subsequence we may assume that $(\theta_n,\alpha_n,\omega_n) \to (\theta_0,\alpha_0,\omega_0) \in D_\eta$.
	Since $S_0$ is globally asymptotically stable for $\Phi^t_{0}$, for all sufficiently large $t > 0$, $\Phi^t_{0}(\theta_0,\alpha_0,\omega_0)\in N_{\delta/2}$.
	By continuity of the map $(t,\epsilon,\theta,\alpha,\omega)\mapsto \Phi^t_{\epsilon}(\theta_\epsilon,\alpha_\epsilon,\omega_\epsilon)$, it follows that for all  sufficiently large $t,n >0$, $\Phi^t_{\epsilon_n}(\theta_n,\alpha_n,\omega_n) \in N_\delta$.
	This is a contradiction, showing that for all sufficiently small $\epsilon > 0$, $\Ws(S_\epsilon) = T^2 \times \R$ for the dynamics \eqref{eq:GSP_pendulum_eq_ss_fast_time}. 
	
\end{proof}

Because the eigenvalues of the critical manifold are $(0,0,-c_0)$, after taking $\epsilon_1$ smaller if necessary we see that the $r$-center bunching conditions \eqref{eq:fiber-bundle-cor-center-bunching} are satisfied.
Therefore, Proposition \ref{prop:GSP_pend_GAS_open} and Corollary \ref{co:fiber-bundle-NAIM-center-bunching} of Theorem \ref{th:fiber_bundle_theorem} show that there exists a a $C^{r-1}$ diffeomorphism $\varphi_\epsilon\colon T^2 \times \R \to T^2 \times \R$ mapping $S_\epsilon$ onto $T^2\times \{0\}$ and mapping stable fibers of $S_\epsilon$ onto sets of the form\footnote{Here, and during the rest of this example, we are using the fact that the normal bundle --- and hence also the stable bundle $E^s$ --- of the slow manifold is trivial.} $(\theta,\alpha)\times \R$.
Using the coordinates $\tilde{\theta},\tilde{\alpha}, \tilde {\omega}=\varphi_\epsilon(\theta,\alpha,\omega)$ and changing back to the original time scale, \eqref{eq:GSP_pendulum_eq_ss_slow_time} takes the form
\begin{equation}\label{eq:GSP_pend_fen_normal_form}
\begin{split}
\tilde\theta' &= F(\tilde\theta,\tilde\alpha,\epsilon)\\
\tilde\alpha' &= 1\\
\epsilon\tilde\omega' &= \Lambda(\tilde\theta,\tilde\alpha,\tilde\omega,\epsilon)\tilde{\omega},
\end{split}
\end{equation}
for some function $\Lambda$.
The same reasoning as in \S \ref{sec:GSP_global_fen_norm_form} can be used to show that $\varphi_\epsilon$ is jointly $C^{r-1}$ in all variables including $\epsilon$.
This result should be compared with Theorem \ref{th:GSP_global_fen_norm_form}, which was formulated for dynamics on a Euclidean space.
We see that Theorem \ref{th:fiber_bundle_theorem} yields a global coordinate system on all of $T^2\times \R$ placing \eqref{eq:GSP_pendulum_eq_ss_slow_time} in the form \eqref{eq:GSP_pend_fen_normal_form}.
In contrast, without Theorem \ref{th:fiber_bundle_theorem} and using only the available results in the literature, we would have only been able to obtain such a coordinate system on a precompact neighborhood of $S_\epsilon$.

Alternatively, because the eigenvalues of the critical manifold are $(0,0,-c_0)$, after taking $\epsilon_1$ smaller if necessary we see that the stronger spectral conditions of Theorem \ref{th:smooth_inflowing_global_linearization} are also satisfied (c.f. \eqref{eq:GSP_sakamoto_eval_assumption}).
Hence Theorem \ref{th:smooth_inflowing_global_linearization} implies that there exists a global $C^{r-1}$ diffeomorphism $\psi_\epsilon\colon T^2 \times \R \to T^2 \times \R$ mapping $S_\epsilon$ onto $T^2\times \{0\}$ and mapping stable fibers of $S_\epsilon$ onto sets of the form $(\theta,\alpha)\times \R$.
Using the coordinates $\tilde{\theta},\tilde{\alpha}, \tilde {\omega}=\psi_\epsilon(\theta,\alpha,\omega)$ and changing back to the original time scale, \eqref{eq:GSP_pendulum_eq_ss_slow_time} takes the form
\begin{equation}\label{eq:GSP_pend_lin_normal_form}
\begin{split}
\tilde\theta' &= F(\tilde\theta,\tilde\alpha,\epsilon)\\
\tilde\alpha' &= 1\\
\epsilon \tilde\omega' &= A(\tilde\theta,\tilde\alpha,\epsilon)\tilde{\omega},
\end{split}
\end{equation}
for some function $A$.
The same reasoning as in \S \ref{sec:GSP_global_fen_norm_form} can be used to show that $\varphi_\epsilon$ is jointly $C^{r-1}$ in all variables including $\epsilon$.
This result should be compared with Theorem \ref{th:GSP_sakamoto_linearization}, which was formulated for dynamics on a Euclidean space.
We used Theorem \ref{th:smooth_inflowing_global_linearization} to derive \eqref{eq:GSP_pend_lin_normal_form}, but since $\partial S_\epsilon = \varnothing$ this result can also be obtained by combining Theorem \ref{th:global_linearization} with the local smooth linearization results of \cite{sakamoto1994smooth}.

Still considering \eqref{eq:GSP_pendulum_eq_ss_slow_time}, we will now consider specific choices of a non-constant damping function $c(\theta)$ and applied torque $\tau(\theta,\alpha)$ which will be chosen so that Theorem \ref{th:global_linearization} does not apply, but so that Theorem \ref{th:smooth_inflowing_global_linearization} does apply to yield a linear normal form.
For the sake of concreteness, let $c(\theta)\coloneqq \cos(\theta) + 1$ and $\tau(\theta,\alpha)\coloneqq -\sin(\theta)+(1/2)\cos(\alpha)$.
Then $c(\pi) = 0$, so it follows that
the critical set $S\coloneqq \{(\theta,\alpha,\omega)\colon c(\theta)\omega = \tau(\theta,\alpha)\}$ is not normally hyperbolic for the fast time system \eqref{eq:GSP_pendulum_eq_ss_fast_time} everywhere.
However, e.g. $c(\theta) > 1$ for $|\theta| < \pi/2$, so it follows in particular that the subset $K_0 \coloneqq \{(\theta,\alpha,\omega)\in S\colon |\theta| \leq \pi/4\}$ is $r$-normally attracting.
Furthermore, $K_0$ is inflowing for the slow time system \eqref{eq:GSP_pendulum_eq_ss_slow_time} restricted to $S$ when $\epsilon = 0$,
because $K_0$ is the graph of $F(\theta,\alpha,0)$ with
\begin{equation*}
F(\theta,\alpha,0)\coloneqq \frac{\tau(\theta,\alpha)}{c(\theta)} =  \frac{-\sin(\theta)+(1/2)\cos(\alpha)}{\cos(\theta)+1}
\end{equation*}
with $|\theta| \leq \pi/4$.
Therefore, the projection of the slow time dynamics restricted to $K_0$ are given by
\begin{equation*}
\begin{split}
\theta' &= \frac{-\sin(\theta)+(1/2)\cos(\alpha)}{\cos(\theta)+1}\\
\alpha' &= 1
\end{split}
\end{equation*}
and clearly the vector field points inward at the boundary of $\{(\theta,\alpha)\colon |\theta| \leq \pi/4\}$.
We can modify the flow locally near the boundary of any larger set $\widehat{K}_0\supset K_0$ to render $\widehat{K}_0$ overflowing,
and therefore there exists $\epsilon_0>0$ such that for all $0 \leq \epsilon \leq \epsilon_0$, $\widehat{K}_0$ (and hence also $K_0$) persists to a nearby $r$-NAIM for the fast time system \eqref{eq:GSP_pendulum_eq_ss_fast_time}.
Since inward pointing of a vector field is an open condition, after possibly shrinking $\epsilon_0$ it follows that $K_\epsilon$ is also inflowing for all $0 \leq \epsilon \leq \epsilon_0$.
Additionally, after possibly shrinking $\epsilon_0$, we see that the hypotheses of Theorem \ref{th:smooth_inflowing_global_linearization} are satisfied for $K_\epsilon$ for all $0 \leq \epsilon \leq \epsilon_0$ (check that \eqref{eq:GSP_sakamoto_eval_assumption} is satisfied on $K_0$ by using $\alpha  = -\sqrt{2}/2-1 \approx 1.7$ and $\beta = -2$, and use the fact that the hypotheses of Theorem \ref{th:smooth_inflowing_global_linearization} are open conditions).
Hence Theorem \ref{th:smooth_inflowing_global_linearization} implies that there exists a $C^{r-1}$ diffeomorphism $\psi_\epsilon\colon \Ws(K_\epsilon) \to K_\epsilon \times \R$ mapping $K_\epsilon$ onto $K_\epsilon\times \{0\}$ and mapping stable fibers of $K_\epsilon$ onto sets of the form $\{(\theta,\alpha)\}\times \R$.
Using the coordinates $\tilde{\theta},\tilde{\alpha}, \tilde {\omega}=\psi_\epsilon(\theta,\alpha,\omega)$ and changing back to the original time scale, \eqref{eq:GSP_pendulum_eq_ss_fast_time} takes the form
\begin{equation}\label{eq:GSP_pend_lin_normal_form_inflowing}
\begin{split}
\tilde\theta' &= F(\tilde\theta,\tilde\alpha,\epsilon)\\
\tilde\alpha' &= 1\\
\epsilon \tilde\omega' &= A(\tilde\theta,\tilde\alpha,\epsilon)\tilde{\omega},
\end{split}
\end{equation}
for suitable functions $A$ and $F$.
The same reasoning as in \S \ref{sec:GSP_global_fen_norm_form} can be used to show that $\varphi_\epsilon$ is jointly $C^{r-1}$ in all variables including $\epsilon$.
This result should be compared with Theorem \ref{th:GSP_sakamoto_linearization}, which was formulated for dynamics on a Euclidean space.
Here we had to use Theorem \ref{th:smooth_inflowing_global_linearization} to derive \eqref{eq:GSP_pend_lin_normal_form_inflowing}, because Theorem \ref{th:global_linearization} does not apply since $\partial K_\epsilon \neq \varnothing$.
Without Theorem \ref{th:smooth_inflowing_global_linearization} and using only the explicitly available results in the literature, we would not have been able to obtain even a local version of this coordinate system.

Finally, we note that the Taylor polynomials of the coordinate change for the normal form \eqref{eq:GSP_pend_fen_normal_form} can in principle be obtained using the SOF method, although as mentioned in \S \ref{sec:GSP_discuss} this does not help to compute the coordinates \emph{globally}. 
We do not pursue this here.
As mentioned in \S \ref{sec:GSP_discuss}, we believe it should be possible in principle to additionally compute the Taylor polynomials of the coordinate changes for the normal forms \eqref{eq:GSP_pend_lin_normal_form} and \eqref{eq:GSP_pend_lin_normal_form_inflowing}, which we hope to explore in future work.

\section{Conclusion}\label{sec:conclusion}

Stated technically, we have proven some results for NHIMs which are of two types: (i) global versions of well-known local results, and (ii) linearization results for inflowing NAIMs.
\linelabel{R1_4_b}We restricted our attention to flows.

We first showed that the global stable foliation of an inflowing NAIM is a fiber bundle, with fibers coinciding with the leaves of the global stable foliation, and that this fiber bundle is as smooth as the local stable foliation.
From that result, we deduced the corresponding result for the global (un)stable foliation of a general NHIM, though one needs to be careful in interpreting this statement as the global (un)stable manifold is generally only an immersed submanifold of $Q$.

We next considered global linearizations, and showed that the linearization result of \cite{pugh1970linearization,hirsch1977} for boundaryless NHIMs applies also to inflowing NAIMs.
Furthermore, this linearization extends to the entire global stable manifold --- inflowing NAIMs are globally linearizable, or topologically conjugate to the flow linearized at the NAIM. 
If some additional spectral gap conditions are assumed, then the global linearizing conjugacy can be taken to be $C^k$.
This extends the results of \cite{lan2013linearization} to the case of arbitrary inflowing NAIMs (although see Remark \ref{rem:mezic_remark}).
A key tool in our proof was the geometric construction of Appendix \ref{app:wormhole}, which allowed us to reduce to the boundaryless case.

We then used our theoretical results to give two applications to slow-fast systems with attracting slow manifolds, in the context of geometric singular perturbation theory (GSP).
First, using our fiber bundle theorem we extended the domain of the Fenichel Normal Form \cite{jones1994tracking,jones1995geometric,kaper1999systems}.
Second, under an additional spectral gap assumption, we derived a global smooth \emph{linear} normal form for GSP problems. 
If the slow manifold is codimension-1, this assumption can always be made to hold (after possibly shrinking the slow manifold; see Remark \ref{rem:GSP_sakamoto_codim_1}).
For this application it was essential that we proved a linearization theorem for inflowing NAIMs, since the slow manifolds appearing in slow-fast systems typically have boundary.
We then illustrated these results on an example of a mechanical system.
We noted that it might be interesting to combine the method of straightening out fibers (SOF method) of \cite{kristiansen2014_SOF_method} with additional normal form computations \cite{guckenheimer1983nonlinear,roberts1989appropriate,roberts2000computer} for the fast variable, in order to develop a systematic technique for computing the Taylor polynomials of the coordinate change for the linear normal form. 
We hope to explore this idea in future work.

Less formally, what we have shown is that the local structure next to an inflowing NAIM extends globally, in terms of structure (as a disk bundle), in its degree of smoothness, and in the fact that the dynamics are often conjugate to their linearization. 
In fact, the linearization is so robust that it can be extended consistently to yield a system linear in its fast variables throughout all sufficiently small perturbations of a singularly perturbed system.

We have considered only compact NHIMs and compact inflowing NAIMs in stating our results.
From our experience, we expect that extending these results to noncompact manifolds should be possible, but possibly quite technical.
However, our results for compact inflowing NAIMs allow our work to be applied to (for example) positively invariant compact subsets of the phase space of a mechanical system.

\appendix
	
\section{Smoothness of linear parallel transport covering an inflowing invariant manifold}\label{app:linear-par-transp}	

In this appendix, we show that a $C^r$ flow on an inflowing invariant manifold $M \subset Q$ can always be lifted to a $C^r$ linear flow on $E$, where $\pi\colon E\to M$ is any $C^r$ subbundle of $T Q|_M$.  
For the definition of a fiber metric \cite[p. 116]{kobayashi1963foundationsV1} see Def. \ref{def:fiber_metric} in Appendix \ref{app:fiber_bundles}.

\begin{Lem}\label{lem:smooth_parallel_transport}
	Let $M$ be a $C^r$ inflowing invariant submanifold of $Q \in C^\infty$ for the flow $\Phi^t$ generated by a $C^r$ vector field.
	Let $\pi\colon E\to M$ be a $C^r$ subbundle of $TQ|_M$ equipped with any fiber metric $g$.
	Then there exists a $C^r$ fiber metric $h$ on $E$ arbitrarily close to $g$ and a $C^r$ flow  $\Pi^t$ on $E$ such that for all $t>0$, $\Pi^t\coloneqq   \Pi(t,\slot)$ is an isometry with respect to the fiber metric $h$, covering $\Phi^t|_M$.
\end{Lem}

\begin{proof}
	We define a $C^r$ submanifold (without boundary) $M_\epsilon$ by the formula $M_\epsilon\coloneqq \Phi^{-\epsilon}(\interior M)$, with $\interior M$ denoting the manifold interior of $M$.
	Because $M$ is inflowing invariant, $M \subset M_\epsilon$.
	We extend $E$ arbitrarily to a $C^r$ subbundle $E_\epsilon \supset E$ of $TQ|_{M_\epsilon}$.
	In \cite[App.~1]{palis1977topological} it is shown that $M_\epsilon$ has a compatible $C^{r+1}$ differentiable structure with respect to which the vector field $f$ restricted to $M_\epsilon $ is\footnote{The theorem in \cite[App.~1]{palis1977topological} is stated for a $C^1$ invariant manifold and $C^1$ vector field, but the same proof works, mutatis mutandis, for a \emph{locally} invariant $C^r$ manifold and $C^r$ vector field, which is our situation here.} $C^r$.

	Denote $M_\epsilon  $ with this $C^{r+1}$ structure by $\widetilde M_\epsilon  $, and let $I\colon \widetilde M_\epsilon  \to M_\epsilon  $ be the $C^r$ diffeomorphism which is the identity map when viewed as a map of sets.
	Thus the pullback bundle $I^* E_\epsilon$ is a $C^r$ vector bundle over $\widetilde M_\epsilon  $ which is $C^r$ isomorphic to $E_\epsilon$ via a vector bundle isomorphism $G_1\colon I^* E_\epsilon \to E_\epsilon$ covering $I$ \cite[p.~97]{hirsch1976differential}.
	Furthermore, a standard argument using a universal bundle shows that there exists a $C^{r+1}$ vector bundle $\widetilde{I^*E_\epsilon}$ over $\widetilde M_\epsilon  $ and a $C^r$ vector bundle isomorphism $G_2\colon \widetilde{I^*E_\epsilon} \to I^*E_\epsilon$ covering the identity \cite[p.~101, Thm~3.5]{hirsch1976differential}.
	This situation is depicted in the following diagram.
	\begin{equation}\label{eq:diagram_pullback_E}
	\begin{tikzcd}
	&\widetilde{I^*E_\epsilon}\arrow{r}{G_2}\arrow{d}{P}
	&I^*E_\epsilon \arrow{r}{G_1}\arrow{d}{P}  & E_\epsilon \arrow{d}{\pi}\\
	&\widetilde M_\epsilon   \arrow{r}{\id_{\widetilde M_\epsilon  }}&\widetilde M_\epsilon   \arrow{r}{I}  &M_\epsilon  
	\end{tikzcd}
	\end{equation}
	Now pull back the fiber metric $g$ on $E_\epsilon$ to $\widetilde{I^*E_\epsilon}$. That is, define $\tilde{g}$ through
	\begin{equation*}
	\tilde{g}(v,w) = G^*(g)(v,w) = g(G (v),G (w)),
	\end{equation*}
	where $G = G_1 \circ G_2$.
	Now choose a $C^{r+1}$ fiber metric $\tilde{h}$ on $\widetilde{I^*E_\epsilon}$ that is close to $\tilde{g}$.
	Let $\widetilde \nabla\colon \Gamma(\T \widetilde M_\epsilon   \otimes \widetilde{I^*E_\epsilon})\to \Gamma(\widetilde{I^*E_\epsilon})$ be a $C^r$ affine connection compatible with the metric $\tilde{h}$ \cite[Chap.~3]{kobayashi1963foundationsV1}.
	Then the map of parallel transport along solution curves of $f$, $\widetilde \Pi^t\colon \widetilde{I^*E_\epsilon} \to \widetilde{I^*E_\epsilon}$, is an isometry since $\widetilde \nabla$ is compatible with $\tilde{h}$, and it is $C^r$ because with respect to local coordinates $x_1,\ldots, x_{n_m}$ and any local frame $(\sigma_1,\ldots, \sigma_{n_s})$, the parallel transport equation takes the form
	\begin{equation}\label{eq:parallel_trans}
	\sum_{k}\left(\frac{d}{dt}v^k\circ \Phi^t(x) + \sum_{i,j}\Gamma^k_{i,j}(v^i f^j)\circ \Phi^t(x)\right)\sigma_k\circ \Phi^t(x) = 0,
	\end{equation}
	where the Christoffel symbols $\Gamma^k_{i,j}$ defined by
	\begin{equation*}
	\widetilde \nabla_{\frac{\partial}{\partial x^i}}\sigma_j = \sum_k \Gamma^k_{i,j} \sigma_k
	\end{equation*}
	are $C^r$ functions $\Gamma^k_{i,j}\colon \widetilde{M}_\epsilon\to \R$.
	Since $f$ is a $C^r$ vector field with respect to the smooth structure of $\widetilde M_\epsilon   $, it follows that \eqref{eq:parallel_trans} defines a $C^r$ ODE for $v$ in local coordinates.
	The ODE theorems on existence, uniqueness, and smooth dependence on parameters imply that the solution to \eqref{eq:parallel_trans} depends smoothly on $x, v(x)$, and $t$. Thus $\widetilde \Pi\colon \R \times \widetilde{I^*E_\epsilon} \to \widetilde{I^*E_\epsilon}$ is indeed $C^r$.
	
	Next, define the fiber metric $h$ on $E_\epsilon$ by setting $h \coloneqq    (G_2^{-1}\circ G_1^{-1})^* \tilde{h}$ and define $\Pi\colon \R \times E_\epsilon \to E_\epsilon$ via
	\begin{equation}
	\Pi^t(v)\coloneqq    \Pi(t,v)\coloneqq    (G_1 \circ G_2) \circ \widetilde \Pi^t \circ (G_1\circ G_2)^{-1}(v).
	\end{equation}
	Since $\tilde{h}$ was arbitrarily close to $\tilde{g}$, the same holds for $h$ and $g$.
	The map $\Pi$ is $C^r$ because it is the composition of smooth functions.
	For any $t \in \R$, $\widetilde \Pi^t$ is an isometry of $(\widetilde{I^*E_\epsilon},\tilde{h})$, and our choice of the pullback metric $h$ on $E_\epsilon$ implies that $G_1\circ G_2$ is an isometry into $(E_\epsilon,h)$.
	Thus for any $t \in \R$, $\Pi^t$ is a composition of vector bundle isometries and is thus an isometry of vector bundles, hence preserves $h$.
	By construction $\Pi^t$ covers $\Phi^t|_{M_\epsilon   }$.
	
	Now $M$ is positively invariant under $\Phi^t$ since $M$ is inflowing invariant, hence also $E$ is positively invariant under $\Pi^t$.
	We therefore obtain a well-defined restriction of $\Pi^t$ to $M \subset M_\epsilon$ and also restrict $h$ to $E$, completing the proof.
\end{proof}

\section{Inflowing NAIMs: reduction to the boundaryless case}\label{app:wormhole}

In this appendix, we prove a result which shows roughly that any compact inflowing NAIM can always be viewed as a subset of a compact boundaryless NAIM.
In particular, this result allows the application of various linearization theorems from the literature \cite{pugh1970linearization,smoothInvariant,hirsch1977,sakamoto1994smooth,robinson1971differentiable,palis1977topological,sell1983linearization,sell1983vector,smoothInvariant} to inflowing NAIMs as in Corollaries \ref{th:smooth_inflowing_global_linearization}  and \ref{co:top_inflowing_global_linearization} and in \S \ref{sec:global_linearization}, despite the fact that in the literature these theorems are formulated only for boundaryless invariant manifolds.
We use this result in \S \ref{sec:applications} to derive a linear normal form result for singular perturbation problems in which the critical manifold is a NAIM.

First, let us describe the intuition behind our construction.
Let $M \subset Q$ be a compact, inflowing NAIM for some vector field on $Q$, and $N \supset M$ a slight extension along the backward flow.
We rip a hole in our space $Q$ by removing a small neighborhood $U_0$ of $\partial N$.
Then we glue two copies of $Q \setminus U_0$ together at their boundaries (thought of as a ``wormhole'') creating a total space $\widehat{Q}$.
We modify the copies of $N$ slightly such that they connect through the wormhole as a smooth, compact submanifold $\widehat{N} \subset \widehat{Q}$.
Finally, we carefully modify the vector field near the wormhole so that $\widehat{N}$ is a NAIM again for the modified vector field.

This procedure is made precise in the proof of Proposition~\ref{prop:wormhole} below, but let us already introduce some more details using Figure~\ref{fig:wormhole-2d}.
A family of smooth tubular neighborhoods $U_0\subset \ldots \subset U_3$ of $\partial N$ are chosen so that the vector field $f$ points inward at each $N\setminus U_i$, and so that each $\Wsl(M)\cap U_i = \varnothing$.
We smoothly rescale $f$ inside $U_3$ to create a vector field $\tilde{f}$ such that $\tilde{f}$ is zero on $\bar{U}_2$, and we smoothly approximate $N$ inside $U_2$ to create a submanifold $\widetilde{N}$ such that $\widetilde{N}\cap U_1$ is a $C^\infty$ submanifold.
We next create a copy of $Q$, remove the subset $U_0$ from each copy to form two copies of $Q'\coloneqq Q\setminus U_0$, and let $\widehat{Q}$ be the double of $Q'$ obtained by glueing the two copies of $Q'$ along $\partial Q' = \partial U_0$, forming a ``wormhole'' between the two spaces.
Using a standard technique from differential topology, we give $\widehat{Q}$ a $C^\infty$ differential structure such that $\widehat{N}\subset \widehat{Q}$ is a $C^r$ submanifold, where $\widehat{N}$ is comprised of the two copies of $\widetilde{N}$ (this step is the reason why we needed to approximate $N$ by $\widetilde{N}$).
We give $\widehat{Q}$ a Riemannian metric which agrees with the original metric on each copy of $Q'$ except on an arbitrarily small neighborhood of $\partial Q'$.
The vector field $\hat{f}_0$, defined to be equal to $\tilde{f}$ on each copy of $Q'$, is automatically $C^r$ since it is zero on a neighborhood of $\partial Q'$.
Finally, we modify $\hat{f}_0$ inside each copy of $U_3$ to create a vector field $\hat{f}$ on $\widehat{Q}$ such that $\widehat{N}$ is an $r$-NAIM for $\hat{f}$.
We show that the resulting global stable foliation $\widehat{W}^s(M)$ for $\hat{f}$ over a copy of $M$ agrees with the global stable foliation $\Ws(M)$ for $f$, and that certain asymptotic rates for $f$ are preserved by $\hat{f}$.

\begin{figure}[htb]
	\centering
	\def\svgwidth{.8\columnwidth}
	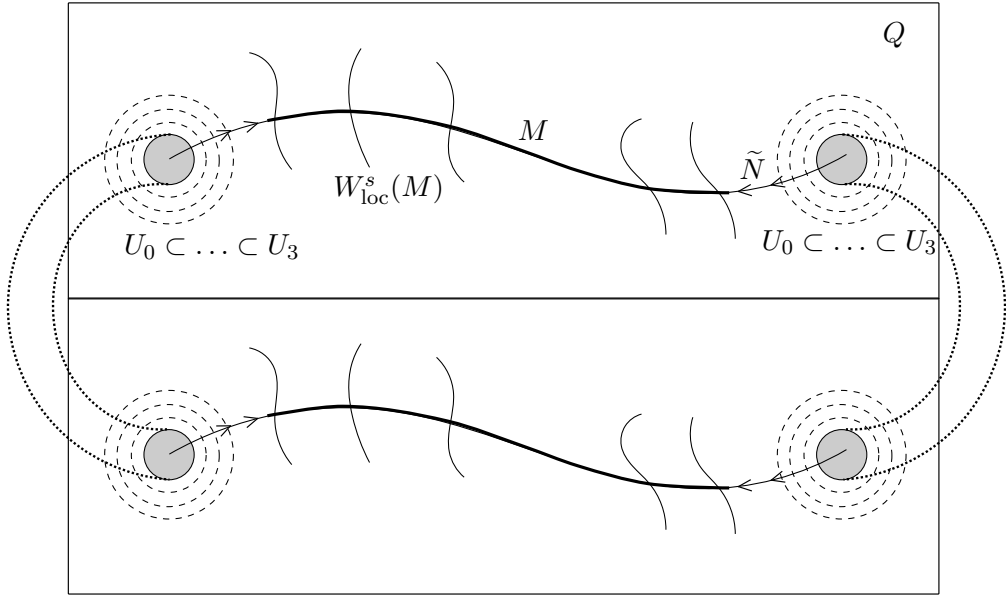
	\caption{A schematic figure of the constructions used in Prop.~\ref{prop:wormhole}.}
  \label{fig:wormhole-2d}
\end{figure}

\begin{Prop}\label{prop:wormhole}
	Let $M,N\subset Q$ be compact inflowing $r$-NAIMs, with $M$ a proper subset of the manifold interior of $N$, for the $C^{r\geq 1}$ flow $\Phi^t$ generated by the $C^{r \geq 1}$ vector field $f$ on $Q$.
	Let $U_0$ be an arbitrarily small tubular neighborhood of $\partial N$, having smooth boundary $\partial U_0$ and disjoint from $\Wsl(M)$.
	Define $\widehat{Q}$ to be the double of $Q\setminus U_0$.
	
	Then there exists a $C^\infty$ differential structure on $\widehat{Q}$ and a $C^r$ vector field $\hat{f}\colon \widehat{Q}\to T\widehat{Q}$ such that
	
	\begin{enumerate}
		\item $\widehat{f}$ is equal to $f$ on each copy of $Q\setminus U_0$, except on an arbitrarily small neighborhood of $\partial U_0$.
		\item There exists a compact and boundaryless $r$-NAIM $\widehat{N}$ for $\hat{f}$, with $\widehat{N}$ equal to $N$ on each copy of $Q\setminus U_0$, except on an arbitrarily small neighborhood of $\partial U_0$.
		\item The global stable foliation of $M$ for $f$ does not intersect $U_0$, and it coincides with the global stable foliation of $M$ for $\hat{f}$, when $M$ and $\Ws(M)$ are identified  via inclusion with subsets of a copy of $Q\setminus U_0$ in $\widehat{Q}$.
		
	\end{enumerate} 
	Let $\widehat{\Phi}^t$ be the $C^r$ flow generated by $\hat{f}$, and let $\widehat{E}^s$ be the $D\widehat{\Phi}^t|_{\widehat{N}}$-invariant stable vector bundle for the NAIM $\widehat{N}$.
    If, additionally, there exist constants $K>0$ and $\alpha < 0$ such that for all $m \in M$, $t \ge 0$ and $0 \le i \le k$ the \concept{$k$-center bunching} condition 
    \begin{equation}\label{eq:app-center-bunching}
    \norm{\D\Phi^{t}|_{T_mM}}^i\norm{\D\Phi^t|_{E_m^s}}\leq Ke^{\alpha t}\minnorm{\D\Phi^t|_{T_mM}}
    \end{equation}
    is satisfied for the original system on $Q$, then \eqref{eq:app-center-bunching} will also be satisfied with $M,E^s$, and $\Phi^t$ replaced by $\widehat{N},\widehat{E}^s$, and $\widehat{\Phi}^t$, and with $\alpha$ replaced by some different constant $\hat{\alpha} < 0$.

	Similarly, if additionally there exist constants $0 < \delta < -\alpha < -\beta$ and $K \geq 1$ such that for all $t \geq 0$
	\begin{equation}\label{eq:app-sakamoto_rates}
	\begin{aligned}
	K^{-1}e^{-\delta t} \leq \minnorm{\D \Phi^t|_{TM}} &\leq \norm{\D \Phi^t|_{TM}} \leq K e^{\delta t}, \\
	K^{-1}e^{-\delta t} \leq \minnorm{(\D \Phi^t|_{TM})^{-1}} &\leq \norm{(\D \Phi^t|_{TM})^{-1}} \leq K e^{\delta t}, \\
	K^{-1}e^{\beta t} \leq \minnorm{\D \Phi^t|_{E^s}} &\leq \norm{\D \Phi^t|_{E^s}} \leq K e^{\alpha t}
	\end{aligned}
	\end{equation}
	uniformly on $TM$ and $E^s$, then we can choose $\hat{f}$ appropriately, such that the same will be true for $\widehat{\Phi}^t$, $T\widehat{N}$, and $\widehat{E}^s$ with modified constants $0<\hat{\delta}< -\hat{\alpha}<-\hat{\beta}$  arbitrarily close to $\delta, \alpha, \beta$. 
\end{Prop}

\begin{Rem}
	It is not an additional hypothesis to require the existence of the manifold $N$ in Proposition~\ref{prop:wormhole}.
	This is because given any compact inflowing NAIM $M$, then for any sufficiently small $\epsilon > 0$, $N\coloneqq \Phi^{-\epsilon}(M)$ will be a compact inflowing NAIM containing $M$. We mention $N$ explicitly only to highlight the fact that \emph{any} compact inflowing NAIM $N$ containing $M$ in its manifold interior will do.
\end{Rem}

\begin{proof}
     Let $\epsilon>0$ be any small positive number, and let $U_0,U_1,U_2,U_3$ be arbitrarily small tubular neighborhoods of $\partial N$, disjoint from $\Wsl(M)$, satisfying $ \partial N \subset U_0 \subset \bar{U}_0 \subset  U_1\subset \bar{U}_1 \subset U_2\subset \bar{U}_2  \subset U_3\subset \Phi^{-\epsilon}(\Wsl(N))$, and such that all $U_i$ have $C^\infty$ boundary $\partial U_i$.
     See Figure \ref{fig:wormhole-2d}.
     Since $N$ is inflowing, we may further construct the $U_i$ so that $f$ is strictly inward pointing at the boundary of each $N\setminus  U_i$, and so that all points in $ U_3$ leave $U_3$ in a uniformly finite time.

     It follows that $\Ws(M)\cap U_3 = \varnothing$.
     We have chosen $U_3$ to be disjoint from $\Wsl(M)$, so to see this, suppose that there exists $m\in M$ and $y \in (\Ws(m)\setminus \Wsl(M))\cap U_3$.
     Since $\Phi^{\epsilon}(U_3) \subset \Wsl(N)$ by construction, by continuity there exists $t_0 > 0$ such that $\Phi^{t_0}(y) \in \Wsl(N)\setminus \Wsl(M)$.
     Let $n \in N \setminus M$ be the unique point such that $\Phi^{t_0}(y) \in \Wsl(n)$.
     Since $y \in \Ws(m)$ and since $\Ws(M)$ is $\Phi^{t}$-invariant, it follows also that $\Phi^{t_0}(y) \in \Ws(\Phi^{t_0}(m))$.
     By uniqueness of the stable fibers, this implies that $n = \Phi^{t_0}(m)$.
     But $n\in N\setminus M$ and $\Phi^{t_0}(m) \in M$ by positive invariance of $M$, so we have obtained a contradiction.
     
     Rescaling $f$ with a smooth cutoff function supported in $Q \setminus U_2$ and identically equal to $1$ on $Q \setminus U_3$, we replace $f$ with a $C^r$ vector field $\tilde{f}$ which is equal to $f$ on $Q\setminus U_3$ and zero on $\bar{U}_2$. By continuity of $D f$, $\bar{U}_2$ consists entirely of nonhyperbolic critical points for $\tilde{f}$.
     
     We next approximate $N$ by a $C^r$ manifold $\widetilde{N}$ such that $\widetilde{N}$ coincides with $N$ on $Q\setminus U_2$ and such that $\widetilde{N}\cap U_1$ is a $C^\infty$ submanifold intersecting $\partial U_0$ transversely.
     This can be achieved by giving $N$ a $C^\infty$ differential structure and then approximating the inclusion $N \hookrightarrow Q$ relative to $N\setminus U_2$ in the $C^r$ topology, see \cite[Ch.~2]{hirsch1976differential} for approximation theory details.
     The fact that $\tilde{f}|_{\bar{U_2}} = 0$ implies that $\tilde{N}$ is invariant under $\tilde{f}$.
    	
    We define $Q'\coloneqq Q \setminus U_0$, a $C^\infty$ manifold with boundary $\partial Q' = \partial U_0$.
    Recall that a $C^\infty$ collar for $\partial Q'$ in $Q'$ is a $C^\infty$ embedding $h\colon \partial Q'\times [0,\infty) \to Q'$ such that $h(x,0)\equiv x$ \cite[p.~113]{hirsch1976differential}.
    We choose a $C^\infty$ collar $h$ for $\partial Q'$ which restricts to a collar of $\partial \widetilde{N}$ in $\widetilde{N}$, i.e., $h|_{\partial \widetilde{N}\times [0,\infty)}\to \widetilde{N}$ is a collar \cite[Thm~6.2]{hirsch1976differential}. 
    Now let $\widehat{Q}$ be the double of $Q'$, the topological space obtained by first forming the disjoint union of two copies of $Q'$, then identifying corresponding points in $\partial Q'$.
	We use the collar $h$ to henceforth endow $\widehat{Q}$, in the usual way, with a $C^\infty$ differential structure (see, e.g., \cite[p.~184]{hirsch1976differential} or \cite[p.~226]{lee2013smooth}), and we let $S$ denote the common image of $\partial Q'$ in $\widehat{Q}$.

	Let $\widehat{N}$ denote the image of $\widetilde{N}$ in $\widehat{Q}$.
	Since $h$ was chosen to restrict to a collar for $\partial \widetilde{N}$ in $\widetilde{N}$, it follows that $\widehat{N}$ is a $C^r$ submanifold of $\widehat{Q}$.
	Letting  $\hat{f}_0$ be the vector field on $\widehat{Q}$ which is equal to $\tilde{f}$ on each copy of $Q'$, it is immediate that $\hat{f}_0\in C^r$ since $\hat{f}_0$ is zero on a neighborhood of $\partial Q'$.
	Finally, using a partition of unity, we give $\widehat{Q}$ a $C^\infty$ Riemannian metric which coincides with the original metric on each copy of $Q'$,  except on an arbitrarily small neighborhood of $S$.

    Next, we modify $\hat{f}_0$ near $S$ to make $\widehat{N}$ normally attracting.
    Let $X$ be a $C^\infty$ manifold which is $C^1$-close to $\widehat{N}$.
    Let $\varphi \colon E' \to \widehat{Q}$ be a $C^\infty$ tubular neighborhood of $X$.
    I.e., $\pi'\colon E'\to X$ is a $C^\infty$ vector bundle and $\varphi$ is an open $C^\infty$ embedding with $\varphi|_{X}$ the inclusion map, identifying $X$ with the zero section of $E'$.
    If $X$ approximates $\widehat{N}$ sufficiently closely, then $\varphi^{-1}(\widehat{N})$ is the image of a $C^r$ section $h\colon X\to E'$.
    Let $V_2\subset \widehat{Q}$ denote the open set which is the image of the two copies of $U_2$ in $\widehat{Q}$, and define $V_3$ similarly.
    Let $\chi\colon E'\to [0,\infty)$ be a $C^\infty$ compactly supported bump function such that $\chi \equiv 1$ on $\bar{V_2}$ and $\supp \chi \subset V_3$.
    We define a $C^{r}$ vector field $\hat{f}$ on $\varphi(E')$ by 
    \begin{equation*}
    \hat{f}\circ \varphi(v_x)\coloneqq \D \varphi_{v_x}\left[ (\varphi^*\hat{f}_0)(v_x) -\rho \chi(v_x)( v_x-h(x))\right], 
    \end{equation*}
    where $\pi'(v_x)=x$,  $\varphi^*\hat{f}_0\coloneqq (\varphi^{-1})_*\hat{f}_0$, and $\rho\coloneqq (\alpha+\beta)/2$ if \eqref{eq:app-sakamoto_rates} holds and $\rho \coloneqq 1$ otherwise.
    Since $\chi$ is compactly supported, it follows that $\hat{f}(v_x)$ is equal to $f(v_x)$ for sufficiently large $\norm{v_x}$, hence we may extend $\hat{f}$ to a $C^{r}$ vector field on $\widehat{Q}$ (still denoted $\hat{f}$) by defining $\hat{f}$ to be equal to $f$ on $\widehat{Q}\setminus \varphi(E)$.
    Let $\widehat{\Phi}_1^t$ denote the flow of $\hat{f}$.
    
    Define a subbundle $E$ of $TQ|_{\widehat{N}}$ by $E \coloneqq \D \varphi (\Ver E'|_{h(X)})$, where $\Ver E'\coloneqq \ker \D \pi'\subset TE'$ is the vertical bundle.
    Since $h$ is a section of $E'$, it follows that $\T E'|_{h(X)} = \T h(X)\oplus\Ver E'|_{h(X)}$, and since $\varphi$ is a local diffeomorphism it follows that $\D\varphi$ preserves this splitting: $\T \widehat{Q}|_{\widehat{N}}= T\widehat{N}\oplus E$.
    Let $\Pi^E\colon T\widehat{Q}|_{\widehat{N}}\to E$ be the projection $T\widehat{N}\oplus E \to E$.
     We now argue that $\widehat{N}$ is an $r$-NAIM for $\hat{f}$; it suffices to show that $\widehat{N}$ is an $r$-NAIM for the linear flow $\Pi^{E}\circ \D\widehat{\Phi}_1^t|_{E}$ \cite[Prop.~1, Thm~6]{fenichel1971persistence}. 
     To do this, by the Uniformity Lemma \cite{fenichel1971persistence} it suffices to show that for each $n \in \widehat{N}$ there exist $C_n > 0$ and $a_n < 0$ such that for any $t \ge 0$ and $0 \le i \le r$, 
     \begin{equation}\label{eq:app_proj_naim_rate}
     \norm{\Pi^E\circ \D\widehat{\Phi}_1^t|_{E_n}} \le C_n e^{a_n t} \minnorm{\D\widehat{\Phi}_1^t|_{\T_n \widehat{N}}}^i.
     \end{equation}
     First note that $\hat{f}$ is equal to $f$ on $\widehat{Q}\setminus V_3$, $\widehat{N}\setminus V_3$ is positively invariant, and $N$ is an $r$-NAIM for $f$.
     It follows that for each $n \in \widehat{N}\setminus V_3$, we can find $a_n$ and $C_n$ such that \eqref{eq:app_proj_naim_rate} holds.
     Next, let $n\in \widehat{N}$ be any point with $\hat{f}_0(n)\neq 0$.
     Since $\widehat{\Phi}^t_1$ takes $n$ into $\widehat{N}\setminus V_3$ in finite time, in this case we can also find $a_n, C_n$ such that \eqref{eq:app_proj_naim_rate} holds.
     Finally, if $n \in \widehat{N}$ is any point with $\hat{f}_0(n)=0$, then $E_n$ is invariant under $\Pi^E\circ \D\widehat{\Phi}_1^t|_{E_n}$.
     The definition of $\hat{f}$ and $E$ imply that $n$ is an exponentially stable fixed point for the restriction of this flow to $E_n$, so we again find $a_n,C_n$ such that \eqref{eq:app_proj_naim_rate} holds.  
     Hence, by the Uniformity Lemma, $\widehat{N}$ is indeed an $r$-NAIM --- in particular, there exists a $\D\widehat{\Phi}_1^t$-invariant stable bundle $\widehat{E}^s_1$ over $\widehat{N}$.
     
     Now suppose additionally that either the $k$-center bunching conditions \eqref{eq:app-center-bunching} or \eqref{eq:app-sakamoto_rates} held for the original system on $Q$.
     Considering now the flow $\D\widehat{\Phi}_1^t|_{\widehat{E}^s_1}$ on $\widehat{E}^s_1$ and
     repeating the argument in the preceding paragraph --- using a different version of the Uniformity Lemma \cite[Lem.~16]{fenichel1974asymptotic} for the case of center bunching conditions --- shows that the corresponding condition still holds for $\hat{f}$ on $\widehat{Q}$.
     
     It remains only to show that the global stable foliation $\widehat{W}^s(M)$ for $\hat{f}$ agrees with $\Ws(M)$ when we identify $M$ and $\Ws(M)$ with either copy of their images in $\widehat{Q}$ --- for definiteness, let us fix one such copy of $M$ and $\Ws(M)$ in what follows (with the former copy a subset of the latter).
     To accomplish this, we first consider the local foliations $\widehat{W}^s_{\text{loc}}(M)$ and $\Wsl(M)$.
     Both local foliations are $\widehat{\Phi}^t$-invariant, the latter because $\hat{f}$ is equal to $f$ on $\Ws(M)$.
     But since $M$ is a compact and inflowing NAIM, the standard Hadamard graph transform \cite{fenichel1974asymptotic} shows that there exists a \emph{unique} local invariant foliation transverse to $M$. 
     More precisely, this means that there exists a sufficiently small neighborhood $J$ of $M$ such that $\forall m\in M: J\cap \Wsl(m) = J\cap \widehat{W}^s_{\text{loc}}(m)$.
     Now for any $m\in M$ and any $y \in \Ws(m)$, there exists $t_0>0$ such that $\Phi^{t_0}(y) \in J\cap \Wsl(\Phi^{t_0}(m)) = J\cap \widehat{W}^s_{\text{loc}}(\Phi^{t_0}(m))$.
     It follows that $y \in \Phi^{-t_0}(\widehat{W}^s_{\text{loc}}(\Phi^{t_0}(m)))\subset \widehat{W}^s(m)$ and therefore that $\Ws(m) \subset \widehat{W}^s(m)$.
     A symmetric argument shows that $\widehat{W}^s(m) \subset \Ws(m)$, and since $m\in M$ was arbitrary it follows that the leaves of $\Ws(M)$ and $\widehat{W}^s(M)$ coincide.
     This completes the proof.
     \end{proof}
     
\section{Fiber bundles}\label{app:fiber_bundles}     

In this appendix, we review the basic notions from the theory of fiber bundles that we use.
Our definition of $C^k$ bundles follows \cite[Def. 1.1.1]{neeb2010differential}.
Other useful references for the topological and $C^\infty$ cases include \cite{steenrod1951topology,husemoller1966fibre,bloch2015nonholonomic}, with a self-contained and brief introduction appearing in \cite[Ch. 2]{bloch2015nonholonomic}. 
We also mention \cite{lee2013smooth,hirsch1976differential,kobayashi1963foundationsV1} as containing nice introductions to vector bundles.

\begin{Def}[Fiber bundles]\label{def:fiber_disk_bundles}
	A $C^k$ fiber bundle, with $1 \leq k \leq \infty$, is a quadruple $(E,B,F,\pi)$ consisting of $C^k$ manifolds $E$, $B$, and $F$ and a $C^k$ map $\pi\colon E\to B$ with the following property of local triviality: each point $b\in B$ has an open neighborhood $U\subset B$ for which there exists a $C^k$ diffeomorphism
	\begin{equation*}
	\varphi_U\colon  \pi^{-1}(U) \to U\times F,
	\end{equation*}
	satisfying
	\begin{equation*}
	\text{pr}_1\circ \varphi_{U} = \pi,
	\end{equation*}
	where $\text{pr}_1\colon U\times F \to U$ is the projection onto the first factor.
	A $C^0$ fiber bundle is defined by replacing all $C^k$ manifolds and diffeomorphisms above with arbitrary topological spaces and homeomorphisms, respectively.
	Often we abuse terminology and simply refer to $E$ or to $\pi\colon E\to B$ as the fiber bundle when the other data is understood.
\end{Def}

\begin{Rem}
The following terminology is common.
$E$ is called the \concept{total space}, $B$ is called the \concept{base space}, $F$ is the \concept{model fiber} or \concept{fiber type}, and $\pi$ is called the \concept{bundle projection}.
Sets of the form $E_b\coloneqq \pi^{-1}(b)$ are called the \concept{fibers} of the bundle or of $\pi$.
The map $\varphi_U$ is called a \concept{local trivialization}.
$(E,B,F,\pi)$ is sometimes called an \concept{$F$-bundle} over $B$.	
\end{Rem}

\begin{Ex}[Disk bundles]
A \concept{$C^k$ disk bundle} is a $C^k$ fiber bundle $(E,B,F,\pi)$ with $F = \R^n$, for some $n \in \N$.	
\end{Ex}

\begin{Def}[Vector bundles]\label{def:vector_bundles}
	A (finite-dimensional) \concept{$C^k$ vector bundle} is a $C^k$ disk bundle $(E,B,F,\pi)$ with the following additional requirement.
	For any open sets $U,V\subset B$ with $U\cap V \neq \varnothing$, the \concept{transition map}
	\begin{equation*}
	\varphi_{U,V}\coloneqq \varphi_{U}^{-1} \varphi_V|_{(U\cap V)\times F}\colon(U\cap V)\times F \to (U\cap V)\times F 
	\end{equation*} 
	is given by
	\begin{equation*}
	\varphi_{U,V}(b,v) = (b, A(b)v),
	\end{equation*}
	where $A\colon B \to \text{GL}(n,\R)$ is a $C^k$ invertible matrix-valued map.
\end{Def}

\begin{Ex}[The tangent bundle]
	Let $Q$ be a smooth ($C^\infty$) $n$-manifold.
	Then its tangent bundle $\pi\colon TQ\to Q$ is a smooth vector bundle.
	To see this, let $(U,\psi_U)$ be a smooth chart for $Q$.
	Identifying $\T \R^n \cong \R^n \times \R^n$, then $\varphi_U\coloneqq (\pi,\text{pr}_2 \circ \D \psi_U)\colon TQ|_U \to U \times \R^n$ satisfies
    $\text{pr}_1\circ \varphi_{U} = \pi$, where $TQ|_U\coloneqq \pi^{-1}(U)$ and $\text{pr}_{i}$ is projection onto the $i$-th factor.
	If $(V,\psi_V)$ is another chart, then 
	\begin{equation*}
	\varphi_{U,V}\left(b,\sum_{i}v^k e_k\right) = \left(b,\sum_{i} \frac{\partial \psi^i_{U,V}}{\partial x^j}v^j e_i\right),
	\end{equation*}
where $(e_k)$ is the standard basis of $\R^n$ and $\psi_{U,V}\coloneqq \psi_U^{-1} \circ \psi_V|_{U\cap V}$.
Hence we may take the Jacobian of $\psi_{U,V}$ to play the role of $A$ from Definition \ref{def:vector_bundles}, so $TQ$ is indeed a smooth vector bundle.
\end{Ex}

\begin{Def}\label{def:fiber_bundle_isomorphisms}
	A $C^k$ isomorphism $\psi\colon E \to E'$ of $C^k$ fiber bundles $(E,B,F,\pi)$ and $(E',B',F',\pi')$ covering a map $\rho\colon B\to B'$ is a $C^k$ fiber-preserving diffeomorphism $\psi$ (homeomorphism if $k = 0$): each fiber $\pi^{-1}(b)$ is bijectively mapped to the fiber $\pi'^{-1}(\rho(b))$.
	This is equivalent to requiring that $\psi$ is a $C^k$ diffeomorphism and
	\begin{equation*}
	\pi'\circ \psi = \pi \circ \rho.
	\end{equation*}
A $C^k$ fiber bundle $(E,B,F,\pi)$ is \concept{trivial} as a $C^k$ bundle if it is $C^k$ isomorphic to the fiber bundle $(B\times F, B, F, \text{pr}_2)$, where $\text{pr}_2$ is projection onto the second factor.
\end{Def}

\begin{Ex}\label{ex:how_to_show_its_a_bundle}
	Let $\pi\colon E\to B$ be any $C^k$ map, and assume there exists a fiber bundle $(E',B,F,\pi')$ and a $C^k$ diffeomorphism $\psi\colon E\to E'$ with the property that
	\begin{equation*}
	\pi'\circ \psi = \id_{B}.
	\end{equation*}
	Then $(E,B,F,\pi)$ is also a $C^k$ fiber bundle.
	This is because if $\varphi_U\colon \pi'^{-1}(U)\to U\times F$ is a $C^k$ local trivialization of $\pi'\colon E'\to B$, then $\varphi_U\circ \psi\colon \pi^{-1}(U) \to U\times F$ is a $C^k$ local trivialization of $\pi\colon E \to B$.
	In particular, if $(E',B,F,\pi')$ is a disk bundle, then so is $(E,B,F,\pi)$. (recall that all vector bundles are disk bundles --- we use this in the proof of Theorem \ref{th:fiber_bundle_theorem}). 
\end{Ex}

\begin{Def}\label{def:section_of_bundle}
A map $X\colon B \to E$ is a $C^k$ \concept{section} of a fiber bundle $(E,B,F,\pi)$ if $X \in C^k$ and $\pi\circ X = \id_B$.	
\end{Def}

\begin{Ex}
	A $C^k$ section $X$ of the tangent bundle $\T Q$ is the same thing as a $C^k$ vector field $X$ on $Q$. 
	The requirement $\pi \circ X = \id_Q$ simply means that $X(q)$ is a tangent vector based at $q\in Q$.
\end{Ex}
We next give definitions of vector subbundles and the Whitney sum of vector bundles. 
These concepts are fundamental to the very definition of normal hyperbolicity, see \S \ref{sec:preliminary_constructions}.

\begin{Def}[Vector subbundle]
	A vector bundle $(E',B,F,\pi)$ is a $C^k$ subbundle of a $C^k$ vector bundle $(E,B,F,\pi)$ if every $b\in B$ has a neighborhood $U$ such that there exist pointwise linearly independent $C^k$ sections $(X_1,\ldots,X_d)$ (called a \concept{local frame}) which span $E'_{c}$ for all $c \in U$. 
\end{Def}

\begin{Def}[Whitney sum]\label{def:Whitney_sum}
	The Whitney sum of two $C^k$ vector bundles $(E,B,\R^n,\pi)$ and $(E',B,\R^m,\pi')$ is the $C^k$ vector bundle $(E \oplus E',B,\mathbb{R}^{n+m},\tilde{\pi})$ whose fiber $(E\oplus E')_{b}$ is given by $E_b \oplus E'_{b}$.
	The $C^k$ vector bundle structure is determined as follows. 
	If $\varphi_U$ and $\varphi_U'$ are local trivializations for $E$ and $E'$ over $U$, then $(\varphi_U, \text{pr}_2 \circ \varphi_U')$ is a local trivialization for $E\oplus E'$ over $U$.
\end{Def}

Following \cite[p. 116]{kobayashi1963foundationsV1}, we now give the definition of a \concept{fiber metric} on a vector bundle, which generalizes the notion of a Riemannian metric on the tangent bundle of a manifold.

\begin{Def}[Fiber metric]\label{def:fiber_metric}
	A $C^k$ fiber metric on a $C^k$ vector bundle $\pi\colon E \to M$ is an assignment, to each $m \in M$, of an inner product $g_m$ on the fiber $\pi^{-1}(m)$, such that for any $C^k$ sections $X,Y\colon M\to E$, the map $m\mapsto g_m(X(m),Y(m))$ is $C^k$.
\end{Def}
Using a partition of unity, it easy to show that fiber metrics always exist on any vector bundle over a paracompact base.
A fiber metric defines a norm on each fiber $\pi^{-1}(m)$ via $\|X\|_m\coloneqq \sqrt{g_m(X,X)}$ for $X \in E_m$.
We will often suppress the subscript $m$ and simply write $\|X\|$.

\bibliographystyle{amsalpha}
\bibliography{references}

\end{document}